\documentclass{article} 
\usepackage{amsmath,amsthm}
\usepackage{amssymb}
\usepackage{epsfig,tikz}
\usepackage{graphicx}
\usepackage{subcaption} 
\usepackage[margin=0.8in]{geometry}
\usepackage{setspace}
\usepackage{enumerate}
\usepackage{tikz}
\usepackage{tikz-cd}
\usepackage{xcolor}
\usepackage{mathtools}
\usepackage{mathabx} 
\usepackage{stmaryrd}
\usepackage{thmtools}
\usepackage{etoolbox} 

\usepackage[backref=page]{hyperref}
\hypersetup{
  colorlinks   = true,          
  urlcolor     = blue,          
  linkcolor    = purple,          
  citecolor   = violet             
}

\theoremstyle{definition}
\newtheorem{definition}{Definition}[section]
\newtheorem{lemma}[definition]{Lemma}
\newtheorem{proposition}[definition]{Proposition}
\newtheorem{theorem}[definition]{Theorem}

\newtheorem{corollary}[definition]{Corollary}
\newtheorem{example}[definition]{Example}

\newtheorem{notation}[definition]{Notation}

\newtheorem{remark}[definition]{Remark}
\newtheorem{assumption}{Assumption}[section]

\newtheorem{question}[definition]{Question}

 \AtEndEnvironment{definition}{\defend}
\providecommand{\defend}{}
\renewcommand{\thmcontinues}[1]{continued}

\let\tilde\widetilde

\newcommand{\hm}{\mathrm{Hom}}

\newcommand{\cc}{\mathbb{C}}
\newcommand{\qq}{\mathbb{Q}}

\newcommand{\zz}{\mathbb{Z}}
\newcommand{\pp}{\mathbb{P}}

\newcommand{\id}{\mathbf{1}}
\newcommand{\s}{\subset}
\newcommand{\ga}{\alpha}
\newcommand{\gb}{\beta}

\newcommand{\cyc}[1]{\ensuremath{\langle #1\rangle}}
\newcommand{\liealg}{\mathfrak{g}}
\newcommand{\heisen}{\mathfrak{h}}
\newcommand{\slalg}{\mathfrak{sl}}

\title{Positivity of coinvariant divisors on $\overline{\mathrm{M}}_{0,n}$ and the parafermions}
\author{Avik Chakravarty}
\date{}
\begin{document}
\maketitle
\begin{abstract}
We give criteria for determining the positivity of line bundles coming from vertex operator algebras (VOAs) on the moduli space $\overline{\mathrm{M}}_{0,n}$ of rational curves with $n$ marked points.
The criteria use the multiplicative structure of VOA representations encoded in the fusion ring. Using them, we construct positive line bundles on $\overline{\mathrm{M}}_{0,n}$ from certain parafermion VOAs. These give the first examples of commutant VOAs producing positive line bundles.
\end{abstract}
\section{Introduction}
The fusion ring of a vertex operator algebra  $V$ (a VOA) encodes the multiplication rules for its representations. These representations play an important role in the geometry arising from VOAs. For instance, under certain assumptions, one can associate a vector bundle on the moduli space $\overline{\mathrm{M}}_{g,n}$ of $n$-pointed genus-$g$ curves to any $n$-tuple of admissible $V$-modules \cite{TK87, TUY89, BFM, NT05, DGT1, DGK2}. These are called \emph{coinvariant vector bundles}, and their first Chern classes are the associated \emph{coinvariant divisors}. 

Our main result, Theorem~\ref{thm: intro1}, concerns the use of the fusion ring to prove positivity properties of coinvariant divisors on $\overline{\mathrm{M}}_{0,n}$. Specifically, we provide criteria for nefness and semi-ampleness of such divisors. A divisor on a projective variety $X$ is \emph{nef} if it has non-negative intersection with every curve on $X$, and it is \emph{semi-ample} if some positive multiple is base-point free. Examples include pullbacks of ample divisors along morphisms, and their study plays a fundamental role in understanding the birational geometry of $X$. We introduce the notion of \emph{positivity} for subrings of the fusion ring (see Definition~\ref{def: positive subring}) and establish a correspondence between the positivity of coinvariant divisors and that of these subrings. The vertex operator algebras we consider must satisfy Assumption~\ref{assump: VB}, which includes all rational, $C_2$-cofinite vertex operator algebras of CFT type.

\begin{theorem}[Theorems~\ref{thm: intro1.i},~\ref{thm: intro1.ii} and~\ref{thm: intro1.iii}]
\label{thm: intro1}
Let $\mathcal{S}$ be a subring of the fusion ring of a VOA $V$.
\begin{enumerate}[a.]
\item For $\mathcal{S}$ positive and $n$  simple $V$-modules $W^i$ in $\mathcal{S}$, the coinvariant divisor $\mathbb{D}_{0,n}(V,\bigotimes_{i=1}^nW^i)$ is nef on $\overline{\mathrm{M}}_{0,n}$. 
\item If all simple \( V \)-modules in \( \mathcal{S} \) have non-negative conformal weight, and \( W \) is a simple module in \( \mathcal{S} \) with maximal conformal weight, then the symmetric coinvariant divisor \( \mathbb{D}_{0,n}(V, W^{\otimes n}) \) is nef if it is an $F$-divisor.
\item Let $\overline{V}$ be a VOA with a fusion subring $\overline{\mathcal{S}}$ such that $(\mathcal{S},\overline{\mathcal{S}})$ are proportional (cf. Definition~\ref{def: proportional cw pairing}). For $W^1,\ldots,W^n$ simple $V$-modules in $\mathcal{\mathcal{S}}$, there exists $n$ simple $\overline{V}$-modules $\overline{W}^i$ satisfying $\mathbb{D}_{0,n}(V,\bigotimes_{i=1}^n W^i)=\eta\ \mathbb{D}_{0,n}(V,\bigotimes_{i=1}^n\overline{W}^i)$ for some $\eta\in\qq_{>0}$. If $\mathbb{D}_{0,n}(V,\bigotimes_{i=1}^n \overline{W}^i)$ is nef (resp. base-point free), $\mathbb{D}_{0,n}(V,\bigotimes_{i=1}^n W^i)$ is nef (resp. semi-ample).
\end{enumerate}
\end{theorem}

As an application of Theorem~\ref{thm: intro1}, we establish several families of nef line bundles on \( \overline{\mathrm{M}}_{0,n} \) arising from the representation theory of the level \( k \) parafermion vertex operator algebra \( K(\mathfrak{sl}_{r+1}, k) \), associated with the simple Lie algebra \( \mathfrak{sl}_{r+1} \). These algebras are rational and \( C_2 \)-cofinite VOAs of CFT-type \cite{ALY14, DR17}, first introduced in \cite{LW2, LP1}, and form the mathematical foundation of \emph{parafermionic conformal field theory} \cite{DL1, BOOK-VA}, originally developed in \cite{ZF1, GEP87}. Their representation theory has been central in the study of Rogers–Ramanujan-type identities \cite{LW2, LW3}, and the algebra \( K(\mathfrak{sl}_2, k) \) coincides with the \( (k+1,k+2) \)-minimal series \( W \)-algebras associated with \( \mathfrak{sl}_k \) \cite{ALY19}. The structure theory was developed in \cite{DLY09, DLWY10, DW10, DW11, ALY14, DW16, DR17, ADJR18}. These VOAs are defined as the commutant of the Heisenberg subalgebra \( M_{\hat{\mathfrak{h}}}(k) \) inside the affine vertex algebra \( L_{\hat{\mathfrak{sl}}_{r+1}}(k,0) \), where \( \mathfrak{h} \subset \mathfrak{sl}_{r+1} \) is the Cartan subalgebra; see Section~\ref{sec: parafermion construction} for details. They are generated in degrees \( 2 \) and \( 3 \), so the results of \cite{DG23} do not apply. Parafermions constitute the first example of a vertex operator algebra for which positivity holds only on certain proper subrings of the fusion ring, rather than on the entire ring (see Remark~\ref{rem: para vs others}). Identifying and characterizing such positive subrings is a subtle and central aspect of this work. Our positivity results for coinvariant divisors associated with these parafermion VOAs on \( \overline{\mathrm{M}}_{0,n} \) are summarized below. The relevant modules are described in Section~\ref{subsec: parafermions and modules sl}.

\begin{theorem}[Theorems~\ref{thm: positive for sl2},~\ref{thm: positive for sl3} and Proposition~\ref{prop: F-positive for sl3}] 
\label{thm: intro2}
\begin{enumerate}[i.]
\item For any integer $k\geq1$, all coinvariant divisors on $\overline{\mathrm{M}}_{0,n}$ associated to an $n$-tuple of simple $K(\slalg_2,k)$-modules of the form $M^{2a,a}$ are semi-ample, and hence nef.
\item For any integer $k\geq1$, all coinvariant divisors on $\overline{\mathrm{M}}_{0,n}$ associated to an $n$-tuple of simple $K(\slalg_2,k)$-modules of the form $M^{k,b}$, where $b\in[0,k-1]\cap\zz$, are semi-ample, and hence nef.
\item For any integer $1\leq k\leq 10$, all coinvariant divisors on $\overline{\mathrm{M}}_{0,n}$ associated to an $n$-tuple of simple $K(\slalg_3,k)$-modules of the form $M^{0,0-(a,b)}$, where $a,b\in[0,k-1]\cap\zz$, intersect all $F$-curves non-negatively.
Furthermore, all such coinvariant divisors are nef if $k\leq 5$. 
\item Let $1\leq k\leq 10$ be an integer and let $\mathbb{D}$ be the symmetric coinvariant divisor on $\overline{\mathrm{M}}_{0,n}$ associated to the $K(\slalg_3,k)$-module $M^{0,0-(\lfloor 2k/3\rfloor,\lfloor k/3\rfloor)}$. Then, $\mathbb{D}$ is nef.
\end{enumerate}
\end{theorem}
The proof of Theorem~\ref{thm: intro2}(iii) is an application of Theorem~\ref{thm: intro1}(a), and is somewhat technical owing to the combinatorial complexity of the definition of the conformal weight for the modules (cf. Proposition~\ref{prop: slr cw}). This is the primary reason for the finiteness restrictions on the level $k$ above. While the statements of the results are similar, a different approach is taken for proving Theorems~\ref{thm: intro2}(i-ii). In particular, the computational complexity is avoided by establishing a proportional pairing of the respective subrings with certain subrings of the fusion rings of particular affine vertex operator algebras. This realizes the coinvariant divisors in (i) and (ii) as multiples of certain coinvariant divisors associated to affine vertex operator algebras (cf. Remark~\ref{rem: affine vs parafermion}). Since affine VOAs are generated in degree~1, the corresponding coinvariant divisors are base-point free~\cite{Fak12,DG23}. Parafermion divisors are positive rational multiples of these base-point free divisors, and therefore they are semi-ample.
This illustrates how Theorem~\ref{thm: intro1}(iii) provides a framework for understanding the relations between coinvariant divisors associated with representations of different VOAs. Finally, Theorem~\ref{thm: intro2}(iv) is a direct application of Theorem~\ref{thm: intro1}(b). 

We define a subring $\mathcal{S}_{r}(k)$ of the fusion ring of the parafermion VOA $K(\slalg_{r+1},k)$ for all integers $r, k \geq 1$, and study its positivity properties in Section~\ref{subsec: parafermions and sl3}. Theorem~\ref{thm: intro2}(ii) shows that the coinvariant divisors associated to representations in $\mathcal{S}_1(k)$ are semi-ample for all $k \geq 1$. Likewise, parts~(ii) and~(iii) of Theorem~\ref{thm: intro2} establish that the subring $\mathcal{S}_2(k)$ is $F$-positive for $k \leq 10$ and positive for $k \leq 5$. This positivity property fails in full generality:

\begin{proposition}[Proposition~\ref{prop: non F-positive}]
Let $k\geq4$ be any integer. The coinvariant divisors associated to  representations in $\mathcal{S}_r(k)$ are nef if and only if $r \leq 2$.
\end{proposition}

We prove this proposition by constructing a proper subring $\mathcal{S}'_{r}(k) \subsetneq \mathcal{S}_r(k)$ that is not $F$-positive. In Propositions~\ref{prop: Sr k2} and~\ref{prop: Sr k3}, we show that all non-trivial symmetric coinvariant divisors associated to representations in $\mathcal{S}_r(2)$ and $\mathcal{S}_r(3)$, for any $r \geq 3$, are nef but not ample. A parallel study is carried out in Section~\ref{sec: symm sl3 para} for symmetric coinvariant divisors arising from subrings $\mathcal{S}_2(k)$ for $k \leq 5$. We also construct several families of nef but non-ample divisors arising from $K(\slalg_2,k)$-representations; see, for example, Propositions~\ref{prop: extremal Mk,a} and~\ref{prop: extremal M2a,a symmetric}. 

As commutants, the parafermions are subalgebras of the affine vertex operator algebras. The coinvariant divisors arising from affine VOAs define morphisms from $\overline{\mathrm{M}}_{0,n}$ to projective varieties. For instance, some are pullbacks of ample line bundles along morphisms to projective varieties with a GIT construction and modular interpretation \cite{Fak12, Noah, NoahAngela, Veronese}. Therefore, it is natural to ask if the coinvariant divisors arising from the parafermions exhibit interesting positivity properties, owing to their close relation to affine VOAs. Theorem~\ref{thm: intro2} is our first attempt to understand the geometry associated with parafermions. We also provide a criterion for the non-triviality of nef divisors arising from $K(\slalg_2,k)$-representations. 
\begin{proposition}[Proposition~\ref{prop: non-triviality M2a,a}]
A coinvariant divisor on $\overline{\mathrm{M}}_{0,n}$ associated to any collection of $n$ many $K(\slalg_2,k)$-representations $M^{2a_i,a_i}$ is non-trivial if and only if $\sum_{i=1}^n a_i>k$. \end{proposition}
The necessary condition for the proposition follows from a similar condition for the $\slalg_2$ affine VOA of level $k$  in \cite[Lemma 4.1]{Fak12}, which was extended for general $\slalg_{r+1}$ affine VOAs in \cite[Proposition 1.3]{CriticallevelBGM}. The sufficiency statement is our contribution in providing a criterion for non-triviality of coinvariant divisors arising from $L_{\hat{\slalg}_2}(k,0)$-representations $L_{\hat{\slalg}_2}(k,2a\Lambda)$, where $\Lambda$ is the fundamental weight, using the relation we establish between coinvariant divisors associated to $L_{\hat{\slalg}_2}(k,0)$-representations and $K(\slalg_2,k)$-representations (cf. Remark~\ref{rem: affine vs parafermion}). 

\begin{proposition}[Proposition~\ref{prop: non-triviality Mk,a}]
A coinvariant divisor on $\overline{\mathrm{M}}_{0,n}$ associated to any collection of $n$ many $K(\slalg_2,k)$-representations $M^{k,a_i}$ is non-trivial if and only if there is a partition $I\cup J\cup K\cup L=\{1,\ldots,n\}$ so that 
\[
\overline{\sum_{i\in I}a_i}+\overline{\sum_{i\in J}a_i}+\overline{\sum_{i\in K}a_i}+\overline{\sum_{i\in L}a_i}=2k,
\]
where $\overline{b}$ denotes the residue of $b$ modulo $k$. 
\end{proposition}
Finally, for higher genera $g\geq1$, we establish a set of nef divisors on $\overline{\mathrm{M}}_{g,n}$, for any $n\geq0$, obtained by adding a rational multiple of the lambda class $\lambda$ to the coinvariant divisors described in Theorem~\ref{thm: intro2}(i-iii). 

\begin{corollary}[Corollaries~\ref{cor: lambda twisted nef classes for sl2} and~\ref{cor: lambda twisted nef classes for sl3}]
Let $k\geq1$ be any integer. There is a $q\in\qq_{\geq0}$ so that the divisor $(q'\lambda+\mathbb{D})$ is nef for all $q'\in\mathbb{Q}_{\geq q}$ and any coinvariant divisor $\mathbb{D}$ on $\overline{\mathrm{M}}_{g,n}$ given by representations $M^{2a_1,a_1},\ldots,M^{2a_n,a_n}$ of the VOA $K(\slalg_2,k)$. An analogous statement holds when each $M^{2a_i,a_i}$ is replaced by $M^{k,a_i}$, for any $k \geq 1$. Furthermore, the same applies for $K(\slalg_3,k)$, with $M^{2a_i,a_i}$ replaced by $M^{0,0-(a_i,b_i)}$, for $1 \leq k \leq 5$.
\end{corollary}

\subsection{Sketch of the proof for Theorem~\ref{thm: intro1}}
We provide a brief outline of the proof of the results in the theorem, paraphrasing technical details. 
(a) A positive subring $\mathcal{S}$ has two properties: it is $F$-positive, and every coinvariant divisor $\mathbb{D}$ on $\overline{\mathrm{M}}_{0,n}$ associated to $V$-modules in $\mathcal{S}$ can be written as $\mathbb{D} = cK_{\overline{\mathrm{M}}_{0,n}} + E$ for some $c \in \mathbb{Q}_{\geq 0}$ and an effective sum $E$ of boundary divisors. The first property guarantees that \( \mathbb{D} \) is an \( F \)-divisor, while the second enables a reduction of the nefness of \( \mathbb{D} \) on \( \overline{\mathrm{M}}_{0,n} \) to verifying the nefness of a coinvariant divisor on \( \overline{\mathrm{M}}_{0,t} \) for some \( 4 \leq t \leq 7 \). The conclusion then follows from Theorem~\ref{thm: KM Fcong result} and Proposition~\ref{lem: pullback of divisors}.
(b) Given a symmetric coinvariant $F$-divisor $\mathbb{D}$ on $\overline{\mathrm{M}}_{0,n}$ associated to a representation $W \in \mathcal{S}$ of maximal conformal weight, we construct an $F$-divisor $\mathbb{D}'$ on $\overline{\mathrm{M}}_{n}$ such that $f^* \mathbb{D}' = \mathbb{D}$ via the flag map $f \colon \overline{\mathrm{M}}_{0,n} \to \overline{\mathrm{M}}_n$ obtained by attaching an elliptic curve at each marked point of the rational curve. The maximality of the conformal weight of $W$ and the fact that $\mathbb{D}$ is an $F$-divisor together ensure that $\mathbb{D}'$ is nef on $\overline{\mathrm{M}}_n$, and hence $\mathbb{D}$ is nef as the pullback of a nef divisor.
(c) If $(\mathcal{S}, \mathcal{S}')$ form a proportional subring pair, then there exists an injective map $f \colon \mathcal{S}^{\mathrm{sim}} \to (\mathcal{S}')^{\mathrm{sim}}$ between their simple objects, inducing an injection of rings $f \colon \mathcal{S} \to \mathcal{S}'$. Moreover, the ratio of conformal weights $\mathrm{cw}(f(W)) / \mathrm{cw}(W)$ is a constant positive rational number for all simple modules $W \in \mathcal{S}^{\mathrm{sim}}\setminus{V}$. The ring injection implies that the fusion rules of $\mathcal{S}$ and its image in $\mathcal{S}'$ coincide. Since conformal weight ratios are constant, any coinvariant divisor associated to simple modules in $\mathcal{S}$ is a rational multiple of a coinvariant divisor associated to simple modules in $\mathcal{S}'$. The conclusion follows from Equation~\ref{eqn: first chern class}.

\subsection{Outline of the paper}  
Section~\ref{sec: prelimnaries} establishes the background and notation needed for this paper and the results we use. Section~\ref{subsec: prelim and VOAs} gives a short introduction to the fusion ring of a vertex operator algebra and the associated coinvariant divisors. In particular, several results about coinvariant divisors that are important for our paper are reviewed in Section~\ref{subsubsec: results for divisors}. Section~\ref{sec: general VOA results} is devoted to the proof of Theorem~\ref{thm: intro1}. Important notions such as \emph{positive subrings} and \emph{proportional pairings} are also defined in this section. Section~\ref{sec: parafermions} details the positivity results for parafermion vertex operator algebras. In Section~\ref{subsec: parafermions and modules sl}, we derive information about parafermions $K(\slalg_{r+1},k)$ needed for the rest of Section~\ref{sec: parafermions}. In Section~\ref{subsec: parafermions and sl2}, we prove Theorem~\ref{thm: intro2}(i-ii) and study the extremality of the corresponding nef divisors. Section~\ref{subsec: parafermions and sl3} is devoted to the proof of Theorem~\ref{thm: intro2}(iii-iv) and studying positivity of the subring $\mathcal{S}_r(k)$. 
The last subsection provides a proof of Proposition~\ref{prop: degree for sl2}, the degree formula for a coinvariant divisor on $\overline{\mathrm{M}}_{0,4}$ with $K(\slalg_2,k)$-modules. 

\subsection*{Acknowledgments}  
The author expresses deep gratitude to Angela Gibney for her invaluable guidance, constant support, and insightful discussions that shaped this work. We thank Han-Bom Moon for valuable questions and feedback on an earlier draft of the paper. Heartfelt thanks to Jianqi Liu and Qing Wang for generously sharing their expertise on parafermions. We are indebted to Daebeom Choi for his thoughtful feedback and collaborative discussions, which significantly enhanced the quality of this paper and special appreciation goes to Niuniu Zhang for his expertise in Python and invaluable assistance with the initial setup.

\section{Preliminaries}
\label{sec: prelimnaries}
\subsection{Fusion ring and the coinvariant divisors}
\label{subsec: prelim and VOAs}
In this section, we briefly discuss the multiplicative structure of the representations of a vertex operator algebra (or, shortly, VOA), and the coinvariant divisors they define on $\overline{\mathrm{M}}_{g,n}$. We refer readers to \cite{FHL93} and \cite[Section 1]{DGT1} for an introduction to the theory of vertex operator algebras and their representations. Other useful references for detailed discussions of vertex operator algebras can be found in \cite{FBZ04, BOOK-VA, LL04}.
\subsubsection{Fusion ring}
\label{sec: fusion ring}
This subsection is mostly based on \cite{FHL93}. For any vector space $V$, denote by 
\(V\{z\}=\{\sum_{n\in\qq}v_nz^n\mid v_n\in V\}\)
the space of $V$-valued formal series involving the rational powers of $z$. 
\begin{definition}
Let $V$ be a vertex operator algebra and let $(W^i,Y_i),$ $(W^j,Y_j)$ and $(W^k,Y_k)$ be simple $V$-modules. An \emph{intertwining operator} of type $\left(\begin{smallmatrix}
  & i & \\
j &   & k
\end{smallmatrix}\right)$ or type 
$\left(\begin{smallmatrix}
&W^i&\\ W^j&&W^k
\end{smallmatrix}\right)$
 is a linear map 
\(
W^j\otimes W^k\to W^i\{z\}
\)
which is equivalent to the map
\begin{align*}
W^j&\to(\hm(W^k,W^i))\{z\}\\
\omega&\mapsto\mathcal{Y}(\omega,z)=\sum_{n\in\qq}\omega_nz^{-n-1}\quad\text{(where $\omega_n\in\hm(W^k,W^i)$)}
\end{align*}
such that all defining properties of a module action that make sense hold. 
\end{definition}
\begin{definition}
The intertwining operators of type $\left(\begin{smallmatrix}
&W^i&\\ W^j&&W^k
\end{smallmatrix}\right)$ form a vector space, which we denote by $V^{W^i}_{W^jW^k}$ or $V^i_{jk}$ and we set 
\[
N^i_{jk}:=N^{W^i}_{W^jW^k}:=\dim V^i_{jk}\quad(\leq\infty).
\]
These numbers $N^i_{jk}$ are called the \emph{fusion rules} associated with modules of the vertex operator algebra $V$.
\end{definition}
\begin{remark}
\label{rem: fusion rule and rank duality}
Yet another way to describe the space $V^{W^i}_{W^j,W^k}$ is the following
\[
V^{W^i}_{W^j,W^k}=\hm(W^i\otimes W^k,W^j)\cong\hm((W^j)'\otimes W^i\otimes W^k,\cc),
\]
where $(W^j)'$ is the dual of the module $W^j$. See \cite[Proposition 2.1(iii)]{TK86} for details.
\end{remark}

The following properties of the fusion rules will be useful for this paper. 
\begin{proposition}[\cite{FHL93,AA13}]
\begin{enumerate}[(a)]
\item Given $V$ and the $V$-module $W$ is nonzero, $N^V_{VV},N^W_{VW},N^W_{WV}\geq1$. 
\item For $V$-modules $W^i,W^j,W^k$, we have $N^{i}_{jk}=N^{i}_{k,j}=N^{k'}_{j,i'}$ and $N^{i'}_{jk}=N^{\sigma(i)'}_{\sigma(j)\sigma(k)}$ for all permutations $\sigma\in Sym(3)$. 
\item Under some additional assumptions\footnote{Precisely, the identities (5.5.9) and (5.5.10) in \cite{FHL93}. (They are not satisfied by affine VOAs, for example.)}, we have \(N^{i}_{jk}=N^{\sigma(i)}_{\sigma(j),\sigma(k)},
\) for all permutations $\sigma\in Sym(3)$. 
\item $N^{W^1}_{V,W^2}=0$ if $W^1\not\cong W^2$ and $=1$ if $W^1\cong W^2$. 
\end{enumerate}
\end{proposition}
\begin{proof}
Details for (a-c) can be found in \cite[Section 5.4 and 5.5]{FHL93} and (d) is \cite[Proposition 2.7]{AA13}.
\end{proof}
\begin{definition}
Let $V$ be a vertex operator algebra and $W^1$ and $W^2$ be two $V$-modules. A module $(W,\mathcal{Y})$, where $\mathcal{Y}\in V^{W}_{W^1,W^2}$, is called a \emph{fusion product} of $W^1$ and $W^2$ if for any $V$-module $M$ and $\mathcal{Y}^M\in V^M_{W^1,W_2}$, there is a unique $V$-module homomorphism $f:W\to M$, such that $\mathcal{Y}^M=f\circ\mathcal{Y}$. We denote $(W,\mathcal{Y})$ by $W^1\boxtimes_VW^2$.
\end{definition}
If $V$ is rational, then fusion product between any two irreducible $V$-modules $W^1$ and $W^2$ exists and it can be explicitly written as 
\begin{equation*}
W^1\boxtimes_VW^2=\sum_{W\in\mathcal{W}}N^W_{W^1,W^2}W,    
\end{equation*}
Finally, we are ready to define the fusion ring of a vertex operator algebra.
\begin{definition}
\label{def: fusion ring}
The \emph{fusion ring} $\mathcal{R}$ of a vertex operator algebra $V$ is freely generated by simple $V$-modules as a group with addition defined formally and the fusion rules defining the multiplicative structure. 
\end{definition}
Note that for a rational vertex operator algebra $V$, we have $V$ is an element of $\mathcal{R}$ since it is a $V$-module and thus can be written as a finite sum of simple $V$-modules. 
\begin{remark}
\label{rem: fusion ring operations}
We sometimes write $M_1\oplus M_2$ as $M_1+M_2$ and denote the contragradient dual $M^{\vee}$ of a simple $V$-module $M$ by $-M$. With this notation, we have the following set-theoretic description of the fusion ring
\[
\mathcal{R} \overset{set}{=} \{\sum_{ M^i\in\mathcal{W}} z_iM^i\mid z_i\in\zz\text{ is nonzero for finitely many $i$}\},
\]
where the sum is taken over the set of simple $V$-modules $\mathcal{W}$.
\end{remark}
\begin{notation}
Let $V$ be a vertex operator algebra, $\mathcal{R}$ be its fusion ring, and $\mathcal{W}$ the set of simple $V$-modules. For any subring $\mathcal{S}\s\mathcal{R}$, we denote its subset $\mathcal{S}\cap\mathcal{W}$ of simple $V$-modules by $\mathcal{S}^{sim}$. In particular, $\mathcal{R}^{sim}=\mathcal{W}$.
\end{notation}
\subsubsection{Coinvariant vector bundles and coinvariant divisors}
\label{subsubsec: results for divisors}
Let $V$ be a self-contragradient vertex operator algebra (VOA) of CFT-type, that is, the contragradient dual $V'$ of $V$ is itself and $\dim_{\cc} V_0=1$. Given an $n$-tuple of admissible $V$-modules $M^1, \ldots, M^n$, one can associate a \emph{sheaf of coinvariants} $\mathcal{F}(V, M^1 \otimes \cdots \otimes M^n)$ on the moduli space $\overline{\mathrm{M}}_{g,n}$ of stable pointed curves. The dual of this sheaf is referred to as the corresponding \emph{sheaf of conformal blocks}. The construction of sheaves of coinvariants originated in the setting of affine VOAs in \cite{TK87, TK86, TUY89, Tsu93}, and has since been extended to more general classes of VOAs in \cite{BB93, FBZ04, DGT1, DGT2, DGT3, DGK, DGK2}. If $V$ satisfies certain natural finiteness and semi-simplicity conditions, these sheaves are known to form vector bundles over $\overline{\mathrm{M}}_{g,n}$. In this section, we recall key results on the structure of these sheaves that will be instrumental in Section~\ref{sec: general VOA results}. We highlight two fundamental properties that will serve as foundational tools in the proofs and constructions that follow.

\begin{theorem}{(Propagation of Vacua, \cite[Theorem 3.6]{Cod19}, \cite[Theorem 5.1]{DGT1})}
\label{thm: POV}
Let $V$ be a VOA and let $\pi_{n+1}:\overline{\mathrm{M}}_{g,n+1}\to\overline{\mathrm{M}}_{g,n}$ be the map that forgets the $(n+1)$-th marked point. Then, there exists an isomorphism 
\[
\pi^*_{n+1}\mathcal{F}(V,M^1\otimes\cdots\otimes M^n)=\mathcal{F}(V,M^1\otimes\cdots\otimes M^n\otimes V). 
\]
of sheaves of coinvariants on $\overline{\mathrm{M}}_{g,n+1}$. 
\end{theorem}
\begin{theorem}{(Factorization theorem, \cite[Theorem 6.2.6]{TUY89},\cite[Theorem 7.0.1]{DGT2},\cite{DGK})}
\label{thm: factorization}
Let $V$ be a rational VOA and $\mathcal{F}(V,M^{\bullet})$ be a coherent sheaf of coinvariants associated with admissible $V$-modules $M^1,\ldots,M^n$. Let $\xi_I:\overline{\mathrm{M}}_{g_1,|I|+1}\times\overline{\mathrm{M}}_{g_2,|I^c|+1}\to\overline{\mathrm{M}}_{g,n}$ be the clutching map with $I\s[n]$. There exists a canonical isomorphism 
\[
\xi^*\mathcal{F}(V,\bigotimes_{i\in[n]}M^i)\cong\bigoplus_{S\in\mathcal{W}}\pi^*_1\mathcal{F}(V,\bigotimes_{i\in I}M^i\otimes S)\otimes\pi^*_2\mathcal{F}(V,\bigotimes_{i\in I^c}M^i\otimes S'),
\]
where the sum runs over all simple modules $S$ of the fusion ring of $V$ and $S'$ is the dual of the module $S$.
Similarly, for the clutching map $\xi_{\text{irr}}:\overline{\mathrm{M}}_{g-1,n+2}\to \overline{\mathrm{M}}_{g,n}$, we have a canonical isomorphism
\[
\xi^*_{\text{irr}} \mathcal{F}(V,\bigotimes_{i\in[n]}M^i)
\cong \bigoplus_{S\in \mathcal{W}}\mathcal{F}(V,\bigotimes_{i\in[n]}M^i\otimes S\otimes S').
\]
\end{theorem}

The sheaf of coinvariants $\mathcal{F}(V, M^{\bullet})$ on $\overline{\mathrm{M}}_{g,n}$ is well-defined whenever the $V$-modules $M^1, \ldots, M^n$ are simple \cite[Remark 4.3]{DGK}. If, in addition, the $V$ is $C_2$-cofinite, then this sheaf is coherent \cite[Corollary 4.2]{DGK}. When $V$ is both rational and $C_2$-cofinite, $\mathcal{F}(V, M^{\bullet})$ is known to form a vector bundle on $\overline{\mathrm{M}}_{g,n}$ \cite[VB corollary]{DGT1}. While vector bundle structures can persist under conditions weaker than rationality and $C_2$-cofiniteness (see \cite{DGK, DGK2} for more details), in this work we restrict to the case where $V$ is rational and assume that all sheaves of coinvariants possess vector bundle structures. We impose rationality in order to invoke the factorization theorem, which is a foundational tool for the developments in this paper. Moreover, rationality implies that the set $\mathcal{R}^{\mathrm{sim}}$ of isomorphism classes of simple $V$-modules is finite, and that every admissible $V$-module decomposes as a direct sum of simple modules. Consequently, one can only work with simple modules. To explicitly reflect the standing assumption on the sheaves, we introduce the following definition.

\begin{definition}
\label{def: coinvariant vb}
Let $V$ be a vertex operator algebra and assume that the sheaf of coinvariants $\mathcal{F}(V, M^{\bullet})$ on $\overline{\mathrm{M}}_{g,n}$ associated with $V$-modules $M^1,\ldots, M^n$ is a vector bundle. Then, we will say this is a \emph{coinvariant vector bundle} associated to representations $M^{\bullet}$ of $V$ and denote it by $\mathbb{V}_{g,n}(V,M^{\bullet})$. Moreover, its first Chern class is called the associated \emph{coinvariant divisor}, and we denote it by $\mathbb{D}_{g,n}(V, M^{\bullet})$. We sometimes call the divisor $\mathbb{D}_{0,n}(V,M^{\otimes n})$ associated to a simple $V$-module $M$ a \emph{symmetric coinvariant divisor} as it is a divisor on $\tilde{\mathrm{M}}_{0,n}/Sym(n)$, where the group $Sym(n)$ acts by permuting the $n$ marked points.
\end{definition}
\begin{remark}
If $V$ is rational, the coinvariant vector bundles $\mathbb{V}_{g,n}(V,M^{\bullet})$ satisfy Theorems~\ref{thm: POV} and~\ref{thm: factorization}. 
\end{remark}

Since the foundational work of \cite{TUY89}, coinvariant vector bundles associated with affine vertex operator algebras have been a subject of ongoing interest in algebraic geometry. The fibers of the corresponding conformal blocks bundles—defined as the duals of these coinvariant vector bundles—are canonically isomorphic to spaces of generalized theta functions \cite{Thaddeus94, BL94, Faltings94, KNR94, Pauly96, LS97, BLS98, Huang05}. The total Chern character of such coinvariant vector bundles defines a cohomological field theory \cite{MOP, MOP+2}, a construction extended in \cite{DGT3} to all rational, \(C_2\)-cofinite VOAs of CFT-type. Of particular relevance to this work, the coinvariant bundles on \(\overline{\mathrm{M}}_{0,n}\) arising from affine VOAs are globally generated, and their first Chern classes define base-point-free divisors that induce morphisms from \(\overline{\mathrm{M}}_{0,n}\) to projective varieties \cite{Fak12}. Some of these morphisms pull back ample line bundles from target varieties with GIT constructions and modular interpretations \cite{Noah, NoahAngela, Veronese}. Recent work in \cite{DG23} extends the global generation result to bundles associated with strongly generated VOAs in degree one. In this paper, we establish several criteria that extend these positivity results to higher-degree cases. Our proofs rely on the geometry of \(F\)-curves and the following explicit formula for coinvariant divisors.

\begin{theorem}[\cite{DGT3, DGK2}]
\label{thm:coinvariant divisor formula}
Let $V$ be a rational and $C_1$-cofinite VOA of CFT-type with central charge $c$ and let $M^i$ be simple $V$-modules with conformal weight $a_i$. Assume that $c,a_i\in\qq$. Then, the coinvariant divisor $\mathbb{D}_{g,n}(V,M^{\bullet})$ associated to the coinvariant vector bundle $\mathbb{V}_{g,n}(V,M^{\bullet})$ can be written as
\[
\mathbb{D}_{g,n}(V,M^{\bullet}) = \text{rank}\mathbb{V}_{g,n}(V,M^{\bullet})\left(\frac{c}{2}\lambda+\sum_ia_i\psi_i\right)-b_{\text{irr}}\delta_{\text{irr}}-\sum_{i,I}b_{i,I}\delta_{i,I},\text{ where}
\]
\[
b_{\text{irr}} = \sum_{W\in\mathcal{W}}a_W\ \text{rank}\mathbb{V}_{g-1,n+2}(V,M^{[n]}\otimes W\otimes W'), \text{ and}
\]
\[
b_{g,I}=\sum_{W\in\mathcal{W}}a_W\ \text{rank}\mathbb{V}_{i,|I|+1}(V,M^{I}\otimes W)\ \text{rank}\mathbb{V}_{g-i,n-|I|+1}(V,M^{I^c}\otimes W').
\]

Here, $a_W$ is the conformal weight of $W$. For any $I\s[n]=\{1,2,\ldots,n\}$, we set $M^I=\otimes_{j\in I}M^j$ and the last sum is taken over all tuples $(i,I)\in\{0,1,\ldots,g\}\times[n]$ modulo the relation $(i,I)\equiv(g-i,I^c)$, where $I^c:=[n]\setminus I$.
\end{theorem}
In particular, for genus $g=0$, the identity simplifies to
\begin{equation}\label{eqn: first chern class}
\mathbb{D}_{0,n}(V,M^{\bullet}) = \text{rank}\mathbb{V}_{0,n}(V,M^{\bullet})\left(\sum_ia_i\psi_i\right)-\sum_{\substack{I\s[n]\\ 2\leq |I|\leq n/2}}b_{0,I}\delta_{0,I}.    
\end{equation}
Finally, we briefly discuss the results we need about \emph{$F$-curves} for this paper. There are six types of $F$-curves; we will only discuss that of type $6$ as this is the only type we use. Details about $F$-curves can be found in \cite{GKM}.
\begin{definition}
\label{def: F-curve}
The moduli space $\overline{\mathrm{M}}_{0,n}$ admits a stratification where the codimension $i$ strata consist of the stable rational curves with at least $i$ nodes. The numerical equivalence classes of the irreducible components of the codimension $(n-4)$ strata are called \emph{$F$-curves}. In particular, $F$-curves are numerically equivalent to image of the maps $\overline{\mathrm{M}}_{0,4}\to\overline{\mathrm{M}}_{0,n}$ described by a partition $I_1\cup I_2\cup I_3\cup I_4=[n]$ as follows: to any rational curve $C_0$ of with four marked points $p_1,\ldots,p_4$, attach fixed rational curves $C_j$ with $|I_j|+1$ marked points $q^j_1,\ldots,q^j_{|I_j|},q^j_{|I_j|+1}$ to $C_0$ by attaching points $p_i$ and $q^j_{|I_j|+1}$ for each $1\leq i\leq 4$. We will denote such $F$-curves by $F_{I, J, K, L}$. 
\end{definition}
A divisor on $\overline{\mathrm{M}}_{0,n}$ is called an \emph{$F$-divisor} if it intersects all $F$-curves non-negatively. Every nef divisor is evidently an $F$-divisor. The converse, originally posed as a question in \cite{KM13}, is known as the \emph{$F$-conjecture} (see also \cite{GKM}). The following result serves as a base case for arguments in Theorem~\ref{thm: intro1.i}.
\begin{theorem}[\cite{KM13,Lar11,Fed20}]
\label{thm: KM Fcong result}
$F$-conjecture holds on $\overline{\mathrm{M}}_{0,n}$ for all $n\leq 7$ and for symmetric cases $\tilde{\mathrm{M}}_{0,n}=\overline{\mathrm{M}}_{0,n}/S_n$ upto $n\leq 35$.
\end{theorem}
\begin{remark}
While we only focus on the genus $0$ case, the $F$-conjecture is stated for $\overline{\mathrm{M}}_{g,n}$ with any genus $g$. So far, we know that the $F$-conjecture holds for $\overline{\mathrm{M}}_{g}$ with $g\leq 35$ \cite{Fed20} and for $\overline{\mathrm{M}}_{g,n}$ with $g+n\leq 7$ \cite{GKM}. Moreover, by \cite[Theorem 0.3]{GKM}, the $F$-conjecture for $\overline{\mathrm{M}}_{g,n}$ can be reduced to that for $\overline{\mathrm{M}}_{0,g+n}$.
\end{remark}

\subsection{Parafermion vertex operator algebras}
\label{subsec: parafermions}
We summarize the results and notation related to parafermion vertex operator algebras that will be used in Section~\ref{sec: parafermions}.

\subsubsection{Notation}
\label{subsubsec: para and notation}
In this subsection, we establish some notation needed for the rest of the paper corresponding to the representation theory of Lie algebras. Readers are recommended to consult \cite{Hum72} for details.

Let $\liealg$ be a finite-dimensional simple Lie algebra of rank $l$. Let  $\Delta$ denote the root system and $\Delta_+$ denote the set of positive roots. The simple roots are $\ga_1,\ldots,\ga_l$, and the root lattice is $Q=\sum_{i=1}^l\zz\ga_i$. Denote the highest root as $\theta$ and $\rho=\frac{1}{2}\sum_{\ga\in\Delta_+}\ga$. The Lie algebra $\liealg$ has lie sub-algebras $\liealg^{\ga}$ associated with each $\ga\in\Lambda_+$ that is isomorphic to $\slalg_2$ as Lie algebras, that is, as a vector space $\liealg^{\ga}=\cc x_{\ga}+\cc h_{\ga}+\cc x_{-\ga}$, with the isomorphism to $\slalg_2$ given by $x_{\ga}\mapsto\begin{pmatrix}0&1\\0&0\end{pmatrix}$, $x_{-\ga}\mapsto\begin{pmatrix}0&0\\1&0\end{pmatrix}$ and $h_{\ga}\mapsto\begin{pmatrix}1&0\\0&-1\end{pmatrix}$.

Normalize the Cartan-Killing form $\cyc{-,-}$ so that $\cyc{\ga,\ga}=2$ if $\ga\in\Delta$ is a long root. Let $\heisen\s\liealg$ denote the Cartan lie subalgebra. The fundamental weights $\Lambda_i$ of $\liealg$ are defined as the roots $\Lambda_i\in Q$ so that $\frac{2\cyc{\Lambda_i,\ga_j}}{\cyc{\ga_j,\ga_j}}=\delta_{ij}$ for all $1\leq j\leq l$. The weight lattice $P$ of $\liealg$ is the set of elements $\lambda\in\mathfrak{h}^*$ so that $\frac{2\cyc{\lambda,\ga}}{\cyc{\ga,\ga}}\in\zz$ for all $\ga\in\Delta$ and we can write it in terms of the fundamental weights as $P=\bigoplus_{i=1}^l\zz\Lambda_i.$
The set of positive weights is denoted by $P_+$, and we denote by $P_+^k$ the subset
$P_+^k=\{\Lambda\in P\mid \cyc{\Lambda,\theta}\leq k\}.$
The sub-lattice of $Q$ spanned by the long roots is denoted  by $Q_L$ and its dual is the set $Q^{\perp}_L=\{\lambda\in\mathfrak{h}^*\mid\cyc{\lambda,\ga}\in\zz,\quad\forall\ga\in  Q_L\}$, which is equal to $P$ by \cite[Lemma 3.1]{ADJR18}. 

Finally, the finite-dimensional irreducible $\liealg$-modules are completely characterized by the set of dominant integral weights $\Lambda$, and we denote them by $L_{\liealg}(\Lambda)$. The weight space decomposition of $L_{\liealg}(\Lambda)$ is given by $L_{\liealg}(\Lambda)=\bigoplus_{\lambda\in\heisen^*}L_{\liealg}(\Lambda)_{\lambda}$, where $L_{\liealg}(\Lambda)_{\lambda}$ is the weight space of $L_{\liealg}(\Lambda)$ with weight $\lambda$. Let $P(L_{\liealg}(\Lambda))=\{\lambda\in\heisen^*\mid L_{\liealg}(\Lambda)_{\lambda}\ne0\}$. 
\subsubsection{Description and notation for the parafermion VOA}\label{sec: parafermion construction}
In this section, we follow \cite{FZ92,DL93,LL04} for the description of affine vertex operator algebras, and \cite{DLY09,DLWY10,DW10,DW11,DW16,DR17,ADJR18} for parafermion vertex operator algebras.

Recall that $\hat{\liealg}=\liealg\oplus\cc[t^{\pm}]\oplus\cc K$ is the affine Lie algebra associated to $\liealg$ with Lie bracket
\[
[a(m),b(n)]=[a,b](m+n)+m\cyc{a,b}\delta_{m+n,0}K,\text{ and } [K,\hat{\liealg}]=0,
\]
for $a,b\in\liealg$ and $m,n\in\zz$ where $a(m)=a\otimes t^m$. 
Let $L_{\liealg}(\Lambda)$ be the irreducible $\liealg$-module of highest weight $\Lambda\in\heisen^*$. Then we get an induced $\hat{\liealg}$-module $V_{\hat{\liealg}}(k,\Lambda)$ associated to an integer $k\geq1$ given by
\[
V_{\hat{\liealg}}(k,\Lambda)
=Ind^{\hat{\liealg}}_{\liealg\otimes\cc[t]\oplus\cc K}L_{\liealg}(\Lambda),
\]
where $\liealg\otimes\cc[t]$ acts as $0$, $\liealg=\liealg\otimes t^0$ acts as $\liealg$ and $K$ acts as $k\cdot id$ on $L_{\liealg}(\Lambda)$. Then $\hat{\liealg}$-algebra $V_{\hat{\liealg}}(k,\Lambda)$ has a vertex operator structure with the vacuum vector $\id$ and the Virasoro vector
\[
\omega_{\text{aff}}=\frac{1}{2(k+h^{\vee})}
\left(\sum_{i=1}^l u_i(-1)u_i(-1)\id+\sum_{\ga\in\Delta}\frac{\cyc{\ga,\ga}}{2}x_{\ga}(-1)x_{\ga}(-1)\right)
\]
of central charge $\frac{k\dim\liealg}{k+h^{\vee}}$, where $h^{\vee}$ is the dual Coxeter number of $\liealg$ and $\{u_1,\ldots,u_l\}$ an orthonormal basis of $\heisen$. The  has a unique maximal submodule $\mathcal{I}$ of the $\hat{\liealg}$-module $V_{\hat{\liealg}}(k,\Lambda)$, we get the irreducible $\hat{\liealg}$-module $L_{\hat{\liealg}}(k,\Lambda)$.
\begin{theorem}[Theorem 3.1.3 in \cite{FZ92}]
For any positive integer $k$, the $\hat{\liealg}$-module $L_{\hat{\liealg}}(k,0)$ admits the structure of a simple and rational vertex operator algebra. Moreover, the modules  $L_{\hat{\liealg}}(k,\Lambda)$ with conformal weight $\frac{\cyc{\Lambda,\Lambda+2\rho}}{2(k+h^{\vee})}$, where $\Lambda\in P^k_+$,
gives the complete list of irreducible $L_{\hat{\liealg}}(k,0)$-modules.
\end{theorem}

Note that $\hat{\heisen}=\heisen\otimes\cc[t^{\pm}]\oplus\cc K$ is a subalgebra of $\hat{\liealg}$. Let $M_{\hat{\heisen}}(k)$ be the vertex operator subalgebra of $V_{\hat{\liealg}}(k,0)$ generated by $h(-1)\id$ for $h\in\heisen$ with the Virasoro element 
\(\omega_{\heisen} =\frac{1}{2k}\sum_{i=1}^lu_i(-1)u_i(-1)\id\) of central charge $l$. For each $\lambda\in\heisen^*$, we denote by $M_{\hat{\heisen}}(k,\lambda)$ the irreducible highest weight module for $\hat{\heisen}$ with a highest weight vector $e^{\lambda}$ such that $h(0)e^{\lambda}=\lambda(h)e^{\Lambda}$ for $h\in\heisen$.
The parafermion VOAs are defined as $K(\liealg,k):=Com(M_{\hat{\heisen}}(k), L_{\hat{\liealg}}(k,0))$, the commutant of $M_{\hat{\heisen}}(k)$ in $L_{\hat{\liealg}}(k,0)$. For the reader's convenience, we provide the theorem describing commutants. 
\begin{theorem}[Theorem 5.1, 5.2 in \cite{FZ92}]
Let $V$ be a VOA of CFT-type, $W\s V$ a vertex operator subalgebra so that $\omega$ and $\omega'$ are the Virasoro vectors of $V$ and $W$ respectively satisfying $L_1\omega'=0$. The commutant 
\[
Com(W,V):=\{b\in V\mid a(n)b=0,\forall n\geq0,\forall a\in W\}
\]
is a vertex operator subalgebra of $V$ with the Virasoro vector $\omega-\omega'$. Moreover, $Com(W,V)=\{b\in V\mid L'_1b=0\}.$
\end{theorem}
Using the theorem above, we see that 
\(
N(\liealg,k):=Com(M_{\hat{\heisen}(k)}, V_{\hat{\liealg}}(k,0))=\{v\in V_{\hat{\liealg}}(k,0)\mid h(n)v=0,h\in\heisen,n\geq0\}
\)
is a vertex operator algebra with Virasoro vector $\omega=\omega_{\text{aff}}-w_{\heisen}$ whose central charge is $\frac{k\dim \liealg}{k+h^{\vee}}-l$. 
For $\ga\in\Delta_+$, let $k_{\ga}=\frac{\cyc{\theta,\theta}}{\cyc{\ga,\ga}}k\geq0$ and set
\(
\omega_{\ga}=\frac{1}{2k_{\ga}(k_{\ga}+2)}
\left(-k_{\ga}h_{\ga}\id-h_{\ga}(-1)^2\id+2k_{\ga}x_{\ga}(-1)x_{-ga}(-1)\id\right)
\), and
\begin{align*}
W^3_{\ga}
&=k_{\ga}^2h_{\ga}(-3)\id+3k_{\ga}h_{\ga}(-2)h_{\ga}(-1)\id+2h_{\ga}(-1)^3\\
&\quad-6k_{\ga}h_{\ga}(-1)x_{\ga}(-1)x_{-\ga}(-1)\id +3k_{\ga}^2x_{\ga}(-2)x_{-\ga}(-1)\id-3k_{\ga}^2x_{\ga}(-1)x_{-\ga}(-2)\id.
\end{align*}
\begin{theorem}[Theorem 3.1 in \cite{DW10}]
The VOA $N(\liealg,k)$ is generated by $\dim\liealg-l$ vectors $\omega_{\ga}$ and $W^3_{\ga}$ for $\ga\in\Delta_+$. For a fixed $\ga\in\Delta_+$, the subalgebra $\hat{P}_{\ga}$ of $N(\liealg,k)$ generated by $\omega_{\ga}$ and $W^3_{\ga}$ is isomorphic to $N(\slalg_2,k)$. 
\end{theorem}
\begin{remark}
As it is pointed out in Remark 3.1 in \cite{DR17}, it is proved in \cite{DLY09} that $N(\slalg_2,k)$ is strongly generated by $\omega_{\ga},W_{\ga}^3,W_{\ga}^4,W_{\ga}^5$, where $\ga$ is the unique positive root of $\slalg_2$. Here $W^4_{\ga}, W^5_{\ga}$ are the highest vectors of weight 4 and 5, which can be found in \cite{DLY09, DW10}. However, it is unclear if this generalizes to Lie algebras $\liealg$.
\end{remark}

Similar to $M_{\hat{\heisen}}(k)\s V_{\hat{\liealg}}(k,0)$ as subalgebras, the Heisenberg Lie algebra $M_{\hat{\heisen}}(k)$ is a simple subalgebra of $L_{\hat{\liealg}}(k,0)$ and the parafermion vertex operator algebra $K(\liealg,k)$ is the commutant of $M_{\hat{\heisen}}(k)$ in $L_{\hat{\liealg}}(k,0)$
\[
K(\liealg,k)=Com\left(M_{\hat{\mathfrak{h}}}(k),L_{\hat{\liealg}}(k,0)\right)=\{v\in L_{\hat{\liealg}}(k,0)\mid \omega_{\mathfrak{h}}(0)\cdot v=0\},
\]
 with the Virasoro vector given in terms of the Virasoro vectors $\omega_{\ga}$ associated to subalgebras $\liealg^{\ga}$, for $\ga\in \Delta^+$, by 
$$\omega=\sum_{\alpha\in\Delta_+}\frac{k(k_{\alpha}+2)}{k_{\alpha}(k+h^{\vee})}\omega_{\alpha}.$$

Note that $K(\liealg,k)$ is a quotient of $N(\liealg,k)$ and we still denote by $w_{\ga}, W^3_{\ga}$ for their images in $K(\liealg,k)$. 
\begin{proposition}[Proposition 4.1 in \cite{DR17}, Theorem 4.2 in \cite{DW10}]
The subalgebra of $K(\liealg,k)$ generated by $w_{\ga},W^3_{\ga}$ is isomorphic to $K(\liealg^{\ga},k_{\ga})$ and the vertex operator algebra $K(\liealg,k)$ is generated by $\omega_{\ga_i}, W^3_{\ga_i}$ for $i=1,\ldots,l$. 
\end{proposition}
Rationality of $K(\liealg,k)$ is proved for $\mathfrak{g}=\mathfrak{sl}_2$ in \cite{ALY14} and using a different method, for any finite-dimensional simple Lie algebra $\mathfrak{g}$ in \cite{DL17}. The parafermions $K(\liealg,k)$ are also $C_2$-cofinite for any simple Lie algebra $\liealg$ and any integer $k\geq1$ \cite[Theorem 10.5]{ALY14}. Finally, it is of CFT-type \cite[Theorem 3.3 (1)]{ADJR18}. Therefore, $K(\liealg,k)$ is a vertex operator algebra of CohFT-type and satisfies the Assumption~\ref{assump: VB}.
\subsubsection{Module structure}
In this subsection, we review the module structure for the general case $K(\liealg,k)$. Readers are referred to Section~\ref{subsec: parafermions and modules sl} for discussion in the case of $\liealg=\slalg_{r+1}$, which can be seen as an extended example for this subsection.

The $L_{\hat{\liealg}}(k,0)$-modules $L_{\hat{\liealg}}(k,\Lambda)$, for $\Lambda\in P^k_+$, are completely reducible $M_{\hat{\heisen}}(k)$-modules with decomposition
\[
L_{\hat{\liealg}}(k,\Lambda)=\bigoplus_{\lambda\in \Lambda+Q} L_{\hat{\mathfrak{h}}}(k,\Lambda)(\lambda);\quad 
L_{\hat{\mathfrak{h}}}(k,\Lambda)(\lambda)=M_{\hat{\mathfrak{h}}}(k,\lambda)\otimes M^{\Lambda,\lambda}
\] 
as $M_{\hat{\mathfrak{h}}}(k)$-modules, where
\(
M^{\Lambda,\lambda}=\{v\in L_{\hat{\liealg}}(k,\Lambda)\mid h(m)v=\lambda(h)\delta_{m,0}v\text{ for }h\in\mathfrak{h}, m\geq0\}
\). The following theorem presents the necessary results for the $K(\liealg,k)$-modules $M^{\Lambda,\lambda}$.
\begin{theorem}[\cite{Li01,ALY14,DR17}]
\label{thm: module theorem}
\begin{enumerate}[(a)]
\item $M^{0,0}=K(\mathfrak{g},k)$.
\item Let $\Lambda\in P^k_+$ and $\lambda\in\Lambda+Q$. Then, $M^{\Lambda,\lambda}$ is an irreducible $K(\mathfrak{g},k)$-module.
\item Let $\Lambda\in P^k_+$ and $\lambda\in\Lambda+Q$. Then, $M^{\Lambda,\lambda+k\gb}\cong M^{\Lambda,\lambda}$ for any $\gb\in Q_L$. 
\item Let $\theta=\sum_{i=1}^la_i\ga_i$. Denote $I=\{i\in\{1,\ldots,l\}\mid a_i=1\}$. It can be proven that $|I|=|P/Q|-1$.
\item For each $i\in I$ and each $\Lambda\in P^k_+$, there is a unique $\Lambda^{(i)}\in P^k_+$ such that for any $\lambda\in\Lambda+Q$, $M^{\Lambda,\lambda}\cong M^{\Lambda^{(i)},\lambda+k\Lambda_i}$.
\end{enumerate}    
\end{theorem}
\begin{proof}
The proof for (d) can be found in \cite{Li01}, and the rest of the results are proved in \cite{ALY14, DR17}.
\end{proof}
The above results indicate that there are at least
\[
\frac{|P^k_+||Q/kQ_L|}{|P/Q|}
\]
many inequivalent irreducible $K(\mathfrak{g},k)$-modules. 
\begin{theorem}[Theorem 5.1 in \cite{ADJR18}]
\label{thm: these are all simple modules}
These are all inequivalent irreducible modules of $K(\mathfrak{g},k)$.
\end{theorem}
Before discussing the fusion rules for these modules, we wish to expand on statement (e) in Theorem~\ref{thm: module theorem} for the reader's convenience. In particular, the following theorem can be useful in finding $\Lambda^{(i)}$ in the statement. 
\begin{theorem}({\cite{Li97,Li01},\cite[Theorem 4.5]{DL17}})
\label{thm: Li's result}
For any $\Lambda\in P^k_+$, we have 
\(
L_{\hat{\mathfrak{g}}}(k,k\Lambda_i)\boxtimes
L_{\hat{\mathfrak{g}}}(k,\Lambda) 
= L_{\hat{\mathfrak{g}}}(k,\Lambda^{(i)}).
\)
\end{theorem}

\begin{example}\label{exmp: Li's sl2}
Using the theorem above, we calculate $\Lambda^{(1)}$ explicitly for affine vertex operator algebra $L_{\hat{\slalg_2}}(k,0)$. For any weight $\Lambda=a(\ga/2)=a\Lambda_1$, where $a\in\{0,1,2,\ldots,k\}$, we see that $\Lambda^{(1)}=(k-a)(\ga/2)$ from the following:
\begin{align*}
L_{\hat{\mathfrak{sl}_2}}(k,k(\ga/2))\boxtimes
L_{\hat{\mathfrak{sl}_2}}(k,a(\ga/2))
&=\sum_{\substack{l=|k-a|\\ a+k+l\in2\zz }}^{\min(a+k,k-a)} L_{\hat{\mathfrak{sl}_2}}(k,l(\ga/2))
\quad = L_{\hat{\mathfrak{sl}_2}}(k,(k-a)(\ga/2)).
\end{align*}
\end{example}
\begin{notation}
For rest of the paper, we consider modules of the form $M^{\Lambda,\Lambda+\lambda}$ where $\Lambda\in P^k_+$ and $\lambda\in Q/kQ_L$, as the fusion product between two $K(\liealg,k)$-modules are defined in terms of these modules (see Theorem~\ref{thm: para fusion product}). 
\end{notation}

The conformal weights of the modules are defined in the case of $\liealg=\slalg_2$ in \cite[Proposition 4.5]{DLY09} and for a general Lie algebra $\liealg$ in \cite[Lemma 3.3]{DR17}. We describe the general case below. For a given $\Lambda\in\heisen^*$, let $\lambda\in L_{\liealg}(\Lambda)$ be a weight of the $\liealg$-module $L_{\liealg}(\Lambda)$. Then, the conformal weight $f^{\Lambda,\Lambda+\lambda}$ of $M^{\Lambda,\Lambda+\lambda}$, whenever defined, is  
\[
f^{\Lambda,\Lambda+\lambda}:=\frac{\cyc{\Lambda,\Lambda+2\rho}}{2(k+h^{\vee})}-\frac{\cyc{\Lambda+\lambda,\Lambda+\lambda}}{2k},
\]
up to equivalence of irreducible inequivalent modules. To understand the need for defining modulo the equivalence relation, 
note that it is not in general true that $\lambda$ is an element of $P(L_{\liealg}(\Lambda))$, for any choices of a module $M^{\Lambda,\Lambda+\lambda}$ with any $\Lambda\in P^k_+$ and any $\lambda\in Q/kQ_L$. However, the modules equivalent to $M^{\Lambda,\lambda}$ form a set  
$\{M^{\Lambda^{(s)},\Lambda+\lambda+k\Lambda_s}\mid 1\leq s\leq l\}$, and the conformal weight of $M^{\Lambda,\Lambda+\lambda}$ is defined to be $f^{\Lambda^{(s)},\lambda+k\Lambda_s+\Lambda}$, for an integer $1\leq s\leq l$ so that $(\lambda+k\Lambda_s+\Lambda)\in P(L_{\liealg}(\Lambda^{(s)}))$. We write it as a definition below for future reference.
\begin{definition}
\label{def: para cw}
The conformal weight, denoted by $cw^{\Lambda,\Lambda+\lambda}$, of $M^{\Lambda,\Lambda+\lambda}$ is 
\(
cw^{\Lambda,\Lambda+\lambda}:=f^{\Lambda^{(s),\lambda+k\Lambda_s+\Lambda}},
\)
where $1\leq s\leq l$ is any integer so that $(\lambda+k\Lambda_s+\Lambda)\in P(L_{\liealg}(\Lambda^{(s)}))$.
\end{definition}
Since two isomorphic modules have the same conformal weight, it is well-defined. Finally, the fusion rules were first established in \cite{DW16} for the case $\mathfrak{g}=\mathfrak{sl}_2$, and this formula was generalized in \cite{ADJR18} to any finite-dimensional simple Lie algebra $\mathfrak{g}$. In both cases, fusion rules for affine vertex operator algebra $L_{\hat{\mathfrak{g}}}(k,0)$ were used, which is given as follows: the fusion product of two irreducible $L_{\hat{\mathfrak{g}}}(k,0)$-modules are given by 
\[
L_{\hat{\mathfrak{g}}}(k,\Lambda^1)\boxtimes L_{\hat{\mathfrak{g}}}(k,\Lambda^2)=\sum_{\Lambda^3\in P^k_+}N^{\Lambda^3}_{\Lambda^1,\Lambda^2}\quad L_{\hat{\mathfrak{g}}}(k,\Lambda_3),
\]
where $\Lambda^1,\Lambda^2\in P^k_+$ and $N^{\Lambda^3}_{\Lambda^1,\Lambda^2}$ are the fusion rules for the irreducible $L_{\hat{\mathfrak{g}}}(k,0)$-modules.
\begin{theorem}[Theorem 5.2 in \cite{ADJR18}]
\label{thm: para fusion product}
Let $\Lambda^1,\Lambda^2\in P^k_+$ and $i,j\in Q/kQ_L$. Then,
\begin{equation*}
M^{\Lambda^1,\Lambda^1+\gb_i}\boxtimes_{K(\mathfrak{g},k)}M^{\Lambda^2,\Lambda^2+\gb_j}=
\sum_{\Lambda^3\in P^k_+}N^{\Lambda^3}_{\Lambda^1,\Lambda^2}\quad M^{\Lambda^3,\Lambda^1+\Lambda^2+\gb_i+\gb_j}.
\end{equation*}
Moreover, $M^{\Lambda^3,\Lambda^1+\Lambda^2+\gb_i+\gb_j}$ with $N^{\Lambda^3}_{\Lambda^1,\Lambda^2}\ne0$ are the inequivalent irreducible $K(\mathfrak{g},k)$-modules. 
\end{theorem}
\section{Positivity of coinvariant divisors}
\label{sec: general VOA results}
In this section, we describe several criteria for the positivity of coinvariant divisors on $\overline{\mathrm{M}}_{0,n}$ arising from representations of a vertex operator algebra $V$ satisfying the following assumption.
\begin{assumption}
\label{assump: VB}
$V$ is a self-contragradient and rational vertex operator algebra of CFT-type. Moreover, all sheaves of coinvariants $\mathcal{F}(V,M^{\bullet})$ on $\overline{\mathrm{M}}_{0,n}$ associated with a collection of $n$ admissible $V$-modules $M^{\bullet}$ are vector bundles. Finally, the central charge of $V$ and the conformal weight of each simple $V$-module are rational numbers.
\end{assumption}
Indeed, an extensive collection of self-contragradient CFT-type VOAs satisfy Assumption~\ref{assump: VB}. See Section~\ref{subsubsec: results for divisors} for more details. We first need the following definition to state the three criteria we establish.
\begin{definition}
\label{def: positive subring}
A subring $\mathcal{S}$ of the fusion ring of a vertex operator algebra $V$ is \emph{$F$-positive} if for any four elements $M^1,\ldots,M^4\in\mathcal{S}$, the degree of the corresponding coinvariant divisor $\mathbb{D}_{0,4}(V,M^{\bullet})$ is non-negative. 

Moreover, $\mathcal{S}$ is \emph{positive} if for any collection of $n\geq8$ elements $M^1,\ldots,M^n\in\mathcal{S}$, the divisor $\mathbb{D}_{0,n}(V,M^{\bullet})$ can be written as $cK_{\overline{\mathrm{M}}_{0,n}}+E$ for some $c\in\qq_{\geq0}$ and an effective sum $E$ of boundary divisors.
\end{definition}
\begin{remark}
The definition for $F$-positivity of a subring comes naturally from the factorization theorem (Theorem~\ref{thm: factorization}). See the proof of Lemma~\ref{lem: Fpositive} for an explicit description. The positivity criterion of subrings may appear more complicated to check, but the simplicity of checking this criterion lies in Proposition~\ref{lem: pullback of divisors}. In particular, because of Proposition~\ref{lem: pullback of divisors}, the positivity of an $F$-positive subring reduces to simply calculating conformal weights of all simple $V$-modules in $\mathcal{S}$ and checking for certain inequalities. For an explicit example, see the proof of Theorem~\ref{thm: positive for sl3}.
\end{remark}
\begin{theorem}[Theorem~\ref{thm: intro1}.i.]
\label{thm: intro1.i}
A coinvariant divisor on $\overline{\mathrm{M}}_{0,n}$ associated with simple $V$-modules of a positive subring of the fusion ring of $V$ is nef.    
\end{theorem}
In order to prove this theorem, we need the following two results. 
\begin{lemma}
\label{lem: Fpositive}
A coinvariant divisor associated with simple $V$-modules of a $F$-positive subring of the fusion ring is an $F$-divisor; it intersects every $F$-curve non-negatively.
\end{lemma}
\begin{proof}
Let $\mathbb{D}$ be a coinvariant divisor associated with $n$ simple $V$-modules $M^1,\ldots,M^n$ in a $F$-positive subring $\mathcal{S}$.  
\[
\mathbb{D}\cdot F_{I_1,I_2,I_3,I_4}=\sum_{W^{\bullet}\in\mathcal{W}}\deg\mathbb{D}_{0,4}(V,W^{\bullet})\cdot \prod_{i=1}^4\text{ rank }\mathbb{V}_{0,|I_i|+1}(V,M^{I_i}\otimes W^i),
\]
where $\bigsqcup_{i=1}^4 I_i=[n]$ is a partition, and the sum is taken over all four tuples $(W^1,\ldots, W^4)\in \mathcal{W}^{\oplus 4}$ of simple $V$-modules. This formula follows from factorization (Theorem~\ref{thm: factorization}).
We claim that $W^i\in\mathcal{S}^{sim}=\mathcal{S}\cap\mathcal{W}$ for each $i$. If the claim holds, then the conclusion follows immediately. Indeed, since $\mathcal{S}$ is $F$-positive, the degree term in each summand is non-negative, and the rank of vector bundles is always non-negative. 

To prove the claim, it suffices to show that if  $W^i\notin \mathcal{S}$ for any $1\leq i\leq 4$, the corresponding summand contributes trivially to the sum. We prove this by induction. Given $M^1,M^2,M^3\in \mathcal{S}$ and any $W\in \mathcal{W}$, by factorization
\[
\text{rank }\mathbb{V}_{0,4}(V,\bigotimes_{j=1}^3M^j\otimes W')
=\sum_{X\in\mathcal{W}}\text{rank }\mathbb{V}_{0,3}(V,M^1\otimes W'\otimes X)\cdot\text{rank }\mathbb{V}_{0,3}(V,M^2\otimes M^3\otimes X').
\]
Since $M^2, M^3 \in \mathcal{S}$, we see that $\text{rank }\mathbb{V}_{0,3}(V,M^2\otimes M^3\otimes X')=0$ if $X\notin \mathcal{S}$ and hence
\[
\text{rank }\mathbb{V}_{0,4}(V,\bigotimes_{j=1}^3M^j\otimes W')
=\sum_{X\in S}\text{rank }\mathbb{V}_{0,3}(V,M^1\otimes W'\otimes X).
\]
Since $M^1, X\in \mathcal{S}$, we must have $W'\in \mathcal{S}$ in order to possibly have a non-trivial contribution for any of the summands running over $X\in \mathcal{S}$. As the trivial module is in the subring, we have $W\in \mathcal{S}$. Assume that for any $M^1,\ldots,M^{n-1}\in \mathcal{S}$ and any $W\in \mathcal{W}$, we have that $\text{rank }\mathbb{V}_{0,n}(V,\bigotimes_{j=1}^{n-1}M^j\otimes W')=0$ if $W\notin \mathcal{S}$. Then, for any element $M^{n}\in \mathcal{S}$,
\[
\text{rank }\mathbb{V}_{0,n+1}(V,\bigotimes_{j=1}^nM^j\otimes W')
=\sum_{X_1\in \mathcal{W}} \text{rank }\mathbb{V}_{0,n}(V,\bigotimes_{j=1}^{n-2}M^j\otimes W'\otimes X_1)
\cdot \text{rank }\mathbb{V}_{0,3}(V,M^{n-1}\otimes M^{n}\otimes X_1').
\]
Since $M^{n-1}, M^{n}\in \mathcal{S}$, we see that $\text{rank }\mathbb{V}_{0,3}(V,M^{n-1}\otimes M^{n}\otimes X_1')=0$ if $X_1\notin \mathcal{S}$ and hence
\begin{align*}
\text{rank }\mathbb{V}_{0,n+1}(V,\bigotimes_{j=1}^nM^j\otimes W')
&=\sum_{X_1\in\mathcal{S}} \text{rank }\mathbb{V}_{0,n}(V,\bigotimes_{j=1}^{n-2}M^j\otimes W'\otimes X_1),
\end{align*}
and by induction hypothesis, if $W\notin\mathcal{S}$, then $\text{rank }\mathbb{V}_{0,n+1}(V,\bigotimes_{j=1}^nM^j\otimes W')=0$. 
\end{proof}
\begin{proposition}
\label{lem: pullback of divisors}
Let $V$ be a rational vertex operator algebra and $\mathcal{S}$ be a subring of the fusion ring of $V$. Under the clutching map $\rho:\overline{\mathrm{M}}_{g_1,|I|+1}\times\overline{\mathrm{M}}_{g_2,|I^c|+1}\to\overline{\mathrm{M}}_{g,n}$ with projection maps $\rho_I:=\text{pr}_1\circ\rho_I$ and $\rho_{I^c}:=\text{pr}_2\circ\rho_{I^c}$, the pullbacks of a coinvariant divisor $\mathbb{D}_{g,n}(V, M^{\bullet})$, associated with simple $V$-modules $M^1,\ldots, M^n$ in $\mathcal{S}$, can be written as an effective sum of coinvariant divisors with irreducible representations in $\mathcal{S}$. 
\end{proposition}
\begin{proof}
Let $\mathbb{V}_{g,n}(V,M^{\bullet})$ be a coinvariant vector bundle associated with simple $V$-modules $M^1,\ldots,M^n\in\mathcal{S}$. By factorization (Theorem~\ref{thm: factorization}), the pullback of $\mathbb{V}_{g,n}(V,M^{\bullet})$ along $\rho_I$ is 
\[
\rho^*_I\mathbb{V}_{g,n}(V,M^{\bullet})=\bigoplus_{W\in \mathcal{S}^{sim}}\mathbb{V}_{g_1,|I|+1}(V,M^I\otimes W)^{\oplus\text{ rank }\mathbb{V}_{g_2,|I^C|+1}(V,M^{I^C}\otimes W')},
\]
where we only consider $W\in\mathcal{W}$ that are in $\mathcal{S}^{sim}=\mathcal{S}\cap\mathcal{W}$, as $\mathcal{S}$ is a subring. Its first Chern class is 
\begin{equation}\label{eqn:pullback}
\rho^*_I\mathbb{D}_{g,n}(V,M^{\bullet}) = \sum_{W\in \mathcal{S}^{sim}}\left( \text{ rank }\mathbb{V}_{g_2,|I^C|+1}(V,M^{I^C}\otimes W')\right)\quad \mathbb{D}_{g_1,|I|+1}(V,M^I\otimes W).    
\end{equation}
Since $\left(\text{rank }\mathbb{V}_{g_2,|I^C|+1}(V,M^{I^C}\otimes W')\right)$ is always non-negative, we are done.
\end{proof}
\begin{proof}[Proof of Theorem~\ref{thm: intro1.i}]
Let $\mathcal{S}$ be a positive subring of the fusion ring, and let $\mathbb{D}$ denote the coinvariant divisor $\mathbb{D}_{0,n}(V, M^{\bullet})$ associated with simple $V$-modules $M^1,\ldots, M^n\in\mathcal{S}$. Using an argument in the proof of \cite[Theorem 3.1]{Gib09}, we show that there exists no extremal ray $R$ on the cone of curves so that $\mathbb{D}\cdot R<0$. For contradiction, assume there is an extremal ray $R$ spanned by an irreducible curve $C$ so that $\mathbb{D}\cdot R<0$. Since $\mathbb{D} = cK+E$ by assumption, for some $c\in\qq_{\geq0}$ and $E$ an effective sum of boundary divisors, we must have $\mathbb{D}\cdot\delta_{0,I}<0$ for some $I\s[n]$, that is, $\rho_I^*\mathbb{D}\cdot C<0$ or $\rho^*_{I^c}\mathbb{D}\cdot C<0$. Without loss of generality, assume that $\rho_I^*\mathbb{D}\cdot C<0$. By Lemma~\ref{lem: pullback of divisors}, $\rho^*_I\mathbb{D}$ is again an effective sum of coinvariant divisors given by the representations of the positive subring $\mathcal{S}$ and therefore, we continue the process of pulling back iteratively until $\mathbb{D}$ pulls back to a divisor, denote it by $\mathbb{D}'$, on $\overline{\mathrm{M}}_{0,t}$ for some $t\leq 7$ with the property that $\mathbb{D}'\cdot C<0$. Note again, $\mathbb{D}'$ is an effective sum of coinvariant divisors that intersects all $F$-curves non-negatively since $\mathcal{S}$ is $F$-positive (by Lemma~\ref{lem: Fpositive}). Then, by Theorem~\ref{thm: KM Fcong result}, $\mathbb{D}'$ is nef, which contradicts $\mathbb{D}'\cdot C<0$. If $R$ is a limit of irreducible curves, a similar argument holds.
\end{proof}
\begin{theorem}[Theorem~\ref{thm: intro1}.ii.]
\label{thm: intro1.ii}
Let \( \mathcal{S} \) be a subring of the fusion ring such that the conformal weight of every simple module \( W \in \mathcal{S} \) is non-negative, and let \( M \) be a simple \( V \)-module in \( \mathcal{S} \) with maximal conformal weight. If the symmetric divisor \( \mathbb{D}_S := \mathbb{D}_{0,n}(V, M^{\otimes n}) \) is an \( F \)-divisor, then it is nef.

\end{theorem}
\begin{proof}
Let $\mathbb{D}:=a\lambda-\sum_{i=0}^{\lfloor n/2\rfloor}b_i\delta_i$ be a divisor in $\overline{\mathrm{M}}_n$. The pullback of $\mathbb{D}$ under the morphism $f:\overline{\mathrm{M}}_{0,n}\to\overline{\mathrm{M}}_{n}$, defined by attaching an elliptic tail to each marked point of a curve in $\overline{\mathrm{M}}_{0,n}$, is given in \cite[Lemma 2.4(i)]{Gib09} as
\[
f^*\mathbb{D}=b_1\psi-\sum_{i=2}^{\lfloor n/2\rfloor}b_iB_i,\quad \text{ where }B_i=\sum_{\substack{I\s[n]\\ |I|=i}}\delta_{0,I}. 
\]
Let $b_i=\sum_{W\in \mathcal{S}}cw^W\text{ rank }\mathbb{V}_{0,i+1}(V,M^{\otimes i}\otimes W)\text{ rank }\mathbb{V}_{0,i+1}(V,M^{\otimes (n-i+1)}\otimes W')$ for each $i\geq 2$ and let $b_1=cw^M\text{ rank }\mathbb{V}_{0,n}(V,M^{\otimes n})$. Then, the coinvariant divisor $\mathbb{D}_S$ is equal to $f^*\mathbb{D}$. We may choose $a$ and $b_0$ however we like so that the inequalities in $(i),(ii),(iv),(v)$ of \cite[Theorem 2.1]{GKM} needed for $\mathbb{D}$ to be an $F$-divisor is satisfied. The inequality in $(vi)$ holds for $\mathbb{D}$ since $\mathbb{D}_S$ is an $F$-divisor, and finally, the inequality in $(iii)$ is satisfied trivially. Therefore, $\mathbb{D}$ is a $F$-divisor on $\overline{\mathrm{M}}_{n}$. Finally, $b_1\geq b_i$ for each $i\geq 2$, since
\begin{align*}
&\sum_{W\in S}c^W\text{ rank }\mathbb{V}_{0,i+1}(V,M^{\otimes i}\otimes W)\text{ rank }\mathbb{V}_{0,i+1}(V,M^{\otimes (n-i+1)}\otimes W')\\
&\leq 
c^M\sum_{W\in S}\text{ rank }\mathbb{V}_{0,i+1}(V,M^{\otimes i}\otimes W)\text{ rank }\mathbb{V}_{0,i+1}(V,M^{\otimes (n-i+1)}\otimes W')=c^M\text{ rank }\mathbb{V}_{0,n}(V,M^{\otimes n}),
\end{align*}
and therefore $\mathbb{D}$ is a nef divisor in $\overline{\mathrm{M}}_{n}$ by \cite[Corollary 5.3]{Gib09}. As the pullback of a nef divisor, $\mathbb{D}_S$ is nef too.
\end{proof}
The last criterion is given by comparing representations of two separate VOAs using the following definition.
\begin{definition}
\label{def: proportional cw pairing}
Let \(V_1\) and \(V_2\) be two VOAs, and let \(\mathcal{S}_1\) (respectively, \(\mathcal{S}_2\)) denote a subring of the fusion ring of \(V_1\) (respectively, \(V_2\)). We say that the pair \((\mathcal{S}_1, \mathcal{S}_2)\) forms a \emph{proportional pairing} if there exists a ring homomorphism \(f : \mathcal{S}_1 \to \mathcal{S}_2\) satisfying the following properties
\begin{enumerate}[i.]
\item $f$ is injective and $f(M)\in \mathcal{S}_2^{sim}$ for all $M\in\mathcal{S}_1^{sim}$, and
\item $cw^{f(M)}/cw^{M}=\eta$ for all $M\in \mathcal{S}_1^{sim}\setminus\{V_1\}$, where $\eta\in\qq_{>0}$ is some constant. 
\end{enumerate}
\end{definition}
\begin{theorem}[Theorem~\ref{thm: intro1}.iii.]
\label{thm: intro1.iii}
Let $V_1$ and $V_2$ be any two VOAs with fusion rings $\mathcal{R}_1$ and $\mathcal{R}_2$, respectively. Assume that there exist two subrings $\mathcal{S}_1\s\mathcal{R}_1$ and $\mathcal{S}_2\s\mathcal{R}_2$ that are propositional. If $\mathbb{D}_{0,n}(V_1,\otimes_i M^{i})$ is nef (resp. base-point free) for $n$-simple $V_1$-modules $M^i\in\mathcal{S}_1$, then $\mathbb{D}_{0,n}(V_2,\otimes_i f(M^i))$ is nef (resp. semi-ample). 
\end{theorem}
\begin{proof}
Since $f$ is injective, the following equality is a direct application of the factorization theorem
\[
\text{rank }\mathbb{V}_{0,n}(V,\bigotimes_{i=1}^n W_i)
= \text{rank }\mathbb{V}_{0,n}(V';\bigotimes_{i=1}^n f(W_i)).
\]
Then, by Equation~\ref{eqn: first chern class}, we conclude that
\(
\mathbb{D}_{0,n}(V';\bigotimes_{i=1}^n f(W_i))=\eta\cdot \mathbb{D}_{0,n}(V,\bigotimes_{i=1}^n W_i).
\)
\end{proof}
Finally, we describe how one can use positivity of coinvariant divisors on $\overline{\mathrm{M}}_{0,n}$ to produce positive divisors over $\overline{\mathrm{M}}_{g,n}$ by adding a positive rational multiple of the Hodge class $\lambda$. To prove this result, we use the flag maps $F_{g,n}:\overline{\mathrm{M}}_{0,4}\to\overline{\mathrm{M}}_{g,n}$ defined in \cite{Gib09} as follows: given a partition $I_1\sqcup I_2\sqcup I_3\sqcup I_4$ of $[n]=\{1,\ldots,n\}$ and another partition $g_1+\cdots+g_4=g$, fix a point $(C_i,I_i\sqcup\{q_i\})$ in $\overline{M}_{g_i,|I_i|+1}$, for each $i=1,\ldots,4$. The map sends a pointed curve $(C,\{p_1,\ldots,p_4\})$ in $\overline{\mathrm{M}}_{0,4}$ to a curve in $\overline{\mathrm{M}}_{g,n}$ by attaching the curves $(C_i,I_i\sqcup\{q_i\})$ to $C$ via $p_i\sim q_i$.  
\begin{proposition}
\label{prop: lambda twist nef high higher genus}
Let $\mathcal{S}$ be a subring of the fusion ring of a VOA $V$ so that all coinvariant divisors on $\overline{\mathrm{M}}_{0,n}$ associated with simple modules in $\mathcal{S}$ are positive. Then, there is a unique $t\in\qq_{\geq0}$ so that for all $q\in\qq_{\geq t}$ and any coinvariant divisor $\mathbb{D}$ on $\overline{\mathrm{M}}_{g,n}$ corresponding to modules in $\mathcal{S}^{sim}$, the divisor $(q\lambda+\mathbb{D})$ is nef.
\end{proposition}
\begin{proof}
This is proven for $\mathcal{S}$ equal to the fusion ring of the affine VOAs in \cite[Proposition 6.5]{Fak12}, and the proof extends more generally. First we prove that $(t\lambda+\mathbb{D})$ is an $F$-divisor on $\overline{\mathrm{M}}_{g,n}$. Since all $F$-curves of type $(2)-(6)$ can be described as image of a flag map $F_{g,n}:\overline{\mathrm{M}}_{0,4}\to\overline{\mathrm{M}}_{g,n}$, as in the proof of Lemma~\ref{lem: Fpositive}, $F_{g,n}^*(t\lambda+\mathbb{D})=F^*_{g,n}\mathbb{D}$ is positive. Since the $F$-curves of type $(1)$ are the image of the flag map $F_{1,n}:\overline{M}_{1,1}\to\overline{\mathrm{M}}_{g,n}$, one again checks that 
\[
F^*_{1,n}\mathbb{D}=F^*_{1,n}\mathbb{D}_{g,n}(V,M^{\bullet}) = \sum_{W\in\mathcal{S}}\deg\mathbb{D}_{1,1}(V,W)\text{ rank }\mathbb{V}_{g-1,n+1}(V,M^{\bullet}\otimes W'),\text{ with}
\]
\[
\deg\mathbb{D}_{1,1}(V,W)=\sum_{\tilde{W}\in\mathcal{W}}(\frac{c}{2}+cw(W)-12cw(\tilde{W}))\text{ rank }\mathbb{V}_{0,3}(V,W\otimes\tilde{W}\otimes\tilde{W}'),
\]
where we have used the fact that $\lambda=\psi_1=\delta_{irr}/12$ on Equation~\ref{eqn: first chern class}.
Positivity of $F^*_{1,n}(t\lambda+\mathbb{D})$ is given by choice of $t$ so that $t+\frac{c}{2}+cw(W)-12cw(\tilde{W})\geq0$ for any $W\in\mathcal{S}^{sim}$ and any $\tilde{W}\in\mathcal{W}^{sim}$. Now pulling back the $F$-divisor $t\lambda+\mathbb{D}$ via the map $F:\overline{M}_{0,n+1}\to\overline{\mathrm{M}}_{g,n}$, which attaches a genus $g$ curve at the $(n+1)$-th marked point, we get an effective sum of coinvariant divisors given by representations in $\mathcal{S}$ (by Proposition~\ref{lem: pullback of divisors}). By Theorem~\ref{thm: intro1.i}, $F^*(t\lambda+\mathbb{D})=F^*\mathbb{D}$ is nef and therefore $t\lambda+\mathbb{D}$ is nef by \cite[Theorem 0.3]{GKM}.
\end{proof}
\section{Parafermions and associated positive divisors on  \texorpdfstring{$\overline{\mathrm{M}}_{0,n}$}{M{0,n}}}
\label{sec: parafermions}
In this section, we describe the $F$-positive and positive subrings of the parafermion algebras $K(\mathfrak{g},k)$ for the Lie algebras $\mathfrak{sl}_2$ and $\mathfrak{sl}_3$. Section~\ref{subsec: parafermions and modules sl} carries out the necessary computations and introduces the notation required to apply the general results presented in Section~\ref{sec: general VOA results}. The structure of the positive subrings for $K(\mathfrak{sl}_2,k)$ is detailed in Section~\ref{subsec: parafermions and sl2}, and the corresponding results for $K(\mathfrak{sl}_3,k)$ are presented in Section~\ref{subsec: parafermions and sl3}.

\subsection{\texorpdfstring{$K(\slalg_{r+1},k)$}{K(sl(r+1),k)}-Modules}
\label{subsec: parafermions and modules sl}
Let $k\geq1$ be an integer and let $\liealg=\slalg_{r+1}$ be the Lie algebra of rank $r$ for this section, some integer $r\geq1$. Let $\{\ga_1,\ldots,\ga_r\}$ denote the simple roots with the normalized Cartan-Killing form satisfying 
\[
\cyc{\ga_i,\ga_i}=2\quad 1\leq i\leq r,\quad 
\cyc{\ga_i,\ga_j}
=\begin{cases}
-1&\text{ if } |i-j|=1\\
0 &\text{ if } |i-j|>1.
\end{cases}
\]
The $s$-th fundamental weight $\Lambda_s$ is given by 
\[
(r+1)\Lambda_s =\sum_{i=1}^s i(r+1-s)\ga_i+\sum_{i=s+1}^r s(r+1-i)\ga_i
\]
for each $1\leq s\leq r$, and it follows that $\cyc{\Lambda_i,\ga_j}=\delta_{i,j}$, where $\delta_{i,j}$ is the Kronecker delta function. The maximal root is $\theta=\sum_{i=1}^r\ga_i$ and 
\(
2\rho=\sum_{\ga\in\Delta_+}\ga = \sum_{1\leq a\leq b\leq r}\sum_{t=a}^b\ga_t.
\)
The root lattice and the weight lattice are $Q=\bigoplus_{i=1}^a\zz \ga_i$ and $P=\bigoplus_{s=1}^r\Lambda_s$, respectively, with $Q_L\s Q$, the sub-lattice of long roots, equal to $Q$. 
\subsubsection{Simple modules, their conformal weights and the fusion rules}
\begin{lemma}
The set of inequivalent irreducible $K(\slalg_{r+1},k)$-modules is a finite set of cardinality $\frac{k+1}{r} \binom{k+r}{r-1}$.
\end{lemma}
\begin{proof}
To prove this lemma, we interpret Theorems~\ref{thm: module theorem} and~\ref{thm: these are all simple modules} for $\mathfrak{g}=\slalg_{r+1}$. The set of inequivalent modules is 
\[
\mathcal{W}_r:=\{M^{\Lambda,\Lambda+\lambda}\mid\Lambda\in P^k_+, \lambda\in Q/kQ\}/\sim,
\]
where $M^{\Lambda,\Lambda+\lambda}\sim M^{\Lambda',\Lambda'+\lambda'}$ if and only if $M^{\Lambda',\Lambda'+\lambda'}\cong M^{\Lambda^{(s)},\Lambda+\lambda+k\Lambda_s}$ for each $1\leq s\leq r$. Here, we have used the fact that for $\slalg_{r+1}$, the maximal root  is $\theta=\sum_{i=1}^r\ga_i$, that is, $a_i=1$ for each $i$ (see Theorem~\ref{thm: module theorem}(d)) and therefore it follows that $I=\{1,2,\ldots,r\}$. It follows from \cite{Li01} that $|P/Q|=|I|+1$, that is, $|P/Q|=r+1$.

Moreover, since $\cyc{\Lambda_i,\ga_j}=\delta_{i,j}$, we obtain that for any $(m_1,\dots,m_r)\in \zz_{\geq0}^{\oplus r}$, we have 
\begin{align*}
\cyc{\sum_{s=1}^r m_s\Lambda_s,\sum_{i=1}^r\ga_i}
=\sum_{s=1}^r\sum_{i=1}^r\delta_{s,i}=\sum_{s=1}^rm_s,
\end{align*}
\[
P^k_+ = \left\{ \sum_{s=1}^r m_s\Lambda_s \mid \sum_{s=1}^r m_s \leq k, \, m_s \in \mathbb{Z}_{\geq0}, \quad \forall 1\leq s\leq r \right\}, \quad 
|P^k_+| = \sum_{n=0}^k \binom{n+r-1}{r-1} = \frac{k+1}{r} \binom{k+r}{r-1},
\]
where $\binom{n+r-1}{r-1}$ counts the number of non-negative integer solutions to $x_1+\cdots+x_r=n$.
Finally, $Q/kQ=\{\sum_{i=1}^ra_i\ga_i\mid (a_1,\ldots,a_r)\in\{0,1,\ldots,k-1\}^{\oplus r}\}$ and therefore, the set $\mathcal{W}_r$ is finite and its cardinality is
\[
|\mathcal{W}_r|=\frac{|P^k_+||Q/kQ_L|}{|P/Q|}
=\frac{(k+1) k^r}{r(r+1)}\binom{k+r}{r-1}.
\]
\end{proof}

\begin{remark}
\label{rem: number of modules grow exp}
Note that for $k=3$, there are $\frac{3^{r-1}}{2}(r+2)(r+3)$ irreducible inequivalent $K(\slalg_{r+1},3)$-modules, which shows that the number of modules in consideration grows exponentially with $r$ for any fixed level $k$.
\end{remark}

Next, we explicitly describe the conformal weight of specific modules we need later in the section. In the second superscript of $M^{\Lambda,\Lambda+\lambda}$, $\lambda\in Q/kQ$ is chosen upto equivalence relation and therefore, we may replace $\lambda$ by $-\lambda$ and consider the conformal weight of the modules $M^{\Lambda,\Lambda-\lambda}$ instead. The advantage of this approach is that all weights of $L_{\liealg}(\Lambda)$ are of the form $\Lambda-\lambda$ for some $\lambda\in\heisen^*$, which will be useful in the proof of Proposition~\ref{prop: conformal weight rep theory result}.

\begin{lemma}\label{lem: cw}
Given $\Lambda_s\in P^k_+$ and $k\Lambda_s-\lambda\in P(L_{\liealg(k\Lambda_s)})$ with $\lambda=\sum_{i=1}^r a_i\ga_i$ for some integers $0\leq a_1,\ldots,a_r<k$ and an integer $1\leq s\leq r$, conformal weight of the simple $K(\slalg_{r+1},k)$-module $M^{k\Lambda_s,k\Lambda_s-\sum_{i=1}^r a_i\ga_i}$ is given by
\[
cw^{k\Lambda_s,k\Lambda_s-\sum_{i=1}^r a_i\ga_i}
=a_s-\frac{1}{k}
\left(\sum_{i=1}^ra_i^2-\sum_{1\leq i<j\leq r}a_ia_j\right).
\]
\end{lemma}
\begin{proof}
We first note that
\[
\cyc{k\Lambda_s,\lambda}=
\cyc{k\Lambda_s,\sum_{i=1}^ra_i\ga_i}=k\sum_{i=1}^ra_i\cyc{\Lambda_s,\ga_i}=ka_s,
\]
\[
\cyc{k\Lambda_s,k\Lambda_s}
=k^2
\left(\sum_{i=1}^s\cyc{\Lambda_s,i(r+1-s)\ga_i}+\sum_{i=s+1}^r\cyc{\Lambda_s,s(r+1-i)\ga_i}\right)=\frac{k^2}{r+1}s(r+1-s),
\]
\[
\cyc{\lambda,\lambda}=
\cyc{\sum_{i=1}^ra_i\ga_i,\sum_{j=1}^ra_j\ga_j}=\sum_{1\leq i,j\leq r}a_ia_j\cyc{\ga_i,\ga_j}=2\sum_{i=1}^ra_i^2-2\sum_{1\leq i<j\leq r}a_ia_j,
\]
\[
\cyc{k\Lambda_s,2\rho}=k|\{((a,b)\mid 1\leq a\leq s\leq b\leq r\}|=ks(r+1-s).
\]
Substituting the equations above to 
\[
cw^{k\Lambda_s,k\Lambda_s-\sum_{i=1}^r a_i\ga_i}
=\frac{1}{2(k+r+1)}\left(\cyc{k\Lambda_s,k\Lambda_s}+\cyc{k\Lambda_s,2\rho}\right)-\frac{1}{2k}\left(\cyc{k\Lambda_s,k\Lambda_s}-2\cyc{k\Lambda_s,\lambda}+\cyc{\lambda,\lambda}\right),
\]
we get the desired identity.
\end{proof}
\begin{proposition}
\label{prop: conformal weight rep theory result}
For integers $0\leq a_1,\ldots,a_r<k$, the conformal weight of the simple module $M^{0,0-\sum_{i=1}^ra_i\ga_i}$ is 
\[
cw^{0,0-\sum_{i=1}^ra_i\ga_i}=\max_{1\leq i\leq r}a_i -\frac{1}{k}\left(\sum_{i=1}^ra_i^2-\sum_{1\leq i<j\leq r}a_ia_j\right).
\]
\end{proposition}
\begin{proof} 
Since $P(L_{\slalg_{r+1}}(0))=\{0\}$, the conformal weight is not directly defined for the module $M^{0,0-\sum_{i=1}^ra_i\ga_i}$ if any $a_i\ne0$. However, by Theorem~\ref{thm: Li's result}, we see that $0^{(s)}=k\Lambda_s$ for each $1\leq s\leq r$, i.e., 
\[
M^{0,0-\sum_{i=1}^ra_i\ga_i}\cong M^{k\Lambda_s,k\Lambda_s-\sum_{i=1}^ra_i\ga_i}
\]
for each integer $1\leq s\leq r$. So, we have to find an integer $s$ so that $k\Lambda_s-\sum_{i=1}^ra_i\ga_i$ is a weight of $L_{\slalg_{r+1}}(k\Lambda_s)$, that is, we must verify that $(k\Lambda_s-\sum_{i=1}^ra_i\ga_i)\in P(L_{\slalg_{r+1}}(k\Lambda_s))$.
Recall that given weights $\ga,\mu\in\heisen^*$ of $\slalg_{r+1}$, the $\ga$ weight-string through $\mu$ is given by 
\[
L_{\slalg_{r+1}}(\mu+q\ga)\oplus L_{\slalg_{r+1}}(\mu+(q-1)\ga)\oplus\cdots L_{\slalg_{r+1}}(\mu)\oplus L_{\slalg_{r+1}}(\mu-\ga)\oplus\cdots\oplus L_{\slalg_{r+1}}(\mu-r\ga),
\]
where $r,q\in\zz_{\geq0}$ satisfying $\cyc{\mu,\ga}=r-q$. Letting $\mu=k\Lambda_s$ and $\ga=\ga_i$, we see that $\mu$ is the highest weight of the $\slalg_{r+1}$-module $L_{\slalg_{r+1}}(k\Lambda_s)$. This implies that $q=0$ and $r=\cyc{k\Lambda_s,\ga_i}=k\delta_{si}$. Therefore, the $\ga_s$ weight-string through $k\Lambda_s$ is $\bigoplus_{j=0}^{k}L_{\slalg_{r+1}}(k\Lambda_s-j\ga_s)$. Since $k\ga_s=0$ for the parafermions $K(\slalg_{r+1},k)$, we see that the index $j$ above lies in the set $\{0,1,2,\ldots,k-1\}$. This implies that 
\(
P(L_{\slalg_{r+1}}(k\Lambda_s))\supset \{k\Lambda_s-j\ga_s\mid 0\leq j\leq k-1\}.
\)
For each $1\leq j\leq k-1$, we know that $\cyc{k\Lambda_s-j\ga_s,\ga_i}=j$ if $|s-i|=1$ and zero otherwise. Therefore, 
\[
P(L_{\slalg_{r+1}}(k\Lambda_s))\supset
\{k\Lambda_s-a_s\ga_s-a_{s-1}\ga_{s-1}-a_{s+1}\ga_{s+1}\mid a_{s+1},a_{s-1}\leq a_{s}<k\}.
\]
Continuing this way, we find that 
\[
P(L_{\slalg_{r+1}}(k\Lambda_s))\supset
\{k\Lambda_s-\sum_{i=1}^ra_i\ga_i\mid a_1\leq\ldots\leq a_s,\quad  a_r\leq \ldots\leq a_s,\quad 0\leq a_i\leq k-1 \forall 1\leq i\leq r\}.
\]
Returning to the initial data $\{a_1,\ldots,a_r\}$, let $A=\max_{1\leq i\leq r}a_i$ and let $s=\min\{i\mid a_i=A\}$. Note that conformal weight of $M^{k\Lambda_s,k\Lambda_s-\sum_{i=1}^ra_i\ga_i}$ given in Lemma~\ref{lem: cw} is $Sym{(r-1)}$ invariant under the actions of 
\[
\sigma:(a_1,\ldots,a_{s-1},a_s,a_{s+1},\ldots,a_r)\mapsto (a_{\sigma(1)},\ldots,a_{\sigma(s-1)},a_s,a_{\sigma(s+1)},\ldots,a_{\sigma(r)}),
\]
fixing $s$ and permuting all the remaining indices. Therefore, we may assume that 
\[
a_1\leq\ldots\leq a_s,\quad\text{and}\quad  a_r\leq \ldots\leq a_s.
\]
This means the conformal weight of $M^{0,0-\sum_{i=1}^ra_i\ga_i}$ is equal to the conformal weight of $M^{k\Lambda_s,k\Lambda_s-\sum_{i=1}^ra_i\ga_i}$ with $s$ defined to be the smallest integer so that $a_s=\max_{1\leq i\leq r}a_i$. The conclusion follows.
\end{proof}
\subsubsection{An example: \texorpdfstring{$K(\slalg_2,k)$}{K(sl2,k)}}
We conclude with an explicit description for the modules of $K(\mathfrak{sl}_2,k)$, obtained by specializing the general framework for $K(\mathfrak{sl}_{r+1},k)$ discussed earlier in this section to the case $r=1$. These results will be used in Section~\ref{subsec: parafermions and sl2}. Related descriptions appeared in \cite{DLY09, ALY14, DW16}, formulated in a more specialized setting that directly exploits the structure of $\mathfrak{sl}_2$ and uses different notation (see Notation~\ref{notn: for sl2 parafermions}).

Let $r=1$. Let $\ga$ be the positive root, $\Lambda=\ga/2$ be the fundamental weight and $\Delta=\{\pm\ga\}$ be the root system, $2\rho=\sum_{\gamma\in\Delta_+}\gamma=\ga$ and the maximal root $\theta=\ga$. The root and weight lattices are $Q=\zz\ga$ and $P=\zz(\ga/2)$, respectively. The normalized Cartan-Killing form is defined so that $\cyc{\ga,\ga}=2$ and therefore $\cyc{\Lambda,\ga}=\cyc{\ga/2,\ga}=1$ and $\cyc{\Lambda,\Lambda}=\frac{1}{2}$. We see that \( P^k_+=\{a\Lambda\mid a\in\zz_{\geq0}, a\leq k\}\)
since $\cyc{a\Lambda,\theta}=a$. A priori, the modules are of the form $M^{a\Lambda,a\Lambda-b\ga},$ where $0\leq a\leq k$ and $b\in\zz$ are arbitrary integers. However, since $M^{\Lambda,\lambda}\cong M^{\Lambda,\lambda-k\gb}$ for every $\gb\in Q$, we see that the set of irreducible modules is 
\(\{M^{a\Lambda,a\Lambda-b\ga}\mid 0\leq a\leq k,0\leq b\leq k-1\}.
\)
Using Example~\ref{exmp: Li's sl2}, we have
\[
M^{a\Lambda,a\Lambda-b\ga}\cong M^{(k-a)\Lambda,k\Lambda+a\Lambda-b\ga} = M^{(k-a)\Lambda, (k-a)\Lambda-(b-a)\ga}.
\]
Since $b-a<k-a$, we see that the set of inequivalent irreducible modules is 
\[
\mathcal{W}_1:=\{M^{a\Lambda,a\Lambda-b\ga}\mid 0\leq b<a\leq k\}.
\]
Here, the subscript $1$ stands for $r=1$. Since $\cyc{a\Lambda,\ga}=a$, we see that $\{0,1,\ldots,a\}\subset P(L_{\slalg_2}(a\Lambda))$ and since $0\leq b<a$, the conformal weight is defined for each element in $\mathcal{W}_1$ and it is given as 
\begin{align*}
cw^{a\Lambda,a\Lambda-b\ga}
&=\frac{\cyc{a\Lambda,a\Lambda+\ga}}{2(k+2)}-\frac{\cyc{a\Lambda-b\ga,a\Lambda-b\ga}}{2k}
=\frac{1}{2k(k+2)}
\left(k(a+2ab-2b^2)-(a-2b)^2\right).
\end{align*}
Finally, the fusion rule is given by 
\begin{align*}
M^{a\Lambda,a\Lambda-b\ga}\boxtimes 
M^{a'\Lambda,a'\Lambda-b'\ga}
&=\sum_{c=0}^k N^c_{a,a'}M^{c\Lambda,a\Lambda+a'\Lambda-b\ga-b'\ga}
=\sum_{c=0}^k N^c_{a,a'} M^{c\Lambda, c\Lambda-\overline{\frac{1}{2}(2b+2b'-a-a'+c)}\ga},
\end{align*}
where $N^{c}_{a,a'}$ are the fusion rules for the irreducible $L_{\widehat{\slalg}_2}(k)$-modules and $\overline{a}$ is residue of $a$ modulo $k$.
\begin{notation}
\label{notn: for sl2 parafermions}
In order to be consistent with the literature \cite{DLY09,ALY14,DW16}, let us denote the $K(\slalg_2,k)$-module $M^{i\Lambda,i\Lambda-j\ga}$ by $M^{i,j}$. Then the set of inequivalent irreducible modules is $\mathcal{W}_1=\{M^{i,j}\mid 0\leq i,j\leq k\}/\cong$ with the isomorphism given by $M^{i,j}\cong M^{k-i,j-i}$, that is, 
\(
\mathcal{W}_1=\{M^{i,j}\mid 0\leq j<i\leq k\},
\)
and $|\mathcal{W}_1|=\frac{k(k+1)}{2}$.
This agrees with \cite[Theorem 4.4]{DLY09}, \cite[Theorem 8.2]{ALY14}. 
It follows that the conformal weight of $M^{i,j}$ is 
\begin{equation}
\label{eqn: cw sl2}
cw^{i,j}:=\frac{1}{2k(k+2)}
\left(k(i+2ij-2j^2)-(i-2j)^2\right),
\end{equation}
which agrees with the formula first established in \cite[Proposition 4.5]{DLY09}. 
Finally, the fusion rules simplify to
\begin{equation*}
M^{i,i'}\boxtimes_{K(\mathfrak{sl}_2,k)}M^{j,j'}
=\sum_lM^{l,\overline{\frac{1}{2}(2i'-i+2j'-j+l)}},
\end{equation*}
where $0\leq i,j\leq k$, $0\leq i',j'\leq k-1$, $i+j+l\in2\zz$ and $i+j+l\leq2k.$
Moreover, all modules $M^{l,\overline{\frac{1}{2}(2i'-i+2j'-j+l)}}$ in the summand are inequivalent irreducible modules. This corresponds to Theorem 4.2 in \cite{DW16}.
\end{notation}

\subsection{Positive divisors associated to \texorpdfstring{$\slalg_2$}{sl2} Parafermions}
\label{subsec: parafermions and sl2}
Let $k\geq1$ be any integer and let $K(\slalg_2,k)$ denote the parafermion vertex operator algebra at level $k$ associated to the simple Lie algebra $\slalg_2$. There are $\frac{k(k+1)}{2}$ irreducible inequivalent $K(\slalg_2,k)$-modules forming the set 
\(\mathcal{W}_1:=\{M^{i,j}\mid 0\leq j<i\leq k\}.\) For additional details, we direct the reader to Section~\ref{subsec: parafermions and modules sl}.

\subsubsection{Rank and degree Formulas}
In order to understand the positivity of the coinvariant divisors on $\overline{\mathrm{M}}_{0,n}$ associated with the representations above, we first need to understand the rank of the coinvariant vector bundles 
\(
\mathbb{V}_{0,4}(K(\slalg_2,k), M^{a,a'}\otimes M^{b,b'}\otimes M^{c,c'}\otimes M^{d,d'}),
\)
and the degree of their first Chern classes, for any four modules $M^{a, a'}, M^{b,b'}, M^{c,c'}, M^{d, d'}$ in $\mathcal{W}_1$. We address this in the subsequent two propositions. First, we set some notation and a lemma we need.
\begin{notation}
For ease of discussion, we will denote rank of the bundle $\mathbb{V}_{0,n}(K(\slalg_2,k), \bigotimes_{i=1}^n M^{a_i,a'_i})$ by $\mu_{a_1,\ldots,a_n}$ or $\mu(\bigotimes_{i=1}^n M^{a_i,a'_i})$ and the degree of its first Chern class by $d_{a_1,\ldots,a_n}$ or $d(\bigotimes_{i=1}^n M^{a_i,a'_i})$. 
\end{notation}
\begin{lemma}
Let $M^{a,a'}$ be an irreducible simple $K(\slalg_2,k)$-module. Then, its contragradient dual is $M^{a,a-a'}$. 
\end{lemma}
\begin{proof}
Since the VOA is rational, the dual of $M^{a,a'}$ is a another simple module $M^{b,b'}$ so that the dimension $\mu_{a,b,0}$ of the vector space $\mathbb{V}_{0,3}(M^{a,a'}\otimes M^{b,b'}\otimes M^{0,0})=1$.
By fusion rule, we see that  $\mu_{a,b,0}=1$ if and only if $|a-b|\leq 0\leq \min(a+b,2k-a-b)$ and $a+b\in 2\zz$ and $0=\frac{1}{2}(a+b-2a'-2b')\mod k$. Therefore, we must have $a=b$ and $a-a'-b'=0\mod k$, that is, $b=a$ and $b'=a-a'$. 
\end{proof}
\begin{corollary}
The module $M^{2a,a}$ is self-dual and the contragradient module of $M^{k,a}$ is $M^{k,k-a}$.
\end{corollary}
\begin{proposition}
\label{prop: rank formula sl2}
Let $M^{a,a'},M^{b,b'},M^{c,c'},M^{d,d'}\in \mathcal{W}_1$ be any four simple $K(\slalg_2,k)$-modules and let $s:=a+b+c+d$ and $s':=a'+b'+c'+d'$.
The rank of the vector bundle $\mathbb{V}_{0,4}(K(\slalg_2,k), M^{a,a'}\otimes M^{b,b'}\otimes M^{c,c'}\otimes M^{d,d'})$ is 
\[
\mu_{a,b,c,d}
=\begin{cases}
|A| &\text{ if } s=0\mod 2 \text{ and } s/2=s'\mod k,\\
|B| &\text{ if } s=k\mod 2 \text{ and }(s-k)/2=s'\mod k,\\
0   &\text{ else.}
\end{cases}
\]
Here $|A|$ and $|B|$ denotes size of the sets $A$ and $B$, respectively, which are defined to be
\begin{align*}
A&:=\{\max(b-a,d-c)\leq t\leq \min(a+b,2k-c-d)\mid t=a+b\mod2\},\text{ and }\\
B&:=\{\max(b-a,|k-c-d|)\leq t\leq \min(a+b,2k-a-b,k-d+c)\mid t=a+b\mod2\}.
\end{align*}
\end{proposition}
\begin{proof}
By factorization, we have 
\[
\mu_{a,b,c,d} = \sum_{M^{t,t'}}
\text{rank }\mathbb{V}_{0,3}(M^{a,a'}\otimes M^{b,b'}\otimes M^{t,t'})\mu(M^{c,c'}\otimes M^{d,d'}\otimes M^{t,t-t'}),
\]
where the sum is taken over the irreducible inequivalent modules $M^{t,t'}$ of the parafermions and $M^{t,t-t'}=(M^{t,t'})'$ by the Lemma above. Note that $\mu(M^{a,a'}\otimes M^{b,b'}\otimes M^{t,t'})=1$ if and only if $(t,t')$ is of the form where 
\[
b-a\leq t\leq\min(a+b,2k-a-b),\quad t+a+b=0\mod2,\quad t'=\frac{1}{2}(a+b-2a'-2b'+t)
\]
or its equivalent form (cf. Notation~\ref{notn: for sl2 parafermions})
\[
m_{k-t,k-t+t'}=m_{k-t,\frac{1}{2}(a+b-2a'-2b'-t)\mod k}
\]
with the same restrictions on $t$. Similarly, $\mu(m_{c,c'}\otimes m_{d,d'}\otimes m'_{s,s'})=1$ if and only if $m_{s,s'}$ is of the form 
\[
d-c\leq t\leq\min(c+d,2k-c-d),\quad s+c+d=0\mod2,\quad  s'=\frac{1}{2}(2c'+2d'-a-b+s).
\]
Here, we don't need the equivalence form of $m_{s,s'}$ given by $m_{k-s,k-s+s'}$ since we only need to  consider values of $t,t',s,s'$ so that $m_{t,t'}=m_{s,s'}$ or $m_{k-t,k-t+t'}=m_{s,s'}$. Considering these equations with $m_{k-s,k-s+s'}$ replacing $m_{s,s'}$ will result in double counting. For the first case, letting $m_{t,t'}=m_{s,s'}$, we see that $a+b+c=0\mod2$ based on the condition $t=a+b=c+d\mod 2$ and $\frac{1}{2}(a+d+c+d)=a'+b'+c'+d'\mod k$ based on the condition $t'=s'\mod k$. Finally, since $s=t$, we have $
\max(b-a,d-c)\leq t\leq\min(a+b,2k-c-d)$. Similarly, for the second case letting $m_{k-t,k-t+t'}=m_{s,s'}$, we see that $a+b+c+d-k=0\mod 2$, $\frac{1}{2}(a+b+c+d-k)=a'+b'+c'+d'\mod k$ and $\max(|c+d-k|,b-a)\leq t\leq\min(a+b,2k-a-b,k-d+c)$.
\end{proof}
Proposition~\ref{prop: rank formula sl2} does not give a sufficient condition for the non-triviality of the coinvariant vector bundles. The following example illustrates this, and Corollary~\ref{cor: non-trivial ranks for sl2} provides the precise if and only if condition.
\begin{example}
Let $k=3$ and consider the coinvariant vector bundles 
\[
\mathbb{V}_1:=\mathbb{V}_{0,4}(K(\slalg_2,k),M^{1,0}\otimes M^{3,1}\otimes (M^{3,2})^{\otimes 2})\text{ and }
\mathbb{V}_2:=\mathbb{V}_{0,4}(K(\slalg_2,k),M^{2,1}\otimes  (M^{3,2})^{\otimes 3}).
\]
The vector bundles $\mathbb{V}_1$ and $\mathbb{V}_2$ both have rank zero, since $A=B=\emptyset$.
\end{example}
\begin{corollary}
\label{cor: non-trivial ranks for sl2}
With the notation as in Proposition~\ref{prop: rank formula sl2}, $\mu_{a,b,c,d}\ne0$ if and only if one of the following holds:
\begin{enumerate}[i.]
\item $s=0\mod 2$, $s/2=s'\mod k$ and $A\ne\emptyset$
\item $s=k\mod 2$, $(s-k)/2=s'\mod k$ and $B\ne\emptyset$.
\end{enumerate}
\end{corollary}
Next, we describe the degree of coinvariant divisors on $\overline{\mathrm{M}}_{0,4}$ associated with the simple $K(\slalg_2,k)$-modules. 
\begin{proposition}
\label{prop: degree for sl2}
Let $k\geq1$ be an integer and let $i\in\{a,b,c,d\}$ so that $0\leq i'<i\leq k$ for each $i$ and $a\leq b\leq c\leq d$. Then, degree $d$ of the divisor $\mathbb{D}_{0,4}(K(\slalg_2; k); M^{a,a'}\otimes (M^{b,b'})'  \otimes (M^{c,c'})'\otimes (M^{d,d'})')$ is given in terms of its rank $\mu$ and the sum of the conformal weights $c_{\Sigma}:=\sum_{i} cw(M^{i,i'})$ as follows:
\begin{equation*}
 -d+\mu c_{\Sigma} = \begin{cases}
0&\text{ if }\frac{b+c+d-a}{2}\ne b'+c'+d'-a'\mod k,\\
\Lambda &\text{ else,}
 \end{cases}
\end{equation*}
where 
\[
\Lambda = \sum_{i\in\{b,c,d\}}\left(\sum_{\substack{m_i\leq t\leq M_i\\ (a+i)\in2\zz}}cw\left(M^{t,\ \overline{a'-i'+\frac{t-a+i}{2}}}\right)\right),
\quad 
m_i=\max(|i-a|,\ |\ga-\gb|) \text{ with }\{\ga,\gb\}=\{b,c,d\}\setminus\{i\},\text{ and}
\] 
\[
M_i=\min(a+i,\ b+c+d-i,\ 2k-a-i,\ 2k-b-c-d+i).
\]

\end{proposition}
The proof of this proposition is presented in Section~\ref{subsec: parafermions and proof of degree formula}.

\subsubsection{Positive coinvariant divisors}
The $F$-positivity of subrings of the fusion ring of the parafermions $K(\slalg_2,k)$ reduces to understanding conditions for the degree formula above to be non-negative. However, the highly non-trivial nature of the formula above creates a combinatorial obstruction to doing so in complete generality. The proof suggests that the non-trivial nature of the formula derives from understanding the modules that appear after factorization and their conformal weights. The combinatorial complexity simplifies if we consider subrings $\mathcal{T}(k)$ and $\mathcal{S}_1(k)$ of the fusion ring generated by simple modules of the form $M^{2a,a}$ and $M^{k,b}$, respectively, that is, 
\begin{equation*}
\mathcal{T}(k):=\{\sum_i z_iM^{2a_i,a_i}\mid z_i\in\zz \text{ is nonzero for finitely many }i \text{ and } a_i\in[0,k/2]\cap\zz \},\text{ and }
\end{equation*}
\begin{equation*}
\mathcal{S}_1(k):=\{\sum_i z_iM^{k,a_i}\mid z_i\in\zz \text{ is nonzero for finitely many }i\text{ and } a_i\in[0,k-1]\cap\zz \}.
\end{equation*}
\begin{theorem}
\label{thm: positive for sl2}
Let $k \geq 1$ be an integer. A coinvariant divisor $\mathbb{D}$ on $\overline{\mathrm{M}}_{0,n}$ whose associated simple modules lie in the subring $\mathcal{T}(k)$ is semi-ample. The same conclusion holds for coinvariant divisors arising from the subring $\mathcal{S}_1(k)$.
\end{theorem}
\begin{proof}
Recall that all coinvariant divisors associated to $L_{\hat{\slalg}_{2}}(k,0)$-representations are base-point free~\cite{Fak12,DG23}. We prove that the subring $\mathcal{T}(k)$ (resp., $\mathcal{S}_1(k)$) forms a proportional pairing with a subring of the fusion ring of $L_{\hat{\slalg}_2}(k,0)$ (resp., $L_{\hat{\slalg}_k}(1,0)$). The conclusion then follows from Theorem~\ref{thm: intro1.iii}, since every coinvariant divisor associated to a representation in $\mathcal{T}(k)$ (resp., $\mathcal{S}_1(k)$) can be expressed as a positive rational multiple of a base-point free divisor arising from the representation theory of the affine VOA $L_{\hat{\slalg}_2}(k,0)$ (resp., $L_{\hat{\slalg}_k}(1,0)$).

\begin{enumerate}[i.]
\item For $\mathcal{T}(k)$: For any two elements $M^{2a,a}$ and  $M^{2b,b}$, the fusion rule dictates $\mu(M^{2a,a}\otimes M^{2b,b}\otimes M^{t,t'})=1$ if and only if $|2a-2b|\leq t\leq \min(2a+2b,2k-2a-2b)$, $t+2a+2b\in2\zz$ and $t'=\frac{1}{2}(t)\mod k$, that is, $(t,t')=(2c,c)$ for some integer $|2a-2b|\leq 2c\leq \min(2a+2b,2k-2a-2b)$. The fusion rule for modules in the subring 
\[
\mathcal{Z}(k):=\{\sum_i z_iL_{\hat{\slalg}_{2}}(k,2a_i\Lambda)\mid z_i\in\zz \text{ is nonzero for finitely many }i \text{ and } a_i\in[0,k/2]\cap\zz \}
\]
is exactly the same. We have a ring isomorphism $f:\mathcal{T}(k)\to\mathcal{Z}(k)$ defined by $M^{2a,a}\mapsto L_{\hat{\slalg}_{2}}(k,2a\Lambda)$. Moreover, the ratio of the conformal weights is given by (see Equation~\ref{eqn: cw sl2} and \cite[Section 4]{Fak12}) 
\[
\frac{\text{cw} (M^{2a,a})}{\text{cw} (L_{\hat{\slalg}_2}(k,2a\Lambda))}
=\left(\frac{a+a^2}{k+2}\right)/\left(2a^2+2a\right)
=\frac{1}{2(k+2)},
\]
which is constant for any level $k$. Therefore, $\mathcal{T}(k)$ and $\mathcal{Z}(k)$ form a proportional pairing.
\item For $\mathcal{S}_1(k)$: Similarly, the fusion ring $\mathcal{R}(k)$ of the affine VOA $L_{\hat{\slalg}_{k}}(1,0)$
\[
\mathcal{R}(k):= \{\sum_i z_iL_{\hat{\slalg}_{k}}(1,\Lambda_i)\mid z_i\in\zz \text{ is nonzero for finitely many }i \}
\]
forms a pairing with $\mathcal{S}_1(k)$ via the isomorphism $f:M^{k,a}\mapsto L_{\hat{\slalg}_{k}}(1,\Lambda_a)$ with conformal ratio 
\[
\frac{cw(M^{k,a})}{cw(L_{\hat{\slalg}_{k}}(1,\Lambda_a))} 
=\left(\frac{2a(k+2)(k-a)}{2k(k+2)}\right)/\left(\frac{a(k-a)(k+1)}{k}\right)
=\frac{1}{k+1},
\] again a constant for any level $k$. Therefore, the desired conclusion holds.
\end{enumerate}
\end{proof}
\begin{remark}
\label{rem: affine vs parafermion}
From the above proof, we see that, in fact, the coinvariant divisor $\mathbb{D}_{0,n}(K(\slalg_2,k), \otimes_{i=1}^n M^{2a_i,a_i})$ is a positive rational multiple of the coinvariant divisor $\mathbb{D}_{0,n}(L_{\hat{\slalg}_2}(k,0),\otimes_{i=1}^n L_{\hat{\slalg}_2}(k,2a_i\Lambda))$, where $\Lambda$ is the fundamental weight of the $\slalg_2$ Lie algebra. Similarly, $\mathbb{D}_{0,n}(K(\slalg_2,k), \otimes_{i=1}^n M^{k,a_i})$ is a rational multiple of the coinvariant divisor $\mathbb{D}_{0,n}(L_{\hat{\slalg}_k}(1,0),\otimes_{i=1}^n L_{\hat{\slalg}_k}(1,\Lambda_i))$, where $\Lambda_1,\ldots,\Lambda_{k-1}$ are the fundamental weights of the Lie algebra $\slalg_{k}$.
\end{remark}
\begin{remark}
By \cite[Proposition 4.7]{Fak12} and Remark~\ref{rem: affine vs parafermion}, the bundle $\mathbb{V}_{0,n}(K(\slalg_2,k),\otimes_{i=1}^n M^{2a_i,a_i})$ is the pullback of a vector bundle in an suitable moduli space of weighted stable curves with marked points. Again by Remark~\ref{rem: affine vs parafermion}, the coinvariant divisors $\mathbb{D}_{0,n}(K(\slalg_2,k), {M^{k,a}}^{\otimes n})$ and morphisms they define are studied in \cite{NoahAngela}.
\end{remark}
\begin{remark}
\label{rem: para vs others}
The parafermion vertex operator algebras \( K(\mathfrak{sl}_2,k) \) offer the first known example of a vertex operator algebra whose fusion ring exhibits genuinely new behavior with respect to positivity. For instance, as shown in Example~\ref{exmp: fusion ring not F-poisitive}, the subring generated by the subset \( \mathcal{T}(3) \cup \mathcal{S}_1(3) \) is not even \( F \)-positive, highlighting the subtle structure of the representation theory of Parafermion VOAs. The positivity of coinvariant divisors has been previously studied for the following other VOAs: for affine VOAs in~\cite{Fak12}, for VOAs strongly generated in degree one in~\cite{DG23}, and for the discrete series Virasoro algebras \( \mathrm{Vir}_{2,2k+1} \) with \( k \geq 1 \) in~\cite{ChoiVirasoro}. In the first two cases, every coinvariant divisor associated to any \( n \)-tuple of representations in the fusion ring is base-point free. For \( \mathrm{Vir}_{2,2k+1} \), the dual of each coinvariant divisor is an \( F \)-divisor for all \( k \geq 1 \), and nef for \( k \leq 8 \). Thus, in all three cases, the entire fusion ring behaves well with respect to positivity, in contrast to the fusion ring of \( K(\mathfrak{sl}_2,k) \).
\end{remark}

\begin{question}
Given the relationship between their first Chern classes, is there relationship (such a morphism) between the coinvariant vector bundles $\mathbb{V}_{0,n}(L_{\hat{\slalg}_2}(k,0),\bigotimes_{i=1}^n L_{\hat{\slalg}_2}(k,2a_i\Lambda))$  and  $\mathbb{V}_{0,n}(K(\slalg_2,k),\bigotimes_{i=1}^n M^{2a_i,a_i})$? If there were a surjective morphism from the former to the latter, it would imply that  $\mathbb{V}_{0,n}(K(\slalg_2,k),\bigotimes_{i=1}^n M^{2a_i,a_i})$ is globally generated, since, as Fakhruddin showed \cite{Fak12}, there is a surjective morphism from a constant bundle on $\overline{\mathrm{M}}_{0,n}$ onto the bundle $\mathbb{V}_{0,n}(L_{\hat{\slalg}_2}(k,0),\bigotimes_{i=1}^n L_{\hat{\slalg}_2}(k,2a_i\Lambda))$.
\end{question}
\begin{corollary}
\label{cor: lambda twisted nef classes for sl2}
There is a rational number $t\in\mathbb{Q}_{\geq0}$ so that for all rational numbers $q'\geq t$, the divisor $t\lambda+\mathbb{D}$ on $\overline{\mathrm{M}}_{g,n}$ for any coinvariant divisor $\mathbb{D}$ with all simple modules either in $\mathcal{T}(k)$ or $\mathcal{S}_1(k)$ are nef.
\end{corollary}
\begin{proof}
The claim follows directly from Proposition~\ref{prop: lambda twist nef high higher genus} and Theorem~\ref{thm: positive for sl2}.
\end{proof}

\subsubsection{Characterization of non-trivial coinvariant divisors}
With positivity of subrings $\mathcal{T}(k)$ and $\mathcal{S}_1(k)$ established, it is natural to ask for conditions for non-triviality of the associated coinvariant divisors. If such a coinvariant divisor is not trivial, then it is either ample or it is external in the nef cone of $\overline{\mathrm{M}}_{0,n}$. We provide the non-triviality condition and list a large class of such divisors that lie on the boundary of the nef cone. As part of this computation, we first provide an answer for $n=4$. 

\begin{corollary}
\label{cor: sl2 para deg}
The coinvariant divisor corresponding to modules $M^{2b_1,b_1},\ldots, M^{2b_4,b_4}$ is non-trivial if and only if $b_1+\cdots+b_4>k$ and in that case the degree is a multiple of the rank $\mu(M^{2b_1,b_1}\otimes\cdots\otimes M^{2b_4,b_4})$:
\[
d(M^{2b_1,b_1}\otimes\cdots\otimes M^{2b_4,b_4}) = \frac{1}{2(k+2)}\mu(M^{2b_1,b_1}\otimes\cdots\otimes M^{2b_4,b_4})(-k+b_1+\cdots+b_4). 
\]
The coinvariant divisor corresponding to modules $M^{k,a},M^{k,b},M^{k,c},M^{k,d}$, with $1\leq a\leq b\leq c\leq d\leq k-1$, is non-trivial if and only if $a+b+c+d=2k$ and in this case, the degree is
\[
d(M^{k,a}\otimes M^{k,b}\otimes M^{k,c}\otimes M^{k,d})=\begin{cases}
a/(k+1)&\text{ if } b+c\leq a+d\\
(k-d)/(k+1)&\text{ if } b+c> a+d.
\end{cases}
\]
\end{corollary}
\begin{proof}
The degree formula is 
\begin{align*}
d(\bigotimes_{i=1}^4 M^{2b_i,b_i}) &= 
\mu(\bigotimes_{i=1}^4 M^{2b_i,b_i})\left(\sum_{i=1}^4cw(M^{2b_i,b_i})\right)\\
&-\sum_{t=0}^{\lfloor k/2\rfloor}cw(M^{2t,t})\mu(M^{2b_1,b_1}\otimes M^{2b_2,b_2}\otimes M^{2t,t})\mu(M^{2b_3,b_1}\otimes M^{2b_4,b_4}\otimes M^{2t,t})-(2\leftrightarrow3)-(2\leftrightarrow4),
\end{align*}
where 
$(2\leftrightarrow3)$ is the sum over $t$ with position of $2$ and $3$ replaced and same for $(2\leftrightarrow4)$. By proof of Theorem~\ref{thm: intro1.iii},
\[
d(\bigotimes_{i=1}^4 M^{2b_i,b_i})
=\frac{1}{2(k+2)}d(\bigotimes_{i=1}^4 L_{\hat{\slalg}_2}(k,2b_i\Lambda)),
\]
which gives the desired formula by the degree formula in \cite[Proposition 4.2]{Fak12} and \cite{alexeev_ranks_sl2} for $L_{\hat{\slalg}_2}(k,0)$-representations. For the second case, again the result follows from the theorem above and \cite[Lemma 5.1]{Fak12}.
\end{proof}
\begin{proposition}
\label{prop: non-triviality M2a,a}
The coinvariant divisor $\mathbb{D}_{0,n}(K(\slalg_2,k),\bigotimes_{i=1}^n M^{2a_i,a_i})$ corresponding to any $n$-many simple $K(\slalg_2,k)$-modules $M^{2a_i,a_i}$ is non-trivial if and only if $\sum_{i=1}^n a_i>k$.  
\end{proposition}
\begin{proof}
The necessity condition follows from the proof of Theorem~\ref{thm: intro1.iii} and Lemma 4.1 in \cite{Fak12}. We prove the other direction below. The case for $n=4$ is true by the corollary above. Without loss of generality, assume that $1\leq a_1\leq \cdots\leq a_n\leq k/2$ and assume that $a_1+\cdots+a_n>k$. We break into two cases.

\textbf{Case 1: $a_1+a_2\leq l/2$.} Assume that the statement holds for all $n<N$ for some integer $N>4$. By factorization 
\[
\mathbb{D}|_{\overline{M}_{0,N-1}}=\sum_{t=a_2-a_1}^{a_1+a_2}c_1\mathbb{V}_{0,n}(K(\mathfrak{sl}_2,l);\otimes_{p=3}^nm_{2a_p,p}\otimes m_{2t,t}).
\]
Note that for $t=a_1+a_2$, by assumption $c_1\mathbb{V}_{0,n}(K(\mathfrak{sl}_2,l);\otimes_{p=3}^nm_{2a_p,p}\otimes m_{2t,t})\ne0$ as $a_3+\cdots+a_N+t=a_1+\cdots+a_N>l$. This implies $\mathbb{D}|_{\overline{M}_{0,N-1}}\ne0$, which in turn gives $\mathbb{D}\ne0$ on $\overline{\mathrm{M}}_{0,n}$. 

\textbf{Case 2: $a_1+a_2>l/2$.} 

\textbf{Sub-case 2.1: ($n=5$)} We first prove that the statement holds for $n=5$ by showing that the divisor $\mathbb{D}=\mathbb{D}_{0,5}(K(\mathfrak{sl}_2,l);\otimes_{p=1}^5m_{2a_p,p})$ non-trivially intersects the F-curve $F_{\{1,2\},3,4,5}$ given $a_1+\cdots+a_5>l$. We know that 
\[
\mathbb{D}|_{F_{\{1,2\},3,4,5}} = \sum_{t=a_2-a_1}^{l-a_1-a_2}c_1\mathbb{V}_{0,4}(K(\mathfrak{sl}_2,l); m_{2t,t}\otimes\bigotimes_{p=3}^5m_{2a_p,a_p}). 
\]
If $l-a_1-a_2<a_2-a_1$, then $a_2>l/2$, a contradiction. Therefore, the sum is not empty. 
We now show that $\deg c_1\mathbb{V}_{0,5}(K(\mathfrak{sl}_2,l);\otimes_{p=3}^5m_{2a_p,a_p}\otimes m_{2(l-a_1-a_2),l-a_1-a_2)})\ne0$. This is broken into two cases:
\begin{itemize}
\item If $a_1+a_2\leq \frac{3}{4}l$, then $l-a_1-a_2\geq l/4$, Note that $a_i\geq\frac{a_1+a_2}{2}>l/4$ for each $i\in\{3,4,5\}$, since $a_i\geq a_{i-1}$. Therefore, 
\[
a_3+a_4+a_5+(l-a_1-a_2)>3l/4+l/4=l.
\]
\item If $a_1+a_2>\frac{3}{4}l$, then 
\[
a_3+a_4+a_5+(l-a_1-a_2)\geq a_3+a_4+a_5\geq3\frac{a_1+a_2}{2}>9l/8>l.
\]
where we have used the fact that $0\leq a_2-a_1\leq l-a_1-a_2$. 
\end{itemize}

\textbf{Sub-case 2.2: ($n>5$)}
Having proven the result for $n=4$ and $n=5$, assume that the statement holds for all $n<N$ for some integer $N>5$. Since $a_p\geq \frac{a_1+a_2}{2}>l/4$ for each $p\geq3$, we have 
\[
a_3+\cdots+a_N+(l-a_1-a_2)> \frac{4}{4}l+(l-a_1-a_2)=2l-(a_1+a_2)\geq 2l-l=l,
\]
where we have used the fact that $a_p\leq l/2$ for each $1\leq p\leq N$. Then, 
\[
\mathbb{D}|_{\overline{M}_{0,N-1}} = \mathbb{D}_{0,N-1}(K(\mathfrak{sl}_2,l);\otimes_{p=3}^N m_{2a_p,a_p}\otimes m_{2(l-a_1-a_2),l-a_1-a_2})+\text{others}\ne0.
\]
This completes the proof.
\end{proof}

\begin{proposition}
\label{prop: non-triviality Mk,a}
The coinvariant divisor $\mathbb{D}_{0,n}(K(\slalg_2,k),\bigotimes_{i=1}^n M^{k,a_i})$ corresponding to simple $K(\slalg_2,k)$-modules $M^{k,a_1},\ldots,M^{k,a_n}\in \mathcal{W}_1$ is non-trivial if and only if there is a partition $I_1\cup I_2\cup I_3\cup I_4=\{1,\ldots,n\}$ so that 
\[
\overline{\sum_{i\in I_1}a_i}+\overline{\sum_{i\in I_2}a_i}+\overline{\sum_{i\in I_3}a_i}+\overline{\sum_{i\in I_4}a_i}=2k,\text{ and }\min\{\overline{\sum_{i\in I_p}a_i}\mid 1\leq p\leq4\}\ne0,
\] 
where $\overline{b}$ denotes the residue of $b$ modulo $k$. 
\end{proposition}
\begin{proof}
To prove non-triviality, it suffices to find an $F$-curve $F_{I_1, I_2, I_3, I_4}$ such that the divisor does not contract. For such an $F$-curve, the intersection with the coinvariant divisor is given by 
\[
\mathbb{D}_{0,n}(K(\slalg_2,k),\bigotimes_{i=1}^n M^{k,a_i})|_{F_{I_1,I_2,I_3,I_4}}
=\sum_{t_{1},\ldots,t_4}\deg\mathbb{D}_{0,4} (\bigotimes_{p=1}^4 M^{k,t_p})\prod_{p=1}^4\mu(\bigotimes_{i\in I_p}M^{k,a_i}\otimes M^{k,k-t_p}),
\]
where the sum is taken over all tuples $(t_1,\ldots,t_4)\in([0,k]\cap\zz)^{\oplus 4}$
Now, $\mu(\bigotimes_{i\in I_p}M^{k,a_i}\otimes (M^{k,t_p})'=1$ if $\sum a_i+k-t_p=0\mod k$ and zero otherwise. This implies only one summand remains, and we have 
\[
\mathbb{D}_{0,n}(K(\slalg_2,k),\bigotimes_{i=1}^n M^{k,a_i})|_{F_{I_1,I_2,I_3,I_4}}
=\deg \mathbb{D}_{0,4}(\bigotimes_{p=1}^4 M^{k,\overline{\sum_{i\in I_p}a_i}}),
\]
and the rest follows from Corollary~\ref{cor: sl2 para deg}. In particular, the second condition is needed as otherwise, by propagation of vacua, the bundle $\mathbb{V}_{0,4}(\bigotimes_{p=1}^4 M^{k,\overline{\sum_{i\in I_p}a_i}})$ pulls back to a point, and hence its first Chern class is zero. 
\end{proof}
\subsubsection{Positive coinvariant divisors that contract an F-curve}
The next two propositions provide a large class of coinvariant divisors with representations in the subrings $\mathcal{T}(k)$ and $\mathcal{S}_1(k)$ that are nef but not ample. These divisors lie on the boundary of the nef cone of \(\overline{\mathrm{M}}_{0,n}\).
\begin{proposition}
\label{prop: extremal Mk,a}
The coinvariant divisor $\mathbb{D}:=\mathbb{D}_{0,n}(K(\slalg_2,k), \otimes_{i=1}^n M^{k,a_i})$ on $\overline{\mathrm{M}}_{0,n}$ contracts at least one $F$-curve if the minimum of the set $\{a_i+a_j+a_k\mid 1\leq i<j<k\leq n\}$ is strictly less than $k+1$. 
\end{proposition}
\begin{proof}
For a tuple of pairwise-distinct integers $(x,y,z)$ so that $a_x+a_y+a_z\leq k$, we have $\mathbb{D}\cdot F_{\{x\},\{y\},\{z\},[n]-\{x,y,z\}}=0$ since $a_x+a_y+a_z+\overline{ \sum_{i\in [n]-\{x,y,z\}}a_i}\leq k+(k-1)<2k$. The conclusion then follows from Proposition~\ref{prop: non-triviality Mk,a}.
\end{proof}
\begin{proposition}
Let $M^{2a_i,a_i}$ be any $K(\slalg_2,k)$-module for each integer $1\leq i\leq n$. If $\sum_{i\in T}a_i\leq(k+1)/2$ for any subset $T\s[n]$ of size $2\leq |T|\leq n-2$, the divisor $\mathbb{D}_{0,n}(K(\slalg_2,k),\bigotimes_{i=1}^n M^{2a_i,a_i})$ contracts at-least one $F$-curve.   
\end{proposition} 
\begin{proof}
Intersection of $\mathbb{D}_{0,n}(K(\slalg_2,k),\bigotimes_{i=1}^n M^{2a_i,a_i})$ with an $F$-curve $F_{I_1,I_2,I_3,I_4}$ so that $I_1\cup I_2\cup I_3= T$ is
\[
\sum_{(t_1,\ldots,t_4)}
d(\bigotimes_{p=1}^4 M^{2t_p,t_p}) \prod_{p=1}^4\mu(\bigotimes_{i\in I_p}M^{2a_i,a_i}\otimes M^{2t_p,t_p}).
\]
We now claim that $t_1+t_2+t_3\leq \sum_{i\in T}a_i$ for all tuples $(t_1,\ldots,t_4)$ with possibly nonzero summands. 

Assuming the claim holds, for such a tuple $(t_1,\ldots,t_4)$, we have 
\[
t_1+\cdots+t_4\leq \sum_{i\in T}a_i +t_4\leq \frac{k+1}{2}+\frac{k}{2}<k+1,
\]
where we have used the fact that $t_4\leq k/2$. But then, by Corollary~\ref{cor: sl2 para deg}, we see that the degree contribution in such summands is zero. Therefore, $\mathbb{D}_{0,n}(K(\slalg_2,k),\bigotimes_{i=1}^n M^{2a_i,a_i})$ intersects all such $F$-curves trivially. It remains to prove the claim. Let $0\leq p\leq k/2$ and $n$ be any integers. Applying the fusion rule repeatedly, we get
\begin{align*}
\mu(\bigotimes_{i=1}^n M^{2a_i,a_i}\otimes ^{2p,p})
&=\sum_{b_1=|a_{n-1}-a_n|}^{a_{n-1}+a_n}\mu(\otimes_{i=1}^{n-2} m_{2a_i,a_i}\otimes  m_{2b_1,b_1}\otimes m_{2p,p})\\
&=\sum_{b_1=|a_{n-1}-a_n|}^{a_{n-1}+a_n}\ \sum_{b_2=|a_{n-2}-b_1|}^{a_{n-2}+b_1}\ \cdots 
\sum_{b_{n-3}=|a_3-b_{n-4}|}^{a_3+b_{n-4}}\
\sum_{b_{n-2}=|a_2-b_{n-3}|}^{a_2+b_{n-3}}\  \mu(M^{2a_1,a_1},M^{2b_{n-2},b_{n-2}},M^{2p,p}).
\end{align*}
Note that fusion rule dictates $b_1\leq \min(a_{n-1}+a_n,2k-a_{n-1}-a_n)$. But since $a_{n-1}+a_n\leq a_1+\cdots+a_n\leq k$, we see that $a_{n-1}+a_n\leq 2k-a_{n-1}-a_n$. Similarly, for $b_2$, we have $b_2\leq \min(a_{n-2}+b_1,2k-a_{n-2}-b_1) = a_{n-2}+b_1$ since $a_{n-2}+b_1\leq a_{n-2}+a_{n-1}+a_{n}\leq k$. In particular, we see that $b_{n-2}\leq a_2+a_3+\ldots+a_{n}$. Therefore, again by fusion rules, it is necessary that $p\leq a_1+\ldots+a_n$ for $\mu(a_1,b_{n-2},p)$ to be $1$.

\end{proof}
Next, we describe non-ample symmetric divisors with representations in the subring $\mathcal{T}(k)$. 
\begin{notation}
We introduce new notation for the symmetric cases. Since we are only considering symmetric divisors built from representations on $\mathcal{T}(k)$, the only modules under consideration are of type $M^{2a, a}$. We denote degree of a divisor $\mathbb{D}_{0,4}(K(\slalg_2,k),\bigotimes_{i=1}^4M^{2a_i,a_i)})$ by $d_{a_1,\ldots,a_4}$. Similarly, the rank of the bundle $\mathbb{D}_{0,n}(K(\slalg_2,k),\bigotimes_{i=1}^n M^{2a_i,a_i})$ is denoted by $r_{a_1,\ldots,a_n}$. For special cases, the rank of $\mathbb{D}_{0,n+1}(K(\slalg_2,k),(M^{2a,a})^{\otimes n}\otimes M^{2p,p})$ is denoted by $r_{a^n,p}$.
\end{notation}
\begin{proposition}
\label{prop: extremal M2a,a symmetric}
Let $1\leq a\leq k/2$ be any integer. Then the symmetric divisor $\mathbb{D}_{0,n}(K(\slalg_2,k),(M^{2a,a})^{\otimes n})$ contracts an $F$-curve if any of the following holds:
\begin{enumerate}[i.]
\item $n\geq9$ is even and $a\leq\min(k/4, k/n, 3k/(n+8))$.
\item $n\geq9$ is odd and $a\leq\min(k/4, 2k/(n+4), 3k/(n+8))$.
\item $n$ and $k$ are even and $a=k/2$.
\end{enumerate}
\end{proposition}
\begin{proof}
Recall that the set $\{F_{1,1,i}\mid 1\leq i\leq g\}$, with $n=2g+2$ or $n=2g+3$, forms a basis of $N_1(\overline{\mathrm{M}}_{0,n}/Sym_n)=N_1(\tilde{\mathrm{M}}_{0,n})$ (cf. \cite[Corollary 2.2]{AGS14}). So, it suffices to study intersection of the $F$-curves $F_{1,1,i}$ with the symmetric divisor $\mathbb{D}_a:=\mathbb{D}_{0,n}(K(\slalg_2,k),(M^{2a,a})^{\otimes n})$. We will determine conditions on $a$ and $n$ so that each summand in
\[
\mathbb{D}_a\cdot F_{1,1,i}=\sum_{x,y=0}^{k/2}d_{a,a,x,y}r_{a^i,x}r_{a^{n-2-i},y}
\]
vanishes.
Let $a\leq k/4$.
By Lemma~\ref{lem: rank for symmetric cases}, 
the summand corresponding to $i=t$ is zero if and only if $2a+(at+2a-\eta)+(a(n-2-t)+2a-\tilde{\eta})\leq k$
, that is, if and only if $a\leq (k+\eta+\tilde{\eta})/(n+4)$. First note that $k+\eta+\tilde{\eta}>0$ 
and secondly, $(k+\eta+\tilde{\eta})/(n+4)>k/4$ if and only if $(n-8)k<4(\eta+\tilde{\eta})$. Therefore, for integers $5\leq n\leq 8$, the divisor $\mathbb{D}_a$ contracts all $F$-curves $F_{1,1,i}$. This also directly follows from Proposition~\ref{prop: non-triviality M2a,a} since $n a\leq 8a\leq 2k$. Let $n\geq9$. 
\begin{enumerate}[i.]
\item Let $n$ and $t$ both be even. Then, $\eta=\tilde{\eta}=2a$. Then, $a\leq (k+\eta+\tilde{\eta})/(n+4)$ if and only if $a\leq k/n$. Therefore, $\mathbb{D}_{a}\cdot F_{1,1,t}=0$ if and only if $a\leq k/n$. 
\item Let $n$ be even and $t$ odd. Then, $\eta=k-2a$ and $\tilde{\eta}=k-2a$. Then, $a\leq (k+\eta+\tilde{\eta})/(n+4)$ if and only if $a\leq 3k/(n+8)$. Therefore, $\mathbb{D}_{a}\cdot F_{1,1,t}=0$ if and only if $a\leq 3k/(n+8)$.
\item Let $n$ be odd and $t$ even. As in case (ii), we arrive at the same conclusion.
\item Let $n$ and $t$ both be odd. Then, $\eta = k-2a$ and $\tilde{\eta}=2a$. Then, $a\leq (k+\eta+\tilde{\eta})/(n+4)$ if and only if $a\leq 2k/(n+4)$. Therefore, $\mathbb{D}_{a}\cdot F_{1,1,t}=0$ if and only if $a\leq 2k/(n+4)$.
\end{enumerate}
Let $k/4<a\leq k/2$ be any integer. Then the summand corresponding to $i=t$ is zero if and only if $2a+(2a-\eta+t(k/2-a))+(2a-\tilde{\xi}+(n-2-t)(k/2-a))\leq k$, that is, if and only if $(8-n)a+(n-4)k/2-\xi-\tilde{\xi}\leq 0$ if and only if $a\geq k/2+(4a-\xi-\tilde{\xi})/(n-4)$. 
\begin{enumerate}[i.]
\item Let $n$ and $t$ both be even. Then, $\xi=\tilde{\xi}=2a$ and $k/2+(4a-\xi-\tilde{\xi})/(n-4)=k/2$. Therefore, $\mathbb{D}_{a}\cdot F_{1,1,t}=0$ if and only if $a=k/2$. 
\item Let $n$ be even and $t$ odd. Then, $\xi=k/2$ and $\tilde{\xi}=k/2$ and $k/2+(4a-\xi-\tilde{\xi})/(n-4)>k/2$. Therefore, $\mathbb{D}_{a}\cdot F_{1,1,t}> 0$ for all admissible $a$. 
\item Let $n$ be odd and $t$ even. As in case (ii), we arrive at the same conclusion.
\item Let $n$ and $t$ both be odd. Then, $\xi=k/2$ and $\tilde{\xi}=2a$ and $k/2+(4a-\xi-\tilde{\xi})/(n-4)=k/2+(2a-k/2)/(n-4)>k/2$. Therefore, $\mathbb{D}_{a}\cdot F_{1,1,t}> 0$ for all admissible $a$. 
\end{enumerate}
\end{proof}
\begin{lemma}
\label{lem: rank for symmetric cases}
Let $1\leq a\leq k/2$ and $1\leq t\leq n-3$ be any integers. Let $r_{a^t,x}$ and $r_{a^{n-2-t},y}$ denote ranks of the bundles $\mathbb{V}_{0,t+1}(K(\slalg_2,k),(M^{2a,a})^{\otimes t}\otimes M^{2x,x})$ and $\mathbb{V}_{0,n-1-t}(K(\slalg_2,k),(M^{2a,a})^{\otimes (n-2-t)}\otimes M^{2y,y})$, respectively. 
\begin{equation*}
r_{a^t,x} \ne 0 \text{ if and only if } 
\begin{cases}
\eta-at\leq x\leq at+2a-\eta &\text{ if } a\leq k/4,\\ 
\xi-t(\frac{k}{2}-a)\leq 2a-\xi+t(\frac{k}{2}-a)  &\text{ if } a> k/4,\text{ and}
\end{cases}
\end{equation*}
\begin{equation*}
r_{a^{n-2-t},y}\ne 0 \text{ if and only if } 
\begin{cases}
\tilde{\eta}-(n-2-t)a\leq y\leq a(n-2-t)+2a-\tilde{\eta} &\text{ if } a\leq k/4\\ 
\tilde{\xi}-(n-2-t)(\frac{k}{2}-a)\leq y\leq 2a\tilde{\xi}+(n-2-t)(\frac{k}{2}-a) &\text{ if } a> k/4.
\end{cases}
\end{equation*}

The integers $\eta$, $\tilde{\eta}$, $\xi$ and $\tilde{\xi}$ are defined to be 
\[
\eta = \begin{cases}
k-2a &\text{ if $t$ is odd}\\ 
2a &\text{ if $t$ is even},
\end{cases}
\qquad
\tilde{\eta} = \begin{cases}
k-2a &\text{ if $(n-2-t)$ is odd}\\ 
2a &\text{ if $(n-2-t)$ is even},
\end{cases}
\]
\[
\xi = \begin{cases}
k/2 &\text{ if $t$ is odd}\\ 
2a &\text{ if $t$ is even},
\end{cases}
\qquad
\tilde{\xi} = \begin{cases}
k/2 &\text{ if $(n-2-t)$ is odd}\\ 
2a &\text{ if $(n-2-t)$ is even}.
\end{cases}
\]
\end{lemma}
\begin{proof}
Let \( a \leq \frac{k}{4} \). By the non-triviality criterion of Rasmussen and Walton \cite{sl2rankPhysics, Swinarski_sl2}, we have
\[
r_{a^t,x} \ne 0 \quad \text{if and only if} \quad
at + x \leq \min_{\substack{0 \leq i \leq t+1 \\ i + t \equiv 0 \ (\mathrm{mod}\ 2)}} \left( \frac{t - i}{2}k + 2\left((i - 1)a + \min(a, x)\right) \right).
\]
The minimum is attained at the largest admissible value of \( i \). When \( x \leq a \), the inequality simplifies to \( \eta - at \leq x \); when \( x > a \), it becomes \( x \leq at + 2a - \eta \). Hence, the conclusion follows.
The remaining cases can be treated using the same approach. We omit these calculations for brevity.
\end{proof}

Finally, we list coinvariant divisors on \(\overline{M}_{0,5}\) represented in \(\mathcal{T}(k)\) that contract an \(F\)-curve. A similar list can be obtained rom Lemmas 4.5 and 4.6 in \cite{Swinarski_sl2} using Remark~\ref{rem: affine vs parafermion}.
\begin{corollary}
Let $1\leq\lambda_1\leq\cdots\leq\lambda_5\leq k/2$ be integers. The divisor $\mathbb{D}_{\lambda}$ is non-trivial and nef if $\lambda_1+\cdots+\lambda_5>k$. Moreover, $\mathbb{D}_{\lambda}$ contracts an $F$-curve if any of the following holds.
\begin{enumerate}[i.]
\item $\lambda_1+\lambda_2+\lambda_3\leq \frac{k+1}{2}$.
\item There are integers $1\leq i<j\leq 5$ so that $\lambda_i+\lambda_j\leq k/2$ and $\lambda_1+\cdots+\lambda_5\ne k+1$. 
\item There are integers $1\leq i<j\leq 5$ so that $\lambda_i+\lambda_j> k/2$ and $\lambda_1+\cdots+\lambda_5>1+2\lambda_i+2\lambda_j$. 
\end{enumerate}
In all other cases, $\mathbb{D}_{\lambda}$ is ample. 
\end{corollary}
\begin{proof}
The proof follows directly from the following lemma.
\end{proof}
\begin{lemma}

Let $1\leq\lambda_1\leq\cdots\leq\lambda_5\leq k/2$ be integers, and let $\{\lambda_1,\ldots,\lambda_5\}=\{\mu_1,\mu_2\mid \mu_1\leq \mu_2\}\sqcup\{\nu_1,\nu_2,\nu_3\mid \nu_1\leq\nu_2\leq\nu_3\}$ be a partition. Then, 
\[
\mathbb{D}_{\lambda_{\bullet}}\cdot F_{\{\mu_1,\mu_2\},\nu_1,\nu_2,\nu_3}\ne0\Leftrightarrow
\begin{cases}
\lambda_1+\cdots+\lambda_5=k+1&\text{ if }\mu_1+\mu_2\leq k/2\\
\nu_1+\nu_2+\nu_3-\mu_1-\mu_2\leq 1&\text{ if }\mu_1+\mu_2> k/2.
\end{cases}
\]
\end{lemma}
\begin{proof}
The intersection $\mathbb{D}_{\lambda_{\bullet}}\cdot F_{\{\mu_1,\mu_2\},\nu_1,\nu_2,\nu_3}$ is given by the sum $\sum_{x=0}^{\lfloor k/2\rfloor}d_{\nu_{\bullet},x}r_{\mu_{\bullet},x}$. Now, $d_{\nu_{\bullet},x}\ne0$ if and only if $\nu_1+\nu_2+\nu_3+x\geq k+1$ and $r_{\mu_{\bullet},x}\ne0$ if and only if $\mu_2-\mu_1\leq x\leq \min(\mu_1+\mu_2,k-\mu_1-\mu_2)$.
\begin{enumerate}[i.]
\item Let $\mu_1+\mu_2\leq k/2$. Then, $\min(\mu_1+\mu_2,k-\mu_1-\mu_2)=\mu_1+\mu_2$ and the intersection is non-trivial if and only if $\mu_1+\mu_2\leq k+1-\nu_1-\nu_2-\nu_3$, that is, if and only if $\lambda_1+\cdots+\lambda_5=k+1$. 
\item Let $\mu_2+\mu_2>k/2$. Then, $\min(\mu_1+\mu_2,k-\mu_1-\mu_2)=k-\mu_1-\mu_2$ and the intersection is non-trivial if and only if $\nu_1+\nu_2+\nu_3-\mu_1-\mu_2\leq 1$.
\end{enumerate}
\end{proof}
\subsection{Positive divisors associated with \texorpdfstring{$\slalg_{r+1}$}{sl3} Parafermions}
\label{subsec: parafermions and sl3}
For this section, let $r\geq1$ be any integer and let $\mathcal{S}_r(k)$ be the subring of the fusion ring of $K(\slalg_{r+1},k)$ generated by the simple modules comprising the set
\(
\{M^{0,0-a_{\bullet}}\mid a_{\bullet}=(a_1,\ldots,a_r)\in\{0,1,\ldots,k-1\}^{\oplus r}\}.
\)
Section~\ref{subsec: Ar rank and cw} introduces the rank formulas for the coinvariant bundles and the conformal weights of simple modules in $\mathcal{S}_r(k)$. We study the positivity of $\mathcal{S}_2(k)$ in Section~\ref{subsec: A2 positivity}, and extend the analysis to all $\mathcal{S}_r(k)$ with $r \geq 3$ in Section~\ref{subsec: Ar positivity}.

\subsubsection{Rank formula and conformal weights}
\label{subsec: Ar rank and cw}
\begin{proposition}
\label{prop: rank sl3}
Given any tuple $a_{\bullet},b_{\bullet},c_{\bullet},d_{\bullet}\in \{0,1,\ldots,k-1\}^{\oplus r}$, we have
\[
\text{rank }\mathbb{V}_{0,4}(K(\slalg_{r+1},k);M^{0,0-a_{\bullet}}\otimes M^{0,0-b_{\bullet}}\otimes M^{0,0-c_{\bullet}}\otimes M^{0,0-d_{\bullet}})
=\begin{cases}
1&\text{ if } a_{\bullet}+b_{\bullet}+c_{\bullet}+d_{\bullet}=x_{\bullet}k\\
0&\text{ else,}
\end{cases}
\]
where $x_{\bullet}=(x_1,\dots,x_r)\in\{1,2,3\}^{\oplus r}$.
\end{proposition}
\begin{proof}
The fusion rule for simple modules in $\mathcal{S}_r(k)$ is given as follows:
\[
M^{0,0-a_{\bullet}}\boxtimes M^{0,0-b_{\bullet}} = M^{0,0-\overline{a_{\bullet}+b_{\bullet}}},
\]
where $\overline{a_{\bullet}}:=(a_1\mod k,\ldots,a_r\mod k)$. Therefore, the dual of $M^{0,0-a_{\bullet}}$ is $M^{0,0-(k-a_{\bullet})}$ and the rank is 
\[
\sum_{t_{\bullet}}\mu(M^{0,0-a_{\bullet}}\otimes M^{0,0-b_{\bullet}}\otimes M^{0,0-t_{\bullet}})\mu(M^{0,0-a_{\bullet}}\otimes M^{0,0-b_{\bullet}}\otimes M^{0,0-(k-t_{\bullet})}).
\]
For the sum to be nonzero, we need $t_i=a_i+b_i\mod k$ and $k-t_i=c_i+d_i\mod k$, that is, $a_i+b_i+c_i+d_i=0\mod k$ for each $1\leq i\leq r$. Given $a_i+b_i+c_i+d_i=0\mod k$ for each $i$, there is a unique $t_{\bullet}$ for which the corresponding summand is nonzero. Hence, the rank is either zero or one.
\end{proof}
To understand the positivity of the subrings $\mathcal{S}_r$, we first need to understand the conformal weights. 
\begin{proposition}
\label{prop: slr cw}
Given $a_{\bullet}=(a_1,\ldots,a_r)\in\{0,1,\ldots,k-1\}^{\oplus r}$, conformal weight of $M^{0,0-a_{\bullet}}$ is 
\[
a_s-\frac{1}{k}\left(\sum_{i=1}^r a^2_i-\sum_{1\leq i<j\leq r}a_ia_j\right).
\]
with $s:=\min\{1\leq i\leq r\mid a_i=\max_{1\leq j\leq r}a_j\}$.
\end{proposition}
\begin{proof}
Since $M^{0,0-a_{\bullet}}\cong M^{k\Lambda_s,k\Lambda_s-a_{\bullet}}$ with the integer $s$ defined as above, it follows from Proposition~\ref{prop: slr cw}.
\end{proof}
\begin{remark}\label{rem: conformal weight}
The conformal weight is $Sym{(n-1)}$-invariant fixing $s$. Moreover, if there are two indices $s$ and $s'$ so that $a_{s}=a_{s'}$, then 
\[
a_s-\frac{1}{k}\left(\sum_{i=1}^r a^2_i-\sum_{1\leq i<j\leq r}a_ia_j\right)=a_{s'}-\frac{1}{k}\left(\sum_{i=1}^r a^2_i-\sum_{1\leq i<j\leq r}a_ia_j\right).
\]
Moreover, given any tuple $a_{\bullet}=(a_1,\ldots,a_r)$, denoting $\sigma\in S_r$ so that $b_i:=a_{\sigma(j)}$ with $b_1\leq\cdots\leq b_r$, we have 
\[
\text{ conformal weight of }M^{0,0-a_{\bullet}}
=\text{ conformal weight of }M^{0,0-b_{\bullet}},
\]
as we do not record the index of $a_{\bullet}$ which achieves the maximum $\max_{1\leq i\leq r}a_i$ but only the maximum itself for $a_s$ and the rest of the formula is $Sym(r)$-invariant.

For any $a_{\bullet},b_{\bullet},c_{\bullet},d_{\bullet}\in \{0,1,\ldots,k-1\}^{\oplus r}$, we can assume each of the tuple are a non-decreasing order of integers. Moreover, we may assume that $a_1\geq b_1\geq c_1\geq d_1$ while calculating degree of $c_1\mathbb{V}_{0,4}(K(\slalg_{r+1},k);M^{0,0-a_{\bullet}}\otimes M^{0,0-b_{\bullet}}\otimes M^{0,0-c_{\bullet}}\otimes M^{0,0-d_{\bullet}})$ because of the obvious symmetry of tensor products.

Finally, if $a_1=0$ for the tuple $a_{\bullet}=(a_1,\ldots,a_r)$, then the conformal weight of the $K(\slalg_{r+1},k)$-module $M^{0,0-a_{\bullet}}$
\[
(\max_{1\leq t\leq r}a_t)-\frac{1}{k}\left(\sum_{i=1}^r a^2_i-\sum_{1\leq i<j\leq r}a_ia_j\right)
=(\max_{1\leq t\leq r}a_t)-\frac{1}{k}\left(\sum_{i=2}^r a^2_i-\sum_{2\leq i<j\leq r}a_ia_j\right)
\]
is equal to the conformal weight of the $K(\slalg_{r},k)$-module $M^{0,0-b_{\bullet}}$ with $b_{\bullet}=(a_2,a_3,\ldots,a_{r})$. 
\end{remark}

\begin{corollary}
\label{cor: max cw for sl3}
Let $r\geq 2$ and $ k\geq 1$ be any integers. The simple module in $\mathcal{S}_2^{sim}(k)$ with maximal conformal weight is $M^{0,0-(\lfloor 2k/3\rfloor,\lfloor k/3\rfloor)}$ and for $\mathcal{S}_{r}^{sim}(k)$ with $r\geq3$ is $M^{0,0-(k-1,\ldots,k-1)}$. The conformal weight of simple modules in $\mathcal{S}_2(k)$ (resp., $\mathcal{S}_{r}(k)$ for all $r\geq3$) is less than or equal to $k/3$ (resp., $\frac{1}{k}(k-1)(k+(k-1)r(r-2))$ ) and all non-trivial simple modules have conformal weight greater than or equal to $(k-1)/k$. 
\end{corollary}
\begin{proof}
This result follows from the proposition above by doing a simple analysis. In particular, the nonzero minimum is achieved at $(\delta_{1j},\delta_{2,j},\ldots,\delta_{rj})$ for any fixed $1\leq j\leq r$. Similarly, for $r=2$, the maximum is achieved at $(2k/3,k/3)$ and for all $r\geq3$ at $(k-1,\ldots,k-1)$.
\end{proof}
\subsubsection{\texorpdfstring{$F$}{F}-positive and positive subrings in the \texorpdfstring{$\slalg_3$}{sl\_3} Parafermion fusion ring}
\label{subsec: A2 positivity}
The reliance on the maximality of the tuple \(a_{\bullet}\) renders the analysis of positivity in the subrings \(\mathcal{S}_r\) highly sensitive to the parameter \(r\). Additionally, the exponential growth of \(|\mathcal{W}_r|\) with respect to \(r\) at any fixed level significantly contributes to the combinatorial complexity (see Remark~\ref{rem: number of modules grow exp}). In this paper, we demonstrate that the subring \(\mathcal{S}_2(k)\) is \(F\)-positive for levels \(k \leq 10\), and moreover, positive for \(k \leq 5\).
\begin{proposition}
\label{prop: F-positive for sl3}
The subring $\mathcal{S}_2(k)$ is $F$-positive for any integer $k\leq 10$.  
\end{proposition}
\begin{proof}
We wish to prove that the divisor $\mathbb{D}$ on $\overline{\mathrm{M}}_{0,4}$ corresponding to simple modules $M^{0,0-(a_1,a_2)}$, $M^{0,0-(b_1,b_2)}$, $M^{0,0-(c_1,c_2)}$, and $M^{0,0-(d_1,d_2)}$ in $\mathcal{S}_2(k)$ has non-negative degree. First, we may assume that $a_1+b_1+c_1+d_1=x_1k$ and $a_2+b_2+c_2+d_2=x_2k$ for some integers $0\leq x_1,x_2\leq 3$ by Proposition~\ref{prop: rank sl3}. If $x_1+x_2=1$, then one of them must be zero, and the question reduces to the $F$-positivity of $\mathcal{S}_1(k)$ by Remark~\ref{rem: conformal weight} and the answer is affirmative (cf. Theorem~\ref{thm: positive for sl2}). The case $x_1=x_2=1$ is discussed in Lemma~\ref{lem: sl3 x1=x2=1}. If $x_1\ne x_2$, we may assume that $x_1\geq x_2$ since $c^{0,0-(a_1,a_2)}=c^{0,0-(a_2,a_1)}$ by Remark~\ref{rem: conformal weight} and twitching $(a_i,b_i)$ to $(b_i,a_i)$ for all four modules does not change the fusion rules. 
The remaining cases are checked via a computer program for level $k\leq 10$ (see \cite{Parafermionsgithub}). 
\end{proof}

\begin{lemma}
\label{lem: sl3 x1=x2=1}
Let $k\geq 1$ be any integer.
Let $M^{0,0-(y_1,y_2)}$, for $y\in\{a,b,c,d\}$, be simple modules in $\mathcal{S}_2(k)$ so that $a_1+b_1+c_1+d_1=a_2+b_2+c_2+d_2= k$. The corresponding coinvariant divisor on $\overline{\mathrm{M}}_{0,4}$ is either trivial or ample.
\end{lemma}
\begin{proof}
Consider simple modules $M^{0,0-(a_1,a_2)}$, $M^{0,0-(b_1,b_2)}$, $M^{0,0-(c_1,c_2)}$ and $M^{0,0-(b_1,b_2)}$ of the $K(\slalg_{3},k)$ parafermions. By Proposition~\ref{prop: rank sl3}, the degree of the associated coinvariant divisor $\mathbb{D}$ on $\overline{\mathrm{M}}_{0,4}$ is given by 
\begin{align}
\deg\mathbb{D} &= c^{0,0-(a_1,a_2)}+c^{0,0-(b_1,b_2)}+c^{0,0-(c_1,c_2)}+c^{0,0-(d_1,d_2)}\nonumber\\
&\qquad-c^{0,0-(\overline{a_1+b_1},\overline{a_2+b_2})}-c^{0,0-(\overline{a_1+c_1},\overline{a_2+c_2})}-c^{0,0-(\overline{a_1+d_1},\overline{a_2+d_2})},
\end{align}
where $\overline{a}$ denotes the residue of $a$ modulo $k$. We do a case-by-case analysis to show that the degree is non-negative.
Explicitly, we show that the degree is strictly positive if and only if one of the following two conditions holds:  
(1) \(\alpha_1 + \beta_1 = k\) and \(\gamma_2 + \delta_2 = k\) for some \(\alpha, \beta, \gamma, \delta \in \{a, b, c, d\}\); or  
(2) \(y_1 \neq y_2\) for all \(y \in \{a, b, c, d\}\), with \(\operatorname{sign}(a_1 - a_2) = \operatorname{sign}(y_1 - y_2)\) for exactly one \(y \in \{b, c, d\}\), and \(\operatorname{sign}(a_1 - a_2) = -\operatorname{sign}(z_1 - z_2)\) for all \(z \in \{b, c, d\} \setminus \{y\}\). In all other cases, the degree is zero. We define \(x_i := a_i + b_i + c_i + d_i\) for \(i = 1, 2\).

\begin{enumerate}[(\text{Case }1)]
    \item $\ga_1+\gb_1=k$ and $\gamma_2+\delta_2=k$ for some $\ga,\gb,\gamma,\delta\in\{a,b,c,d\}$: If $\{\ga,\gb,\gamma,\delta\}\ne\{a,b,c,d\}$, then we get the trivial divisor. To see why, assume that $d\notin\{\ga,\gb,\gamma,\delta\}$ and then $M^{0,0-(d_1,d_2)}=M^{0,0-(0,0)}$. So, it is safe to assume that $a_1+b_1=k=c_2+d_2$ with all four integers non-negative. Then the degree is:
    \begin{align*}
    \deg\mathbb{D} &= a_1-\frac{a_1^2}{k}+b_1-\frac{b_1^2}{k}+c_2-\frac{c_2^2}{k}+d_2-\frac{d_2^2}{k}-c^{0,0-(0,0)}-c^{0,0-(a_1,c_2)}-c^{0,0-(a_1,d_2)}\\
    &=k+a_1-\max(a_1,c_2)-\max(a_1,d_2)\\
    &=\begin{cases}
    b_1&\text{ if } c_2,d_2\leq a_1\\
    a_1 &\text{ if } c_2,d_2\geq a_1\\
    \min(c_2,d_2)&\text{ else.}
    \end{cases}
    \end{align*}
\item Assume that $y_i+z_i<k$ for all off-diagonal terms $(y,z)\in\{a,b,c,d\}^{\times 2}$ and all $i=1,2$: The degree is:
    \begin{align*}
    \deg\mathbb{D} 
    &=(a_1+a_2)+\max(a_1,a_2)+\max(b_1,b_2)+\max(c_1,c_2)+\max(d_1,d_2)\\ &-\max(a_1+b_1,a_2+b_2)-\max(a_1+c_1,a_2+c_2)-\max(a_1+d_1,a_2+d_2).
    \end{align*}
    We work on sub-cases:
    \begin{itemize}
    \item ($y_1=y_2=y$ for all $y\in\{a,b,c,d\}$) The degree is $2a+k-3a-(k-a)=0$.
    \item ($y_1=y_2$ for three elements in $\{a,b,c,d\}$) Since $x_1=x_2$, this case reduces to the previous one. 
    \item ($y_1=y_2$ for any two elements in $\{a,b,c,d\}$) By symmetry, we may assume that $c_1=c_2=c$, $d_1=d_2=d$, $a_1>a_2$ and $b_1<b_2$. Then, the degree is $a_1+a_2+a_1+b_2+c+d-\max(a_1+b_1,a_2+b_2)-a_1-c-a_1-d=a_2+b_2-\max(a_1+b_1,a_2+b_2)=(k-c-d)-\max(k-c-d,k-c-d)=0.$
    \item ($y_1=y_2$ for exactly one element in $\{a,b,c,d\}$) Assume that $d_1=d_2=d$. We have cases:
    \begin{itemize}
        \item $a_1>a_2,b_1>b_2,c_1<c_2$: The degree is $a_1+a_2+a_1+b_1+c_2+d-(a_1+b_1)-\max(a_1+c_1,a_2+c_2)-(a_1+d)=a_2+c_2-\max(a_1+c_1,a_2+c_2)$. Note that $c_2-c_1=(a_1-a_2)+(b_1+b_2)>(a_1-a_2)$ and therefore, $\max(a_1+c_1,a_2+c_2)=a_2+c_2$ and the degree is zero. 
        \item Any other case reduces to the case above by symmetry. For example, the case $a_1>a_2, c_1>c_2, b_1<b_2$ reduces to the previous case since any coinvariant divisor $\mathbb{D}_{0,4}(V,M^{1}\otimes\cdots\otimes M^4)$ is invariant under the action of $Sym(4)$ on the indices of modules $M^i$. 
        \item The reverse case of $a_1<a_2,b_1<b_2,c_1>c_2$ reduces to the first case since $c^{0-(a_1,a_2)}=c^{0-(a_2,a_1)}$ and the fusion rule remains the same if we switch the tuples for all four modules in consideration. 
    \end{itemize}
    \item ($y_1\ne y_2$ for all $y$ in $\{a,b,c,d\}$) We again break into cases:
    \begin{itemize}
        \item ($a_1>a_2,b_1>b_2,c_1<c_2,d_1<d_2$): 
        The degree is 
        \[
        \deg\mathbb{D}=\begin{cases}
        b_1-b_2 &\text{ if } a_1+c_1\geq a_2+c_2,a_1+d_1\geq a_2+d_2\\
        a_1-a_2 &\text{ if } a_1+c_1\leq a_2+c_2,a_1+d_1\leq a_2+d_2\\
        c_2-c_1 &\text{ if } a_1+c_1\geq a_2+c_2,a_1+d_1\leq a_2+d_2\\
        d_2-d_1 &\text{ if } a_1+c_1\leq a_2+c_2,a_1+d_1\geq a_2+d_2.
        \end{cases}
        \]
        \item ($a_1>a_2,b_1>b_2,c_1>c_2,d_2<d_2$): 
        The degree is $a_2+d_2-\max(a_1+d_1,a_2+d_2)$.
        Since $d_2-d_1=(a_1-a_2)+(b_1-b_2)+(c_1-c_2)<a_1-a_2$, the degree is zero. 
        \item Symmetries available to us reduce all cases to the one above. For example, the case $(a_1>a_2,b_1<b_2,c_1<c_2,d_1<d_2)$ reduces to the first case above by first switching $(y_1,y_2)$ to  $(y_2,y_1)$ for all $y\in\{a,b,c,d\}$ and then applying $(1,4)\in Sym(4)$ to the divisor. 
    \end{itemize}
    \end{itemize}
    \end{enumerate}
\end{proof}
\begin{corollary}
\label{cor: sl3 maximal cw nef result}
Let $1\leq k\leq 10$ be any integer. The symmetric coinvariant divisor corresponding to the simple $K(\slalg_3,k)$-module $M^{0,0-(\lfloor 2k/3\rfloor,\lfloor k/3\rfloor)}$ is nef. 
\end{corollary}
\begin{proof}
Since the module $M^{0,0-(\lfloor 2k/3\rfloor,\lfloor k/3\rfloor)}$ has maximal conformal weight, the result follows from Theorem~\ref{thm: intro1.ii}.
\end{proof}
\begin{theorem}
\label{thm: positive for sl3}
The subring $\mathcal{S}_2(k)$ is positive for any positive integer $k\leq 5$.  
\end{theorem}
\begin{proof}
If $k=1$, then the only module we need to consider is $M^{0,0}$, and the divisor is trivial as it trivially intersects all F-curves. 
For any integer $k\geq2$, consider $n$ many modules $M^{0,0-(a_i,b_i)}$, for integers $1\leq i\leq n$, with $0\leq a_i,b_i\leq k-1$ so that $\sum_i a_i = \sum_i b_i=0\mod k$. By Theorem~\ref{thm: POV}, we can assume that none of the $n$ modules are trivial. 

Then, $\mathbb{D}:=\mathbb{D}_{0,n}(K(\slalg_3,k);\otimes M^{0,0-(a_i,b_i)})$ can written as 
\[
\mathbb{D}=\sum_{j=1}^nc^{0,0-(a_j,b_j)}\psi_j-\sum_{\substack{I\s[n]\\2\leq |I|\leq n/2}} b_{0,I}\delta_{0,I},\text{ where}
\] 
\begin{align*}
b_{0,I} 
&= \sum_{W\in \mathcal{S}_2(k)} c^W\ \text{rank}\mathbb{V}_{0,|I|+1}(K(\slalg_3,k);\otimes_{i\in I} M^{0,0-(a_i,b_i)}\otimes W)\ \text{rank}\mathbb{V}_{0,|I^c|+1}(K(\slalg_3,k);\otimes_{i\notin I} M^{0,0-(a_i,b_i)}\otimes W')\\
&\leq f_{\max} \text{rank}\mathbb{V}_{0,n}(K(\slalg_3,k);\otimes_{i\in [n]} M^{0,0-(a_i,b_i)}) = f_{\max}.
\end{align*}
Note that 
$\sum_{j=1}^nc^{0,0-(a_j,b_j)}\psi_j \geq f_{\min}\psi.
$
Since $K:=K_{\overline{\mathrm{M}}_{0,n}}=\psi-2\Delta$, for any $c\in\qq_{\geq0}$
\begin{align*}
\frac{1}{f_{\min}}(\mathbb{D}-cK)&\geq (1-\frac{c}{f_{\min}})\psi +\frac{2c-f_{\max}}{f_{\min}}\Delta.
\end{align*}
In order to conclude that $\mathbb{D}$ is nef from Theorem~\ref{thm: intro1.i}, it remains to find some $c$ so that $\frac{f_{\max}}{2}\leq c\leq f_{\min}$. From Proposition~\ref{prop: slr cw}, we know that $\frac{k-1}{k}=f_{\min}\leq c^{0,-(a_i,b_i)}\leq f_{\max}\leq \frac{k}{3}$. Therefore, we need $\frac{k^2}{3(k-1)}\leq 2$, which is true for $k\leq 4$. For $k=5$, an explicit calculation yields $f_{\max}=8/5$ and $f_{\min}=4/5$, that is, $f_{\max}/f_{\min}=2$.
\end{proof}
\begin{remark}
In Proposition~\ref{prop: non F-positive}
, we show that $\mathbb{D}^{r,1}_{0,km}:=\mathbb{D}_{0,km}(K(\slalg_{r+1},k),{M^{0,0-(1,\ldots,1)}}^{\otimes km})$ intersects $F$-curves $F_{1,1,kt+k-3}$ strictly negatively for all integers $r\geq3$, $k\geq4$ and $0\leq t\leq m-1$. Given similar methods should work for $r=2$, one may ask why are divisors $\mathbb{D}^{2,1}_{0,4m}$ and $\mathbb{D}^{2,1}_{0,5m}$ nef as indicated by the Theorem above. A simple calculation shows that for any integers $\epsilon\in\{1,\ldots,k-3\}$, $k,m\geq1$ and $0\leq t\leq m-1$, the intersection number
\[
\mathbb{D}^{2,1}_{0,km}\cdot F_{1,1,kt+\epsilon}
= 2 cw(1)+cw(\epsilon)+cw(k-2-\epsilon)-cw(2)-2 cw(\epsilon+1),
\]
where $cw(a):=c^{0,0-(a,a)}=\frac{a(k-a)}{k}$, is equal to zero and the proof of Proposition~\ref{prop: non F-positive} does not apply for $r=2$. 
\end{remark}
\begin{remark}
As seen above, the combinatorial complexity owing to conformal weights limits our understanding of positivity of the subring $\mathcal{S}_2(k)$ for general $k$. However, there is a subring  $\mathcal{S}'_2(k)$ of $\mathcal{S}_2(k)$ that is propositional to $\mathcal{S}_1(k)$ and therefore, all coinvariant divisors with representations in  $\mathcal{S}'_2(k)$ are again nef by Theorems~\ref{thm: intro1.iii} and~\ref{thm: positive for sl2}. The subring $\mathcal{S}'_2(k)$ is defined as 
\[
\mathcal{S}'_2(k):=\langle M^{0,0-(a,a)}\mid 0\leq a\leq k-1\rangle,
\]
and the isomorphism of rings $f:\mathcal{S}'_2(k)\to \mathcal{S}_1(k): M^{0,0-(a,a)}\mapsto M^{k,a}$ gives proportionality since $cw(M^{0,0-(a,a)})=a-\frac{a^2}{k}=cw(M^{k,a})$. One checks that $M^{0,0-a}\cong M^{k,a}$. However, the generalized map $f:\mathcal{S}'_{r}(k)\to\mathcal{S}_{r-1}(k):M^{0,0-(a,\ldots,a)}\mapsto M^{0,0-(a,\ldots,a)}$ is an injection but does not satisfy proportionality, since 
\[
cw(M^{0,0-(a,\ldots,a)})-
cw(f(M^{0,0-(a,\ldots,a)}))
=a+\frac{a^2r}{2k}(r-3)-a+\frac{a^2(r-1)}{2k}(r-4)=\frac{2a^2}{k}(r-2).
\]
\end{remark}
\begin{question}
The remark shows that all coinvariant divisors with representations in a subring of $\mathcal{S}_2(k)$ are nef for all level $k\geq 1$. Is it true that $\mathcal{S}_2(k)$ is positive for all level $k\geq 6$?    
\end{question}
\begin{corollary}
\label{cor: lambda twisted nef classes for sl3} Let $1\leq k\leq 5$ and $g,n\geq0$ be any integers. There is a rational number $q\geq0$ so that the divisor $(q'\lambda+\mathbb{D})$ is nef for all $q'\in\qq_{\geq q}$ and any coinvariant divisor $\mathbb{D}$ on $\overline{\mathrm{M}}_{g,n}$ with modules in $\mathcal{S}_2(k)$.
\end{corollary}
\begin{proof}
It follows directly from Proposition~\ref{prop: lambda twist nef high higher genus} and Theorem~\ref{thm: positive for sl3}.
\end{proof}

\subsubsection{Symmetric coinvariant divisors associated to \texorpdfstring{$\slalg_3$}{sl\_3} Parafermions}
\label{sec: symm sl3 para}
Given that the subrings $\mathcal{S}_2(k)$ are positive for integers $1\leq k\leq 5$, it is natural to ask if the coinvariant divisors they define are ample. It is difficult to study such notions for an arbitrary coinvariant divisor owing to complexity of the conformal weights. In this subsection, we answer this question of all symmetric coinvariant divisors $\mathbb{D}_{0,n}(K(\slalg_3,k), M^{\otimes n})$ on $\tilde{M}_{0,n}=\overline{M}_{0,n}/Sym_n$, where $M$ is any simple module in $\mathcal{S}_{2}(k)$. The general approach is to intersect these divisors with $F$-curves $F_{1,1,i}$ for all integers $1\leq i\leq g$, with $n=2g+2$ or $n=2g+3$, as they form a basis of $N_1(\tilde{M}_{0,n})$.
\begin{proposition}
The symmetric divisor $\mathbb{D}_{0,n}(K(\slalg_3,2),M^{\otimes n})$ for a simple module $M\in\mathcal{S}_2(2)$ is trivial if and only if $M$ is the trivial module or $n$ is odd. Otherwise, then the intersection of the divisor with $F$-curves $F_{1,1,i}$ is
\[
\mathbb{D}_{0,n}(K(\slalg_3,2),M^{\otimes n})\cdot F_{1,1,i}=\begin{cases}
    0&\text{ if $i$ is odd}\\
    2&\text{ if $i$ is even}.
\end{cases}
\]
\end{proposition}
\begin{proof}
The subring $\mathcal{S}_2(2)$ is generated by simple modules $M_0:=M^{0,0-(0,0)}, M_1:=M^{0,0-(1,0)}, M_2:=M^{0,0-(0,1)}$, $M_3:=M^{0,0-(1,1)}$, where the conformal weight of $M_0$ is zero and the other three have conformal weight equal to $1/2$. Also, all four simple modules are self dual. Note that the divisor $\mathbb{D}_{0,n}(M_0^{\otimes n})$ is trivial, for any $n\geq1$. Let $M_a$ be any of the other three simple modules. The rank of the bundle $\mathbb{V}_{0,n}(K(\slalg_3,2),M_a^{\otimes n})$ is given by
\[
\mu(M_a^{\otimes n})=\sum_{M_b}\mu(M_a^{\otimes(n-2)}\otimes M_b)\mu(M_a^{\otimes 2}\otimes M_b)=\mu(M^{\otimes(n-2)})=\mu(M^{\otimes(n\mod 2)})=\begin{cases}
1 & \text{$n$ is even}\\
0 & \text{$n$ is odd}.
\end{cases}.
\]
Let $n=2g+2$ be an even integer and let $1\leq i\leq g$ be written as $i=2t+\epsilon$, where $t\geq0$ and $\epsilon\in\{0,1\}$ are integers.
\begin{align*}
\mathbb{D}_{0,n}(K(\slalg_3,2),M_a^{\otimes n})\cdot F_{1,1,2t+\epsilon} 
&=\sum_{x,y}d(M_a^{\otimes2}\otimes M_x\otimes M_y)\mu(M_a^{\otimes 2t+\epsilon}\otimes M_x)\mu(M_a^{\otimes(2g-\epsilon)}\otimes M_y)\\
&=\sum_{x,y}d(M_a^{\otimes2}\otimes M_x\otimes M_y)\mu(M_a^{\otimes \epsilon}\otimes M_x)\mu(M_a^{\otimes\epsilon}\otimes M_y).
\end{align*}
If $\epsilon=0$, then $\mathbb{D}_{0,n}(K(\slalg_3,2),M_a^{\otimes n})\cdot F_{1,1,2t+\epsilon} = d(M_a^{\otimes 2}\otimes M_0^{\otimes 2})=d(M_a^{\otimes 2})=0.$ If $\epsilon=1$, then $\mathbb{D}_{0,n}(K(\slalg_3,2),M_a^{\otimes n})\cdot F_{1,1,2t+\epsilon} = d(M_a^{\otimes 4})=4(1/2)-3(0)=2.$
\end{proof}
\begin{proposition}
The symmetric coinvariant divisor $\mathbb{D}:=\mathbb{D}_{0,n}(K(\slalg_{3},3), M^{\otimes n})$ associated to a non-trivial simple module $M:=M^{0,0-(a,b)}$ is trivial unless $3$ divides $n$. If $3\mid n$, for any integer $0\leq t\leq (2/3)n$  
\[
\mathbb{D}\cdot F_{1,1,3t+\epsilon} = \begin{cases}
3\max(a,b)+ab-a^2-b^2&\text{ if $\epsilon=2$},\\
0&\text{ if $\epsilon=0,1$ }.
\end{cases}
\]\end{proposition}
\begin{proof}
Given a non-trivial simple module $M:=M^{(a,b)}$ in $\mathcal{S}_2(3)$, for a choice of integers $m,t\geq0$ and $\delta,\epsilon\in\{0,1,2\}$,
\[
\mathbb{D}_{0,3m+\delta}(K(\slalg_3,3),M^{\otimes (3m+\delta)})\cdot F_{1,1,3t+\epsilon}
=d({M^{(a,b)}}^{\otimes2}\otimes M^{(\overline{\epsilon a},\overline{\epsilon a})}\otimes M^{((\overline{\delta-\epsilon-2)a},\overline{\delta-\epsilon-2)b}))})
\]
is equal to zero if $\delta\ne0$ or $\epsilon\ne2$. For $\epsilon=2$, the intersection number equals to $(3\cdot cw(M))$. 
\end{proof}
\begin{proposition}
For a simple module $M\in\mathcal{S}_2(4)$, the symmetric divisor $\mathbb{D}:=\mathbb{D}_{0,n}(K(\slalg_3,4),M^{\otimes n})$ if (a) $M$ is the trivial module, (b) $M$ has conformal weight $3/4$, or (c) $n$ is odd. Assume $n$ is even. Then $\mathbb{D}$ is ample if $M$ is self-dual and $M$ is not the trivial module. Finally, $M$ is a simple module with conformal weight $5/4$, then 
\[
\mathbb{D}_{0,n}({M}^{\otimes n})\cdot F_{1,1,i}=\begin{cases}
0 &\text{ if $i$ is even}\\
>0 &\text{ if $i$ is odd}.
\end{cases}
\]
\end{proposition}
\begin{proof}
The subring $\mathcal{S}_2(4)$ is generated by simple modules $M^{(a,b)}:=M^{0,0-(a,b)}$, where $0\leq a,b\leq 3$, partitioned into four subsets based on their conformal weight.  
\[
\mathcal{S}_2(4)^{sim}=c^0\sqcup c^{3/4}\sqcup c^{5/4}\sqcup c^{1}, \text{ with }c^0=\{M^{(0,0)}\},\quad c^{1}=\{M^{(2,0)},M^{(2,2)}\},
\]
\[
c^{3/4}=\{M^{(1,0)}, M^{(3,0)},M^{(1,1)},M^{(3,3)}\}, \quad c^{5/4}=\{M^{(1,2)}, M^{(3,2)},M^{(1,2)},M^{(3,1)},M^{(2,3)},M^{(2,1)}\},
\]
where the conformal weight of any module in $c^x$ equals $x$. If $M^{(a_1,b_1)}$ is self dual then, 
\[
\mu:=\mu(M^{(a_1,b_1)}\otimes M^{(a_1,b_1)}\otimes M^{(a_2,b_2)}\otimes M^{(a_3,b_3)})
\]
is $1$ if and only if (a) $M^{(a_2,b_2)}$ and $M^{(a_3,b_3)}$  are dual to each other, or (b) both modules $M^{(a_2,b_2)}$ and $M^{(a_3,b_3)}$ are self-dual. If $M^{(a_1,b_1)}$ is not self-dual, then $\mu=1$ if and only if $a_2+a_3=2\mod 4$ and $(b_2+b_3)$ is even. Similarly, one calculates that $d:=d(M^{(a_1,b_1)}\otimes M^{(a_1,b_1)}\otimes M^{(a_2,b_2)}\otimes M^{(a_3,b_3)})$ is non-zero if any of the following holds:
\begin{itemize}
\item $M^{(a_1,b_1)}$ is self-dual,
\item $M^{(a_1,b_1)}$ has conformal weight $5/4$, or
\item $M^{(a_1,b_1)}$ has conformal weight $3/4$ and $(a_2,b_2)\ne(a_1,b_1)$.
\end{itemize}
Note that $\mathbb{D}_{0,n}({M^{(0,0)}}^{\otimes n})$ is trivial. If $M^{(a,b)}$ is self-dual and not the trivial module, then 
\[
\mathbb{D}_{0,n}({M^{(a,b)}}^{\otimes n})\cdot F_{1,1,i}=\begin{cases}
0 &\text{ if $n$ is odd}\\
>0 &\text{ if $n$ is even},
\end{cases}
\]
and if $M^{(a,b)}$ is not self-dual, then $\mathbb{D}_{0,n}({M^{(a,b)}}^{\otimes n})$ is trivial for any odd integer $n$ and if $n$ is even, then 
\[
\mathbb{D}_{0,n}({M^{(a,b)}}^{\otimes n})\cdot F_{1,1,i}=\begin{cases}
0 &\text{ if $M^{a,b}\in c^{3/4}$}\\
0 &\text{ if $M^{a,b}\in c^{5/4}$ and $i$ is even}\\
>0 &\text{ if $M^{a,b}\in c^{5/4}$ and $i$ is odd}.
\end{cases}
\]
This completes the proof.
\end{proof}
\begin{proposition}
A symmetric coinvariant divisor $\mathbb{D}^M:=\mathbb{D}_{0,n}(K(\slalg_3,5),M^{\otimes n})$, with a simple module $M\in\mathcal{S}_2(5)\setminus\{0\}$, is trivial if $5$ does not divide $n$. If $n=5m$ for some integer $m\geq1$, then $\mathbb{D}^M$ contracts all $F$-curves other than $F_{1,1,5t+4}$ for all integers $0\leq t\leq m-1$ if the conformal weight of $M$ is equal to $4/5$. All other symmetric divisors contract only curves of type $F_{1,1,5t+3}$ for all integers $0\leq t\leq m-1$.
\end{proposition}
\begin{proof}
Given a non-trivial simple module $M^{(a_1,b_1)}$ in $\mathcal{S}_2(5)$, the intersection of the symmetric divisor $\mathbb{D}:=\mathbb{D}_{0,n}(K(\slalg_3,k),{M^{(a_1,b_1)}}^{\otimes n})$ with the $F$-curves $F_{1,1,i}$ is 
\begin{align*}
\mathbb{D}\cdot F_{1,1,i}
&=\sum d({M^{(a_1,b_1)}}^{\otimes 2}\otimes M^{(a_2,b_2)}\otimes M^{(a_3,b_3)})\mu({M^{(a_1,b_1)}}^{\otimes i}\otimes {M^{(a_2,b_2)}}')\mu({M^{(a_1,b_1)}}^{\otimes (n-2-i)}\otimes {M^{(a_3,b_3)}}').
\end{align*}
Note that the smallest integer $x\geq1$ for which $xa_1=xb_1=0$ modulo $5$ is $x=5$, and therefore, 
\[
\mathbb{D}\cdot F_{1,1,i} 
= d({M^{(a_1,b_1)}}^{\otimes 2}\otimes M^{(\overline{\epsilon a_1},\overline{\epsilon b_1})}\otimes M^{(\overline{(\delta-\epsilon-2)a_1},\overline{(\delta-\epsilon-2)b_1})}),
\]
where $n = 5m+\delta$ and $i=5t+\epsilon$ for some integers $m,t\geq0$, $\delta,\epsilon\in\{0,1,\ldots,4\}$. Calculating rank of the bundle $\mathbb{V}_{0,4}(K(\slalg_3,k),{M^{(a_1,b_1)}}^{\otimes 2}\otimes M^{(\overline{\epsilon a_1},\overline{\epsilon b_1})}\otimes M^{(\overline{(\delta-\epsilon-2)a_1},\overline{(\delta-\epsilon-2)b_1})})$ one sees that $\mathbb{D}\cdot F_{1,1,i}=0$ if $\delta\ne0$.

Let $\delta=0$. Then, $\mathbb{D}\cdot F_{1,1,5t+\epsilon} 
= d({M^{(a_1,b_1)}}^{\otimes 2}\otimes M^{(\overline{\epsilon a_1},\overline{\epsilon b_1})}\otimes M^{(\overline{(-\epsilon-2)a_1},\overline{(-\epsilon-2)b_1})})$ is zero if $\epsilon=0$. In order to calculate the intersections for $\epsilon\geq1$, we need some information about the conformal weight of the modules: The set of simple modules admit a partition 
\[
\mathcal{S}_2(5)^{sim}=c^0\sqcup c^{4/5}\sqcup c^{6/5}\sqcup c^{7/5}\sqcup c^{8/5},
\]
an analysis of the fusion rules among these modules give the following data:
\begin{align*}
\mathbb{D}_{0,n}(K(\slalg_3,5),M^{\otimes 5m})\cdot F_{1,1,5t+1}=\mathbb{D}_{0,n}(K(\slalg_3,5),M^{\otimes 5m})\cdot F_{1,1,5t+2} =\begin{cases}
0 & \text{ if $M\in c^0\sqcup c^{4/5}$}\\
1 & \text{ if $M\in c^{7/5}$}\\
2 & \text{ if $M\in c^{6/5}\sqcup c^{8/5}$},
\end{cases}
\end{align*}
\begin{align*}
\mathbb{D}_{0,n}(K(\slalg_3,5),M^{\otimes 5m})\cdot F_{1,1,5t+3}=0\quad\text{for all $M\in\mathcal{S}_3(5)^{sim}$, and }
\end{align*}
\begin{align*}
\mathbb{D}_{0,n}(K(\slalg_3,5),M^{\otimes 5m})\cdot F_{1,1,5t+4}
=\begin{cases}
0 & \text{ if $M\in c^0$}\\
2 & \text{ if $M\in c^{4/5}$}\\
4 & \text{ if $M\in c^{6/5}\sqcup c^{7/5}$}\\
5 & \text{ if $M\in c^{8/5}$}.
\end{cases}
\end{align*}
This completes the proof.
\end{proof}
\begin{question}
\label{question: semiample}
We know  the symmetric coinvariant divisors associated to representations in $\mathcal{S}_2(k)$ for levels $k \leq 5$ are nef.  It is natural to ask whether these divisors are base-point free, and if so, what morphisms they define? Since the intersection numbers of these divisors with $F$-curves satisfy a modularity condition that is truly distinct from those arising in  \cite{Veronese, Noah, NoahAngela} — one should expect something entirely new.
\end{question}

\subsubsection{Coinvariant divisors associated to \texorpdfstring{$\mathfrak{sl}_{r+1}$}{sl\_r+1} Parafermions for \texorpdfstring{$r \geq 3$}{r >= 3}}
\label{subsec: Ar positivity}
We conclude this section by proving that the subrings $\mathcal{S}_r(k)$ are not positive for any integers $r \geq 3$ and $k \geq 4$. The proof is based on exhibiting a subring
\(
\mathcal{S}'_r(k) := \left\langle M^{0,0-(a_1,\ldots,a_r)} \,\middle|\, a_1 = \cdots = a_r = a,\ 0 \leq a \leq r-1 \right\rangle \subset \mathcal{S}_r(k)
\)
that fails to be $F$-positive. For levels $k = 2, 3$, we show that the corresponding symmetric coinvariant divisors are nef and compute their intersection numbers with $F$-curves generating a basis for $N^1(\tilde{\mathrm{M}}_{0,n})$.

\begin{proposition}
\label{prop: non F-positive}
Let $k\geq4$ be any integer. The coinvariant divisors associated to  representations in $\mathcal{S}_r(k)$ are nef if and only if $r \leq 2$.
\end{proposition}
\begin{notation}
Let $M_a$ denote the simple module $M^{0,0-(a,\ldots,a)}$ in $\mathcal{S}'_r(k)$. The symbol $d(a_1,\ldots,a_4)$ denotes the degree of the divisor $\mathbb{D}_{0,4}(K(\slalg_{r+1},k),\otimes_{i=1}^4M_{a_i})$ on $\overline{\mathrm{M}}_{0,4}\cong\pp^1$. The rank of the bundle $\mathbb{V}_{0,n}(K(\slalg_{r+1},k),\otimes_{i=1}^nM_{a_i})$ on $\overline{\mathrm{M}}_{0,n}$ is denoted by $\mu(\otimes_{i=1}^n M_{a_i})$. The conformal weight of the module $M_a$ is written as $cw(a)$ and a simple computation gives $cw(a)=a+\frac{a^2r}{2k}(r-3)$. For any integers $x$ and $a$, we let $\overline{xa}$ denote $xa$ modulo $k$. Finally, we denote the symmetric divisor $\mathbb{D}_{0,n}(K(\slalg_{r+1},k),M_a^{\otimes n})$ on $\tilde{\mathrm{M}}_{0,n}=\overline{\mathrm{M}}_{0,n}/Sym_n$ by $\mathbb{D}_{0,n}^{r,a}$.
\end{notation}
\begin{proof}[Proof of Proposition~\ref{prop: non F-positive}]
One direction is a corollary to Theorems~\ref{thm: positive for sl2} and~\ref{thm: positive for sl3}. For the converse, it suffices to provide a coinvariant divisor $\mathbb{D}$ associated to representations in $\mathcal{S}_r(k)^{sim}$ that intersects an $F$-curve strictly negatively. 
We only need to consider symmetric divisors $\mathbb{D}_{0,n}^{r,a}$ on $\tilde{\mathrm{M}}_{0,n}$. One can show that all such divisors are trivial if  $k$ does not divide $n$. Let $n=km$ for some integer $m\geq1$.
\[
\mathbb{D}^{r,a}_{0,km}\cdot F_{1,1,kt+\epsilon} = d(a,a, \overline{\epsilon a},\overline{(-2-\epsilon)a})
=2cw(a)+cw(\overline{\epsilon a})+cw(\overline{(-2-\epsilon)a})-cw(\overline{2a})-2 cw(\overline{(\epsilon+1)a}).
\]
If $a=1$, then $\overline{\epsilon a}=\epsilon$, $\overline{(-2-\epsilon) a}=k-2-\epsilon$, for any integer $1\leq \epsilon\leq k-2$. Furthermore, If $\epsilon\leq k-3$, then the bundle $\mathbb{V}_{0,4}(K(\slalg_{r+1},k), M_1^{\otimes2}\otimes M_{\epsilon}\otimes M_{k-2-\epsilon})$ is non-trivial. Therefore, for any $1\leq \epsilon\leq k-3$
\begin{align*}
\mathbb{D}^{r,1}_{0,km}\cdot F_{1,1,kt+\epsilon} 
&=2cw(1)+cw(\epsilon)+cw(k-2-\epsilon)-cw(2)-2 cw(\epsilon+1)\\
&=\frac{1}{2}(k-2(2+\epsilon))(r-1)(r-2).
\end{align*}
Note that, for $t=0$ and $\epsilon=k-3$, 
$\mathbb{D}^{r,a}_{0,km}\cdot F_{1,1,k-3}<0$, and this completes the proof.
\end{proof}
The condition $k \geq 4$ ensures the existence of a positive integer $\epsilon$ satisfying $1 \leq \epsilon \leq k - 3$. The necessity for $r\geq3$ comes from the definition of the conformal weight. In the propositions below, we discuss positivity of the symmetric divisors associated to representations in $\mathcal{S}_r(k)$ for levels $k=2$ and $k=3$ and integers $r\geq3$.
\begin{proposition}
\label{prop: Sr k2}
For any integer $r\geq3$, the symmetric divisor $\mathbb{D}:=\mathbb{D}_{0,n}(K(\slalg_{r+1},2), M^{\otimes n})$ associated to a non-trivial simple module $M:=M^{0,0-(a_1,\ldots,a_r)}\in\mathcal{S}_r(2)$ is trivial if $n$ is odd. For $n$ even, letting $q:=a_1+\cdots+a_r$, 
\[
\mathbb{D}\cdot F_{1,1,i} = \begin{cases}
(q-1)(q-2)+2&\text{ if $i$ is odd},\\
0&\text{ if $i$ is even}.
\end{cases}
\]
\end{proposition}
\begin{proof}
All simple modules in $\mathcal{S}_r(2)$ are self-dual. The conformal weight of $M := M^{0,0-(a_1,\ldots,a_r)}$ is $\frac{(q - 1)(q - 2) + 2}{4}$. Writing $n = 2g + 2 + \delta$ and $i = 2t + \epsilon$ for integers $g, t \geq 0$ and $\delta, \epsilon \in \{0, 1\}$, we compute
\[
\mathbb{D} \cdot F_{1,1,i} = d\left( M^{\otimes 2} \otimes M^{(\overline{\epsilon a_1},\ldots,\overline{\epsilon a_r})} \otimes M^{(\overline{(\delta + \epsilon) a_1}, \ldots, \overline{(\delta + \epsilon) a_r})} \right).
\]
The intersection number is non-zero if and only if $\epsilon = 1$ and $\delta = 0$, in which case it equals $4 \cdot cw(M)$. 
\end{proof}
\begin{proposition}
\label{prop: Sr k3}
For any integer $r\geq3$, the symmetric coinvariant divisor $\mathbb{D}:=\mathbb{D}_{0,n}(K(\slalg_{r+1},3), M^{\otimes n})$ associated to a non-trivial simple module $M:=M^{0,0-(a_1,\ldots,a_r)}\in\mathcal{S}_r(3)$ is trivial if $3\not\mid n$. If $3\mid n$, for any $0\leq t\leq (2/3)n$  
\[
\mathbb{D}\cdot F_{1,1,3t+\epsilon} = \begin{cases}
3 cw(M)&\text{ if $\epsilon=2$},\\
0&\text{ if $\epsilon=0,1$ }.
\end{cases}
\]\end{proposition}
\begin{proof}
The proof is similar and therefore omitted.
\end{proof}
\begin{remark}
It was communicated to the author by Daebeom Choi that the divisors discussed in Propositions~\ref{prop: Sr k2} and~\ref{prop: Sr k3} are semi-ample in characteristic~$p$ (see~\cite{DaebeomInPrep}). It would be interesting to investigate the corresponding morphisms.
\end{remark}
\subsection{Proof of Proposition~\ref{prop: degree for sl2}}
\label{subsec: parafermions and proof of degree formula}
In this subsection, we prove Proposition~\ref{prop: degree for sl2}. We restate the proposition below for the reader's convenience.
\begin{proposition}
Let $k\geq1$ be an integer and let $i\in\{a,b,c,d\}$ so that $0\leq i'<i\leq k$ for each $i$ and $a\leq b\leq c\leq d$. Then, the degree $d$ of the divisor $\mathbb{D}_{0,4}(K(\slalg_2, k); M^{a,a'}\otimes (M^{b,b'})'  \otimes (M^{c,c'})'\otimes (M^{d,d'})')$ is given in terms of its rank $\mu$ and the sum of the conformal weights $c_{\Sigma}:=\sum_{i} cw(M^{i,i'})$ as follows:
\begin{equation*}
 -d+\mu c_{\Sigma} = \begin{cases}
0&\text{ if }\frac{b+c+d-a}{2}\ne b'+c'+d'-a'\mod k,\\
\Lambda &\text{ else,}
 \end{cases}
\end{equation*}
where 
\[
\Lambda = \sum_{i\in\{b,c,d\}}\left(\sum_{\substack{m_i\leq t\leq M_i\\ (a+i)\in2\zz}}cw\left(M^{t,\ \overline{a'-i'+\frac{t-a+i}{2}}}\right)\right)
\]
and $$M_i=\min(a+i,\ b+c+d-i,\ 2k-a-i,\ 2k-b-c-d+i),$$ $$m_i=\max(|i-a|,\ |\ga-\gb|),$$ with $\{\ga,\gb\}=\{b,c,d\}\setminus\{i\}$.
\end{proposition}

\begin{lemma}
Consider the three modules $M^{c,c'}, M^{d,d'}, m_{t,t'}$ for integers $0\leq a<a'\leq k$, where $a\in\{c,d,t\}$. Then, 
\[
\text{rank } \mathbb{D}_{0,4}(K(\slalg_2; k); M^{c,c'}\otimes M^{d,d'}\otimes (M^{t,t'})') =
\text{rank } \mathbb{D}_{0,4}(K(\slalg; k); (M^{c,c'})'\otimes (M^{d,d'})'\otimes M^{t,t'}).
\]
\end{lemma}
\begin{proof}
Note that $(M^{a,a'})'\cong m_{a,a-a'}$ for any $a\in\{c,d,t\}$ and therefore
\[
\text{rank } \mathbb{D}_{0,4}(K(\slalg_2; k); (M^{c,c'})'\otimes (M^{d,d'})'\otimes (M^{t,t'})') = 1
\]
if and only if $|d-c|\leq t\leq\min(c+d,2k-c-d)$ and $t=c+d\mod2$ with $$t-t'=\frac{1}{2}(2c-2c'+2d-2d'-c-d+t)\mod k.$$ The last condition reduces to $t' = \frac{1}{2}(2c'+2d'-c-d+t)\mod k.$ These conditions are necessary and sufficient for 
\[
\text{rank } \mathbb{V}_{0,4}(K(\slalg_2; k); M^{c,c'}\otimes M^{d,d'}\otimes (M_{t,t'})') = 1.
\]
Since the rank of these bundles is either zero or one, the result follows.
\end{proof}

Degree of the bundle 
\(
\mathbb{D}:=\mathbb{D}_{0,4}(K(\slalg_2; k); M^{a,a'}\otimes (M^{b,b'})'  \otimes (M^{c,c'})'\otimes (M^{d,d'})')
\)
is given by 
\begin{align*}
d+\mu c_{\Sigma}= \sum_{t}cw(M^{t,t'})\mu(M^{a,a'}\otimes (M^{b,b'})'\otimes (M^{t,t'})')\mu((M^{c,c'})'\otimes (M^{d,d'})'\otimes M^{t,t'}) +(b\leftrightarrow c) + (b\leftrightarrow d),
\end{align*}
where $\mu$ and $d$ are rank and degree, respectively, of the line bundle $\mathbb{D}$ and $c_{\Sigma}$ is sum of the conformal weights of the modules $M^{h,h'}$ where $h$ runs over the set $\{a,b,c,d\}$. Here, we are using the fact that conformal weight is invariant under taking the contragradient dual. Finally, $(b\leftrightarrow c)$ is the same as the sum over $t$ with positions of $b$ and $c$ switched, and similarly for $(b\leftrightarrow d)$. By the Lemma above, we have 
\begin{align*}
&\sum_{t}c(M^{t,t'})\mu(M^{a,a'}\otimes (M^{b,b'})'\otimes (M^{t,t'})')
\mu((M^{c,c'})'\otimes (M^{d,d'})'\otimes M^{t,t'})\\
&\qquad =\sum_{t}c(M^{t,t'})\mu(M^{a,a'}\otimes (M^{b,b'})'\otimes (M^{t,t'})')\mu(M^{c,c'}\otimes M^{d,d'}\otimes (M^{t,t'})').
\end{align*}
For a summand corresponding to $t$ to be nonzero, the first rank element requires that $t$ satisfy the following 
\[
|b-a|\leq t\leq\min(a+b,2k-a-b), t= a+b\mod2, 
t' = a'-b'+\frac{t-a+b}{2}\mod k
\]
and the second rank element requires that 
\[
|d-c|\leq t\leq\min(c+d,2k-c-d), t= c+d\mod2, 
t' = c'+d'+\frac{t-c+d}{2}\mod k.
\]
Combining these conditions, a nonzero contribution from a summand occurs only when $t$ satisfies the following:
\begin{align*}
\max(|b-a|,|d-c|)\leq &t\leq\min(a+b,c+d,2k-a-b,2k-c-d)\\
&t=a+b=c+d\mod2
\end{align*}
along with 
\[
t' = a'-b'+\frac{t-a+b}{2}\mod k=c'+d'+\frac{t-c+d}{2}\mod k
\]
and finally, we need the modules in consideration to satisfy the relation
\[
\frac{b+c+d-a}{2}=b'+c'+d'-a'\mod k.
\]
This completes the proof of the proposition.

\begin{example}
Let $\mathbb{D}:=\mathbb{D}_{0,4}(K(\slalg_2,k),\bigotimes_{p=1}^4 M^{2t_p,t_p})$, where $\sum t_i\geq k$, and $t_2+t_3\leq t_1+t_4$. Since all four modules are self-dual, we can directly apply the proposition above. After some calculations, we have
\[
m_{2t_2}=2t_4-2t_3,\quad  m_{2t_3}=2t_4-2t_2,\quad m_{2t_4}=2t_4-2t_1,
\]
\[
M_{2t_2}=2k-2t_3-2t_4,\quad  M_{2t_3}=2k-2t_2-2t_4,\quad M_{2t_4}=2k-2t_1-2t_4.
\]
Moreover, from Proposition~\ref{prop: rank formula sl2}, we have $\mu = 1+M_{2t_2}-m_{2t_2}=(1+k-2t_4)$, 
\[
\Lambda = \frac{\mu}{k+2}\left(2k+k^2+\sum_{p=1}^4(t_p^2-(k+1)t_p)\right),\quad  d=\mu (-k+\sum_{p=1}^4 t_p).
\]
\end{example}
\begin{example}
\label{exmp: fusion ring not F-poisitive}
The fusion rules for the parafermion VOA \( K(\mathfrak{sl}_2,3) \) imply that the subring of the fusion ring \( \mathcal{R}(K(\mathfrak{sl}_2,3)) \) generated by the union of the subrings \( \mathcal{T}(3) = \langle M^{0,0}, M^{2,1} \rangle \) and \( \mathcal{S}_1(3) = \langle M^{0,0}, M^{3,1}, M^{3,2} \rangle \) is, in fact, the entire fusion ring. Indeed, the relations
\[
M^{2,1} \boxtimes M^{3,1} = M^{1,0}, \quad \text{and} \quad M^{1,0} \boxtimes M^{3,1} = M^{2,0}
\]
show that all remaining simple modules can be generated from elements in \( \mathcal{T}(3) \cup \mathcal{S}_1(3) \). Using the degree formula discussed above, one verifies that the following coinvariant divisors on \( \overline{\mathrm{M}}_{0,4} \) have degree equal to \(-1\):
\[
\mathbb{D}_{0,4}(K(\mathfrak{sl}_2,3), {M^{1,0}}^{\otimes 2} \otimes {M^{2,0}}^{\otimes 2}), \quad 
\mathbb{D}_{0,4}(K(\mathfrak{sl}_2,3), {M^{1,0}}^{\otimes 3} \otimes M^{2,1}), \quad 
\mathbb{D}_{0,4}(K(\mathfrak{sl}_2,3), {M^{2,0}}^{\otimes 3} \otimes M^{2,1}).
\]
Therefore, by definition, the fusion ring \( \mathcal{R}(K(\mathfrak{sl}_2,3)) \) is not \( F \)-positive and hence not positive.
\end{example}

\bibliographystyle{alpha} 
\bibliography{ref_parafermions} 

\newcommand{\etalchar}[1]{$^{#1}$}
\begin{thebibliography}{DLWY10}

\bibitem[AA13]{AA13}
Toshiyuki Abe and Yusuke Arike.
\newblock Intertwining operators and fusion rules for vertex operator algebras arising from symplectic fermions.
\newblock {\em Journal of Algebra}, 373:39--64, 2013.

\bibitem[ADJR18]{ADJR18}
Chunrui Ai, Chongying Dong, Xiangyu Jiao, and Li~Ren.
\newblock The irreducible modules and fusion rules for the parafermion vertex operator algebras.
\newblock {\em Trans. Amer. Math. Soc.}, 370(8):5963--5981, 2018.

\bibitem[AGS14]{AGS14}
Valery Alexeev, Angela Gibney, and David Swinarski.
\newblock Higher-level {$\slalg_2$} conformal blocks divisors on {$\overline M_{0,n}$}.
\newblock {\em Proc. Edinb. Math. Soc. (2)}, 57(1):7--30, 2014.

\bibitem[Ale]{alexeev_ranks_sl2}
Boris Alexeev.
\newblock Ranks and degrees of \(\mathfrak{sl}_2\) conformal blocks.
\newblock In preparation.

\bibitem[ALY14]{ALY14}
Tomoyuki Arakawa, Ching~Hung Lam, and Hiromichi Yamada.
\newblock Zhu's algebra, c2-algebra and c2-cofiniteness of parafermion vertex operator algebras.
\newblock {\em Advances in Mathematics}, 264:261--295, 2014.

\bibitem[ALY19]{ALY19}
Tomoyuki Arakawa, Ching~Hung Lam, and Hiromichi Yamada.
\newblock {Parafermion vertex operator algebras and \ensuremath{\mathit{W}}-algebras}.
\newblock {\em Trans. Am. Math. Soc.}, 371(6):4277--4301, 2019.

\bibitem[BB93]{BB93}
A.~Beilinson and J.~Bernstein.
\newblock A proof of {J}antzen conjectures.
\newblock In {\em I. M. Gelfand Seminar}, volume 16, Part 1 of {\em Adv. Soviet Math.}, pages 1--50. Amer. Math. Soc., Providence, RI, 1993.

\bibitem[BFM91]{BFM}
Alexander Beilinson, Boris Feigin, and Barry Mazur.
\newblock Introduction to algebraic field theory on curves.
\newblock Unpublished manuscript, 1991.

\bibitem[BGM15]{CriticallevelBGM}
Prakash Belkale, Angela Gibney, and Swarnava Mukhopadhyay.
\newblock Vanishing and identities of conformal blocks divisors.
\newblock {\em Algebr. Geom.}, 2(1):62--90, 2015.

\bibitem[BL94]{BL94}
Arnaud Beauville and Yves Laszlo.
\newblock Conformal blocks and generalized theta functions.
\newblock {\em Comm. Math. Phys.}, 164(2):385--419, 1994.

\bibitem[BLS98]{BLS98}
Arnaud Beauville, Yves Laszlo, and Christoph Sorger.
\newblock The {P}icard group of the moduli of {$G$}-bundles on a curve.
\newblock {\em Compositio Math.}, 112(2):183--216, 1998.

\bibitem[Cha25]{Parafermionsgithub}
Avik Chakravarty.
\newblock Parafermions.
\newblock GitHub repository, Available at \url{https://github.com/avikchakravarty10/Parafermions}, 2025.

\bibitem[Cho25a]{ChoiVirasoro}
Daebeom Choi.
\newblock Conformal block divisors for discrete series virasoro voa $\text{Vir}_{2k+1,2}$, 2025.

\bibitem[Cho25b]{DaebeomInPrep}
Daebeom Choi.
\newblock Extremal effective curves and non-semiample line bundles on $\overline{\mathcal{m}}_{g,n}$.
\newblock in preparation, 2025.

\bibitem[Cod20]{Cod19}
Giulio Codogni.
\newblock Vertex algebras and teichm\"{u}ller modular forms, 2020.

\bibitem[DG23]{DG23}
Chiara Damiolini and Angela Gibney.
\newblock On global generation of vector bundles on the moduli space of curves from representations of vertex operator algebras.
\newblock {\em Algebr. Geom.}, 10(3):298--326, 2023.

\bibitem[DGK22]{DGK}
Chiara Damiolini, Angela Gibney, and Daniel Krashen.
\newblock Factorization presentations, 2022.

\bibitem[DGK24]{DGK2}
Chiara Damiolini, Angela Gibney, and Daniel Krashen.
\newblock Conformal blocks on smoothings via mode transition algebras, 2024.

\bibitem[DGT21]{DGT1}
Chiara Damiolini, Angela Gibney, and Nicola Tarasca.
\newblock Conformal blocks from vertex algebras and their connections on {$\overline{M}_{g, n}$}.
\newblock {\em Geom. Topol.}, 25(5):2235--2286, 2021.

\bibitem[DGT22]{DGT3}
Chiara Damiolini, Angela Gibney, and Nicola Tarasca.
\newblock Vertex algebras of {C}oh{FT}-type.
\newblock In {\em Facets of algebraic geometry. {V}ol. {I}}, volume 472 of {\em London Math. Soc. Lecture Note Ser.}, pages 164--189. Cambridge Univ. Press, Cambridge, 2022.

\bibitem[DGT24]{DGT2}
Chiara Damiolini, Angela Gibney, and Nicola Tarasca.
\newblock On factorization and vector bundles of conformal blocks from vertex algebras.
\newblock {\em Ann. Sci. \'Ec. Norm. Sup\'er. (4)}, 57(1):241--292, 2024.

\bibitem[DL89]{DL1}
C.~Y. Dong and J.~Lepowsky.
\newblock A {J}acobi identity for relative vertex operators and the equivalence of {$Z$}-algebras and parafermion algebras.
\newblock In {\em X{VII}th {I}nternational {C}olloquium on {G}roup {T}heoretical {M}ethods in {P}hysics ({S}ainte-{A}d\`ele, {PQ}, 1988)}, pages 235--238. World Sci. Publ., Teaneck, NJ, 1989.

\bibitem[DL93]{DL93}
Chongying Dong and James Lepowsky.
\newblock {\em Generalized vertex algebras and relative vertex operators}, volume 112 of {\em Progress in Mathematics}.
\newblock Birkh\"{a}user Boston, Inc., Boston, MA, 1993.

\bibitem[DLWY10]{DLWY10}
Chongying Dong, Ching~Hung Lam, Qing Wang, and Hiromichi Yamada.
\newblock The structure of parafermion vertex operator algebras.
\newblock {\em Journal of algebra.}, 323(2):371--381, 2010.

\bibitem[DLY10]{DLY09}
Chongying Dong, Ching~Hung Lam, and Hiromichi Yamada.
\newblock W-algebras related to parafermion algebras.
\newblock {\em Journal of algebra.}, 322(7):2366--2403, 2009-10.

\bibitem[DR17a]{DR17}
Chongying Dong and Li~Ren.
\newblock Representations of the parafermion vertex operator algebras.
\newblock {\em Adv. Math.}, 315:88--101, 2017.

\bibitem[DR17b]{DL17}
Chongying Dong and Li~Ren.
\newblock Representations of the parafermion vertex operator algebras.
\newblock {\em Advances in mathematics.}, 315:88--101, 2017.

\bibitem[DW10]{DW10}
Chongying Dong and Qing Wang.
\newblock The structure of parafermion vertex operator algebras: General case.
\newblock {\em Communications in mathematical physics.}, 299(3):783--792, 2010.

\bibitem[DW11]{DW11}
Chongying Dong and Qing Wang.
\newblock Parafermion vertex operator algebras.
\newblock {\em Frontiers of mathematics in China : selected publications from Chinese universities.}, 6(4):567--579, 2011.

\bibitem[DW15]{DW16}
Chongying Dong and Qing Wang.
\newblock Quantum dimensions and fusion rules for parafermion vertex operator algebras.
\newblock {\em Proceedings of the American Mathematical Society.}, 144(4):1483--1492, 2015.

\bibitem[Fak12]{Fak12}
Najmuddin Fakhruddin.
\newblock Chern classes of conformal blocks.
\newblock In {\em Compact moduli spaces and vector bundles}, volume 564 of {\em Contemp. Math.}, pages 145--176. Amer. Math. Soc., Providence, RI, 2012.

\bibitem[Fal94]{Faltings94}
Gerd Faltings.
\newblock A proof for the {V}erlinde formula.
\newblock {\em J. Algebraic Geom.}, 3(2):347--374, 1994.

\bibitem[FBZ04]{FBZ04}
Edward Frenkel and David Ben-Zvi.
\newblock {\em Vertex algebras and algebraic curves}, volume~88 of {\em Mathematical Surveys and Monographs}.
\newblock American Mathematical Society, Providence, RI, second edition, 2004.

\bibitem[Fed20]{Fed20}
Maksym Fedorchuk.
\newblock Symmetric f-conjecture for $g\leq 35$, 2020.

\bibitem[FHL93]{FHL93}
Igor~B. Frenkel, Yi-Zhi Huang, and James Lepowsky.
\newblock On axiomatic approaches to vertex operator algebras and modules.
\newblock {\em Mem. Amer. Math. Soc.}, 104(494):viii+64, 1993.

\bibitem[FZ92]{FZ92}
Igor~B. Frenkel and Yongchang Zhu.
\newblock Vertex operator algebras associated to representations of affine and virasoro algebras.
\newblock {\em Duke Mathematical Journal}, 66:123--168, 1992.

\bibitem[Gep87]{GEP87}
Doron Gepner.
\newblock New conformal field theories associated with lie algebras and their partition functions.
\newblock {\em Nuclear Physics B}, 290:10--24, 1987.

\bibitem[GG12]{NoahAngela}
Noah Giansiracusa and Angela Gibney.
\newblock The cone of type {$A$}, level 1, conformal blocks divisors.
\newblock {\em Adv. Math.}, 231(2):798--814, 2012.

\bibitem[Gia13]{Noah}
Noah Giansiracusa.
\newblock Conformal blocks and rational normal curves.
\newblock {\em J. Algebraic Geom.}, 22(4):773--793, 2013.

\bibitem[Gib09]{Gib09}
Angela Gibney.
\newblock Numerical criteria for divisors on {$\overline M_g$} to be ample.
\newblock {\em Compos. Math.}, 145(5):1227--1248, 2009.

\bibitem[GJMS13]{Veronese}
Angela Gibney, David Jensen, Han-Bom Moon, and David Swinarski.
\newblock Veronese quotient models of {$\overline{\rm M}_{0,n}$} and conformal blocks.
\newblock {\em Michigan Math. J.}, 62(4):721--751, 2013.

\bibitem[GKM02]{GKM}
Angela Gibney, Sean Keel, and Ian Morrison.
\newblock Towards the ample cone of {$\overline M_{g,n}$}.
\newblock {\em J. Amer. Math. Soc.}, 15(2):273--294, 2002.

\bibitem[GL01]{BOOK-VA}
Yongcun Gao and Haisheng Li.
\newblock Generalized vertex algebras generated by parafermion-like vertex operators.
\newblock {\em J. Algebra}, 240(2):771--807, 2001.

\bibitem[Hua05]{Huang05}
Yi-Zhi Huang.
\newblock Vertex operator algebras, the {V}erlinde conjecture, and modular tensor categories.
\newblock {\em Proc. Natl. Acad. Sci. USA}, 102(15):5352--5356, 2005.

\bibitem[Hum72]{Hum72}
James~E. Humphreys.
\newblock {\em Introduction to {L}ie algebras and representation theory}, volume Vol. 9 of {\em Graduate Texts in Mathematics}.
\newblock Springer-Verlag, New York-Berlin, 1972.

\bibitem[KM13]{KM13}
Se\'an Keel and James McKernan.
\newblock Contractible extremal rays on {$\overline M_{0,n}$}.
\newblock In {\em Handbook of moduli. {V}ol. {II}}, volume~25 of {\em Adv. Lect. Math. (ALM)}, pages 115--130. Int. Press, Somerville, MA, 2013.

\bibitem[KNR94]{KNR94}
Shrawan Kumar, M.~S. Narasimhan, and A.~Ramanathan.
\newblock Infinite {G}rassmannians and moduli spaces of {$G$}-bundles.
\newblock {\em Math. Ann.}, 300(1):41--75, 1994.

\bibitem[Lar11]{Lar11}
Paul~L. Larsen.
\newblock Fulton's conjecture for $\overline{\rm{m}}_{0.7}$.
\newblock {\em Journal of the London Mathematical Society}, 85(1):1--21, 11 2011.

\bibitem[Li97]{Li97}
Haisheng Li.
\newblock The physics superselection principle in vertex operator algebra theory.
\newblock {\em Journal of Algebra}, 196(2):436--457, 1997.

\bibitem[Li01]{Li01}
Haisheng Li.
\newblock Certain extensions of vertex operator algebras of affine type.
\newblock {\em Comm. Math. Phys.}, 217(3):653--696, 2001.

\bibitem[LL04]{LL04}
James Lepowsky and Haisheng Li.
\newblock {\em Introduction to vertex operator algebras and their representations}, volume 227 of {\em Progress in Mathematics}.
\newblock Birkh\"auser Boston, Inc., Boston, MA, 2004.

\bibitem[LP84]{LP1}
James Lepowsky and Mirko Primc.
\newblock Standard modules for type one affine lie algebras.
\newblock In David~V. Chudnovsky, Gregory~V. Chudnovsky, Harvey Cohn, and Melvin~B. Nathanson, editors, {\em Number Theory}, pages 194--251, Berlin, Heidelberg, 1984. Springer Berlin Heidelberg.

\bibitem[LS97]{LS97}
Yves Laszlo and Christoph Sorger.
\newblock The line bundles on the moduli of parabolic {$G$}-bundles over curves and their sections.
\newblock {\em Ann. Sci. \'Ecole Norm. Sup. (4)}, 30(4):499--525, 1997.

\bibitem[LW81]{LW2}
James Lepowsky and Robert~Lee Wilson.
\newblock A new family of algebras underlying the rogers-ramanujan identities and generalizations.
\newblock {\em Proceedings of the National Academy of Sciences}, 78(12):7254--7258, 1981.

\bibitem[LW84]{LW3}
James Lepowsky and Robert~Lee Wilson.
\newblock The structure of standard modules, i: Universal algebras and the rogers-ramanujan identities.
\newblock {\em Inventiones mathematicae}, 77:199--290, 1984.

\bibitem[MOP15]{MOP}
Alina Marian, Dragos Oprea, and Rahul Pandharipande.
\newblock The first {C}hern class of the {V}erlinde bundles.
\newblock In {\em String-{M}ath 2012}, volume~90 of {\em Proc. Sympos. Pure Math.}, pages 87--111. Amer. Math. Soc., Providence, RI, 2015.

\bibitem[MOP{\etalchar{+}}17]{MOP+2}
Alina Marian, Dragos Oprea, Rahul Pandharipande, Aaron Pixton, and Dimitri Zvonkine.
\newblock The {C}hern character of the {V}erlinde bundle over {$\overline{\mathcal M}_{g,n}$}.
\newblock {\em J. Reine Angew. Math.}, 732:147--163, 2017.

\bibitem[NT05]{NT05}
Kiyokazu Nagatomo and Akihiro Tsuchiya.
\newblock Conformal field theories associated to regular chiral vertex operator algebras. i. theories over the projective line.
\newblock {\em Duke Mathematical Journal}, 128(3):393--471, 2005.

\bibitem[Pau96]{Pauly96}
Christian Pauly.
\newblock Espaces de modules de fibr\'es paraboliques et blocs conformes.
\newblock {\em Duke Math. J.}, 84(1):217--235, 1996.

\bibitem[RW02]{sl2rankPhysics}
J\o~rgen Rasmussen and Mark~A. Walton.
\newblock Fusion multiplicities as polytope volumes: {$\mathcal{N}$}-point and higher-genus {${\mathfrak{su}}(2)$} fusion.
\newblock {\em Nuclear Phys. B}, 620(3):537--550, 2002.

\bibitem[Swi11]{Swinarski_sl2}
David Swinarski.
\newblock $sl_2$ conformal block divisors and the nef cone of $\bar{M}_{0,n}$, 2011.

\bibitem[Tha94]{Thaddeus94}
Michael Thaddeus.
\newblock Stable pairs, linear systems and the {V}erlinde formula.
\newblock {\em Invent. Math.}, 117(2):317--353, 1994.

\bibitem[TK87]{TK87}
Akihiro Tsuchiya and Yukihiro Kanie.
\newblock Vertex operators in the conformal field theory on {${\bf P}^1$} and monodromy representations of the braid group.
\newblock {\em Lett. Math. Phys.}, 13(4):303--312, 1987.

\bibitem[TK88]{TK86}
Akihiro Tsuchiya and Yukihiro Kanie.
\newblock Vertex operators in conformal field theory on {${\bf P}^1$} and monodromy representations of braid group.
\newblock In {\em Conformal field theory and solvable lattice models ({K}yoto, 1986)}, volume~16 of {\em Adv. Stud. Pure Math.}, pages 297--372. Academic Press, Boston, MA, 1988.

\bibitem[Tsu93]{Tsu93}
Yoshifumi Tsuchimoto.
\newblock On the coordinate-free description of the conformal blocks.
\newblock {\em J. Math. Kyoto Univ.}, 33(1):29--49, 1993.

\bibitem[TUY89]{TUY89}
Akihiro Tsuchiya, Kenji Ueno, and Yasuhiko Yamada.
\newblock Conformal field theory on universal family of stable curves with gauge symmetries.
\newblock In {\em Integrable systems in quantum field theory and statistical mechanics}, volume~19 of {\em Adv. Stud. Pure Math.}, pages 459--566. Academic Press, Boston, MA, 1989.

\bibitem[ZF85]{ZF1}
A~B Zamolodchikov and V~A Fateev.
\newblock Nonlocal (parafermion) currents in two-dimensional conformal quantum field theory and self-dual critical points in z/sub n/-symmetric statistical systems.
\newblock {\em Sov. Phys. - JETP (Engl. Transl.); (United States)}, 62(0), 8 1985.

\end{thebibliography}
\end{document}